\documentclass{amsart}
\usepackage{amssymb,latexsym}
\usepackage{amsmath}
\usepackage{amscd}
\usepackage[dvipdfm]{graphicx}
 \usepackage{color}
\usepackage{enumerate}
\newenvironment{enumeratei}{\begin{enumerate}[\quad\upshape (i)]} {\end{enumerate}}
\numberwithin{equation}{section}
\theoremstyle{plain}

 \newtheorem{theorem}{Theorem}[section]

 \newtheorem{lemma}[theorem]{Lemma}
 \newtheorem{proposition}[theorem]{Proposition}
 
 \newtheorem{observation}[theorem]{Observation}
 \newtheorem{corollary}[theorem]{Corollary}
\theoremstyle{definition}

 \newtheorem{definition}[theorem]{Definition}

 \newtheorem{remark}[theorem]{Remark}
 
\theoremstyle{remark}
 
 \newtheorem*{Caseone}{Case 1}
 \newtheorem*{Casetwo}{Case 2}

\theoremstyle{plain} 
\theoremstyle{definition} 
\theoremstyle{remark} 

\newcommand \birthno {\textup{yb}}
\newcommand \birthtr {\textup{btr}}
\newcommand \lessdesc {<_{\textup{desc}}}

\newcommand \dsty [1] {\displaystyle{#1}}

\newcommand \sssty [1] {\scriptscriptstyle{#1}}
\newcommand \tbf[1] {\textbf{#1}}  

\renewcommand\phi{\varphi}
\renewcommand\rho{\varrho}   
\renewcommand\epsilon{\varepsilon} 
\renewcommand \theta {\vartheta}

\newcommand\init [1] {#1} 

\newcommand\nothing[1] {}

\newcommand \mapping {map} 
\newcommand \poset {ordered set{}}

\newcommand \balpha {{\boldsymbol{\alpha}}}

\newcommand \bdelta{\boldsymbol{\delta}}
\newcommand \bepsilon{\boldsymbol{\epsilon}}

\newcommand \pair [2] {\tuple{#1,#2}}
\renewcommand\phi{\varphi}

\newcommand \restrict [2] {{#1}\kern-1pt \rceil_{\kern-1pt #2}}
\newcommand\set [1]{\{#1\}}
\newcommand \tuple [1] {\langle #1 \rangle}
\newcommand \then {\Rightarrow}

\newcommand \Jir [1] {\textup{Ji}(#1)} 
\newcommand \Mir [1] {\textup{Mi}(#1)}

\newcommand \length {\textup{length}}
\newcommand\ideal[1]{\mathord\downarrow #1}
\newcommand\filter[1]{\mathord\uparrow #1}

\newcommand \chain [1] {\sf C_{#1}}
\newcommand \colcomp [2] {{#1}\mathrel{ \mathord{\triangleleft}\kern -3.6pt \mathord{\cdot}  }{#2} }

\newcommand \leftb [1]  {\textup{C}_{\textup l}(#1)} 
\newcommand \rightb [1] {\textup{C}_{\textup r}(#1)} 
\newcommand \bound [1] {\textup{Bnd}(#1)} 

\newcommand \cornl [1] { \textup{lc}(#1) } 
\newcommand \cornr [1] { \textup{rc}(#1) } 
\newcommand \lsquare [1] {^{\diamond}\kern-3pt #1}

\newcommand \Traj [1] {\textup{Traj}(#1)}
\newcommand \Straj [1] {\widehat{\textup{Traj}}(#1)}

\newcommand \trajeq {\mathrel{ \sim^{\kern-2pt\textup{traj}} } }

\newcommand \wtau {{\widehat\tau}}

\newcommand \preogen {\textup{quor}}

\newcommand\tblokk [1] {{\blokk{#1}\Theta}}
\newcommand \blokk [2] {#1/#2}
\newcommand \con {\textup{con}}

\DeclareMathOperator{\Con}{Con}

\newcommand \intv [1]  {{\mathfrak #1}}  
\newcommand \Prin[1] {\textup{PrInt}(#1)} 
\newcommand \inp {{\mathfrak p}}
\newcommand \inq {{\mathfrak q}} 
\newcommand \inh {{\mathfrak h}} 
\newcommand \inr {{\mathfrak r}} 
\newcommand \ing {{\mathfrak g}} 
 
\newcommand \cproj {\mathrel{ \mathord{\Rightarrow} \kern-7.5pt \mathord{\Rightarrow} }}
\newcommand \cpreq {\mathrel{ \mathord{\Leftarrow}  \kern-7.5pt \mathord{\Leftrightarrow} \kern-7.5pt \mathord{\Rightarrow} }} 
\newcommand \uppers {\,{\buildrel{\sssty{\textup{up}}}\over \rightarrow}\kern-8pt\mathord{\rightarrow}\; } 
\newcommand \upperpa [1] {\,{\buildrel{\sssty{\textup{up}}}\over \rightarrow}\kern-8pt\mathord{\rightarrow}_{#1}\; } 
\newcommand \dnpers {\,{\buildrel{\sssty{\textup{dn}}}\over \rightarrow}\kern-8pt\mathord{\rightarrow}\; } 
\newcommand \dnperpa [1] {\,{\buildrel{\sssty{\textup{dn}}}\over \rightarrow}\kern-8pt\mathord{\rightarrow}_{#1}\; }

\numberwithin{equation}{section}

\newcommand \url [1] {\tt{#1}}
\newcommand \colgammacon{\textup{(C1)}}
\newcommand \colcongamma{\textup{(C2)}}

\newcommand\rectext{normal rectangular extension}
\newcommand\Rectexts{Normal rectangular extensions}

\newcommand \lsp{\textup{lsp}}
\newcommand \rsp{\textup{rsp}}
\newcommand \lside{\textup{LS}}
\newcommand \rside{\textup{RS}}
\newcommand \llb [1]{C_{\textup{ll}}(#1)} 
\newcommand \lrb [1]{C_{\textup{lr}}(#1)} 
\newcommand \lefth {\textup{ljc}^{\sssty{\textup{n}}}} 
\newcommand \righth {\textup{rjc}^{\sssty{\textup{n}}}} 
\newcommand \lhr [1] {\lefth_R(#1)} 
\newcommand \rhr [1] {\righth_R(#1)} 
\newcommand \height {\textup{height}}
\newcommand \emel{m_\textup{l}}
\newcommand \emer{m_\textup{r}}
\newcommand \tel{{\textup{l}}}
\newcommand \ter{{\textup{r}}}
\newcommand \rellambda {\mathrel{\lambda}}
\newcommand\incoord[1]{\textup{ICP}_{#1}(L)}
\newcommand\leftcoord[1]{\textup{LCP}_{#1}(L)}
\newcommand\rightcoord[1]{\textup{RCP}_{#1}(L)}
\newcommand\allcoord[1]{\textup{ACP}_{#1}(L)}
\newcommand\vincoord[1]{\textup{ICP}_{#1'}(L')}

\newcommand\vallcoord[1]{\textup{ACP}_{#1'}(L')}
\newcommand \refl[1]{#1^{\textup{(mi)}}}
\newcommand\adia {\mathcal C_0}
\newcommand\bdia {\mathcal C_1}
\newcommand\cdia {\mathcal C_2}
\newcommand\ddia {\mathcal C_3}
\newcommand\jdia[1]{\mathcal C_{#1}}
\newcommand\glsum {\mathrel{+^{\sssty{\textup{gl}}}}}
\newcommand \swings {\mathrel{\mathbin{\raisebox{2.0pt}{\rotatebox{160}{$\curvearrowleft$}}}}}
\newcommand\perspup {\mathrel{ \overset{\sssty{\textup{up}}}{\sim} }}
\newcommand\perspdn {\mathrel{ \overset{\sssty{\textup{dn}}}{\sim} }}
\newcommand\qparallel {\mathrel{     {\mathord\parallel} _ {\kern-1pt \sssty{\textup{quasi}}}         }}
\newcommand\qnparallel {\mathrel{     {\mathord\nparallel} _ {\kern-1pt \sssty{\textup{quasi}}}         }}

\newcommand\lexleq{\mathrel{\leq_{\textup{lex}}}}
\newcommand\antilexleq{\mathrel{\sqsubseteq_{\textup{lex}}}}
\newcommand \terr[1] {\textup{Terr}(#1)}
\newcommand \oterr[1] {\textup{Terr}_{\sssty{\textup{orig}}}(#1)}
\newcommand\anc[2]{\textup{anc}(#1,#2)}
\newcommand\vrbal{\vec r\,^{\sssty{\textup{ft}}}}
\newcommand\vrjobb{\vec r\,^{\sssty{\textup{gh}}}}
\newcommand\rbal[1]{r^{\sssty{\textup{ft}}}_{#1}}
\newcommand\rjobb[1]{r^{\sssty{\textup{gh}}}_{#1}}
\newcommand\vrvbal{\vec r\,^{\ast\sssty{\textup{ft}}}}
\newcommand\vrvjobb{\vec r\,^{\ast\sssty{\textup{gh}}}}
\newcommand\rvbal[1]{r^{\ast\sssty{\textup{ft}}}_{#1}}
\newcommand\rvjobb[1]{r^{\ast \sssty{\textup{gh}}}_{#1}}
\newcommand\clft {c_{\sssty{\textup{ft}}}}
\newcommand\crght {c_{\sssty{\textup{gh}}}}
\newcommand\interr[1]{\textup{Interior}(#1)}
\newcommand\nggy{G} 
\newcommand\unu {f^{\sssty{\textup{nc}}}} 
\newcommand\inu[2] {\unu_{ #2\mathord{\subseteq} #1 }}
\newcommand\nszgpi {\pi^{\sssty{\bullet}}}
\newcommand\id{\textup{id}}
\newcommand\mixint [2]{[#1,#2]^\ast}
\newcommand\accessible {$\swings\perspdn$-accessible}

\newcommand \vajonkell [1] {}

\begin{document}
\title[Diagrams and rectangular extensions]
{Diagrams and rectangular extensions of planar semimodular lattices}

\author[G.\ Cz\'edli]{G\'abor Cz\'edli}
\email{czedli@math.u-szeged.hu}
\urladdr{http://www.math.u-szeged.hu/~czedli/}
\address{University of Szeged, Bolyai Institute. 
Szeged, Aradi v\'ertan\'uk tere 1, HUNGARY 6720}

\begin{abstract}  In 2009, \init{G.\ }Gr\"atzer and \init{E.\ }Knapp proved that every planar semimodular lattice has a rectangular extension. We prove that, under reasonable additional conditions, this extension is unique. This theorem naturally leads to a hierarchy of special  diagrams of  planar semimodular lattices. Besides that these diagrams are unique in a strong sense,  we explore many of their further properties. Finally, we demonstrate the power of our new diagrams in two ways. First, we prove a simplified version of  our earlier Trajectory Coloring Theorem, which describes the inclusion $\con({\inp})\supseteq\con{(\inq)}$ for prime intervals $\inp$ and $\inq$ in slim rectangular lattices. Second, we prove  \init{G.\,}Gr\"atzer's Swing Lemma for the same lattices, which describes the same inclusion more simply.
\end{abstract}

\thanks{This research was supported by the NFSR of Hungary   (OTKA), grant number K83219} 




\maketitle

\begin{figure}[htb]
\includegraphics[scale=1.0]{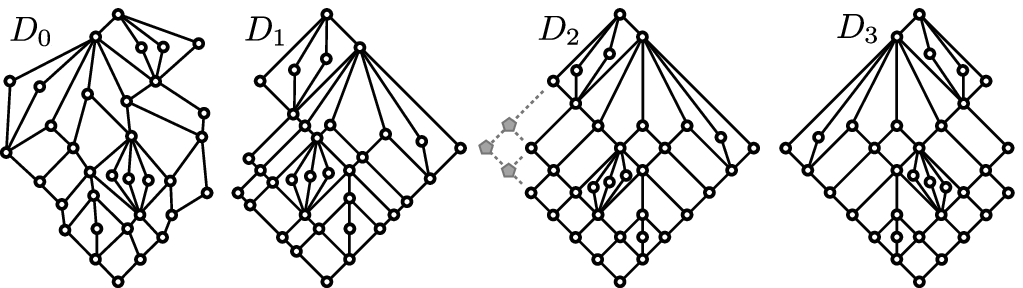}
\caption{$D_0\in \adia\setminus \bdia$, $D_1\in \bdia\setminus \cdia$, $D_2\in \cdia\setminus \ddia$, and  $D_3\in \ddia$  }\label{fig-hierarchy}
\end{figure}

\begin{figure}[htb]
\includegraphics[scale=1.0]{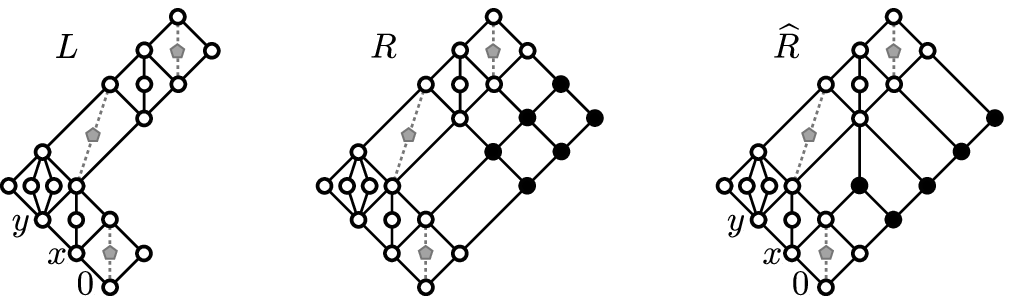}
\caption{$R$ is the \rectext{} but $|\widehat R|<|R|$}\label{fig-normext-notmin}
\end{figure}

\section{Introduction}  
A \emph{planar lattice} is finite lattice that has a planar (Hasse) diagram. All lattices in this paper are assumed to be \emph{finite}, even is this is not emphasized all the time. 
With the appearance  of \init{G.\ }Gr\"atzer and \init{E.\ }Knapp~\cite{gratzerknapp1} in 2007, the theory of planar semimodular lattices became a very  intensively studied branch of lattice theory. This activity is witnessed by more than two dozen papers; 
some of them are included among the References section while some others are overviewed in the  book chapter \init{G.\ }Cz\'edli and \init{G.\ }Gr\"atzer~\cite{czgggltsta}. The study of planar semimodular lattices and, in particular, slim planar semimodular lattices can be motivated by three factors.

First, these lattices are \emph{general enough}; for example 
\init{G.\ }Gr\"atzer, \init{H.\ }Lakser, and \init{E.\,T.\ }Schmidt~\cite{r:GratzerLakserSchmidt} proved that  every finite distributive lattice can be represented as the congruence lattice of a planar semimodular lattice $L$. In addition, one can also stipulate that every congruence of $L$ is principal, see \init{G.\ }Gr\"atzer and \init{E.\,T.\ }Schmidt~\cite{ggscht-periodica2014}. 
Even certain maps between two finite distributive lattices can be represented;
see \init{G.\ }Cz\'edli~\cite{czgrepres} for the latest news in this direction, and see its bibliography for many earlier results.

Second, these lattices offer \emph{useful links} between lattice theory and the rest of mathematics. For example, \init{ G.\ }Gr\"atzer and \init{J.\,B.\ }Nation~\cite{gr-nation} and, by adding a uniqueness part to it, \init{G.\,}Cz\'edli and \init{E.\,T.\ }Schmidt~\cite{czgschtJH}, improve the classical Jordan--H\"older theorem for groups from the nineteenth century. Also, these lattices are connected with combinatorial structures, see \init{G.\ }Cz\'edli~\cite{czgcircles} and \cite{czgquasiplanar}, and they raise interesting combinatorial problems, see 
\init{G.\ }Cz\'edli, \init{T.\ }D\'ek\'any, \init{L.\ }Ozsv\'art, \init{N.\ }Szak\'acs, and \init{B.\ }Udvari~\cite{czgdekanyatall} and its bibliography.

Third, there are  \emph{lots of tools} to deal with these lattices; see, for example, 
\init{G.\,}Cz\'edli~\cite{czg-mtx}, \cite{czgtrajcolor}, \cite{czgcircles}, 
 \init{G.\ }Cz\'edli and \init{G.\ }Gr\"atzer~\cite{czgggresect},   \init{G.\,}Cz\'edli and \init{E.\,T.\ }Schmidt~\cite{czgschvisual},  \cite{czgschslim2}, and \cite{czgschperm}, and \init{G.\ }Gr\"atzer and \init{E.\ }Knapp~\cite{gratzerknapp1} and \cite{gratzerknapp3}
; see also  \init{G.\ }Cz\'edli and \init{G.\ }Gr\"atzer~\cite{czgggltsta}, where most of these tools are overviewed; many of them are needed here.

\begin{figure}[htb]
\includegraphics[scale=1.0]{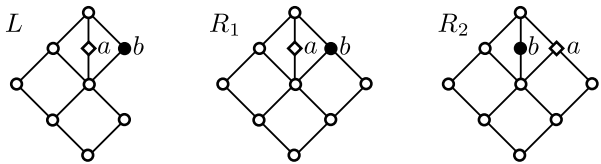}
\caption{Isomorphic but not relatively over $L$}\label{fig-nonreliso}
\end{figure}

\subsection*{Target} The first goal is to extend a planar semimodular lattice into a unique rectangular lattice. Definitions will be given soon.  For a first impression, if we add the three grey pentagon-shaped elements together with the grey dotted edges to $D_2$ in Figure~\ref{fig-hierarchy}, then we obtain its rectangular extension. While the existence of such an extension is known from \init{G.\ }Gr\"atzer and \init{E.\ }Knapp~\cite{gratzerknapp3}, its uniqueness needs some natural additional assumptions and a nontrivial proof. 

The second goal, which originates from the first one,  is to  associate a special diagram with each  planar semimodular lattice $L$. Besides the class $\adia$ of planar diagrams of slim semimodular lattices, we define a hierarchy $\adia\supset \bdia\supset \cdia \supset \ddia$ of (classes of) diagrams. 
For a   first impression, we present Figure~\ref{fig-hierarchy}, where the grey pentagon-shaped elements do not belong to $D_2$ and each of the four diagrams determine the same planar semimodular lattice. Also, we list some of the diagrams or lattices whose diagrams are depicted in the paper:
\begin{enumeratei}
\item In $\bdia\setminus \cdia$, we have $L$ and $R$ of Figure~\ref{fig-normext-notmin} and  Figures~\ref{fig-forkext-bdia}, \ref{fig-idealinDj1}, and \ref{fig-samidealinDj}.
\item In $\cdia\setminus\ddia$, we have  $\widehat R$ in  Figure~\ref{fig-normext-notmin} and $L_1$ and $R_1$ in  Figure~\ref{fig-nonindecomposable}.
\item In $\ddia$, we have  $L_2$ and $R_2$ in Figure~\ref{fig-nonindecomposable}, $D$ and $E$ in Figure~\ref{fig-centgravleft},
and Figures~\ref{fig-nonreliso}  and  Figure~\ref{fig-oterr}.
\end{enumeratei}
Although the systematic study and several statements on $\cdia$, $\ddia$, even on $\adia$ and, mainly, on $\bdia$ are new, note that we have often used diagrams from $\bdia$ and $\cdia$ previously.  Choosing  a  smaller hierarchy class, the diagrams of $L$ become  unique in a stronger sense.   For example, in the plane of complex numbers (with $0,1\in\mathbb C$ fixed),  $L$ has exactly one diagram that belongs to $\ddia$. 
Besides introducing new diagrams, we prove several useful properties for them.
While $\cdia$ and $\ddia$ seem to have only some aesthetic advantage over $\bdia$, the passage from $\adia$ to $\bdia$ gives some extra  insight into the theory of planar semimodular lattices. 

Finally,  to demonstrate that our diagrams and the toolkit we elaborate are useful,  we 
improve the Trajectory Coloring Theorem from \init{G.\,}Cz\'edli~\cite[Theorem 7.3.(i)]{czgtrajcolor}, which describes the ordered set of join-irreducible congruences of a slim rectangular lattice. 
The improved version is based on $\bdia$;  it is more pictorial and easier to understand and apply than the original one. As a nontrivial joint application of the improved Trajectory Coloring Theorem and our toolkit for $\bdia$, we  prove \init{G.\ }Gr\"atzer's Swing Lemma for slim rectangular lattices.  
The Swing Lemma gives a particularly elegant condition for $\con(\inp) \geq \con(\inq)$, where $\inp$ and $\inq$ are prime intervals of a slim rectangular lattice. Although we know from \init{G.\ }Gr\"atzer~\cite{swinglemma} that this lemma holds also for a larger class of lattices, the slim semimodular ones, the lion's share of the difficulty  is to conquer the slim rectangular case.

\subsection*{Outline} The present section is introductory. In Section~\ref{mainresultsection}, we introduce the concept of a \rectext{} of a slim semimodular lattice, and state its uniqueness in Theorem~\ref{thmmain}.  
Also, this section contains some analysis of this theorem and that of the way we prove it in subsequent sections.
To make the paper easier to read, some concepts and earlier results are surveyed in Section~\ref{preparesection}. 
Section~\ref{proofssection} is devoted to the proof of Theorem~\ref{thmmain}, but many of the auxiliary statements are of further interest. Namely, Lemma~\ref{dkjTghNVc} on cover-preserving sublattices of slim semimodular lattices, Lemmas~\ref{sldiHjNxY}  and \ref{samecoordin} on join-coordinates, Lemma~\ref{outlemmA} on the explicit description of \rectext{}s, and Lemma~\ref{slimcompmml} on the categorical properties of the antislimming procedure  deserve separate mentioning here.
In Section~\ref{uniquediagramssection}, a hierarchy $\adia\supseteq\bdia\supseteq\cdia\supseteq \ddia$ of classes of diagrams of planar semimodular lattices is introduced and appropriate  uniqueness statements are proved. Here we only mention Proposition~\ref{kTrTzmyY} on $\adia$, which extends the scope of a known result from ``slim semimodular'' to ``planar semimodular", and  Theorem~\ref{pcvWsWdG} on $\bdia$. 
Section~\ref{toolkitsection} proves several easy statements on diagrams in $\bdia$ and their trajectories. The rest of the paper  demonstrates the usefulness of $\bdia$ and the toolkit presented in  Section~\ref{toolkitsection}. 
Section~\ref{congruenceSection} improves the Trajectory Coloring Theorem,
while  Section~\ref{swingsection} proves G.\ Gr\"atzer's Swing Lemma for slim rectangular lattices.

\subsection*{Method}  
Our lattices are planar and they are easy to imagine. Thus,  intuition gives many ideas 
on their properties. However, experience shows that many of these ``first ideas'' are wrong or need serious improvements. Therefore, instead of relying too much on pictorial intuition, we give rigorous proofs for many auxiliary statements. Fortunately, we can use  a rich toolkit available in the referenced papers, including   \init{D.\ }Kelly and \init{I.\ }Rival  \cite{kellyrival} and  \init{G.\ }Gr\"atzer and \init{E.\ }Knapp~ \init{G.\ }Gr\"atzer and \init{E.\ }Knapp~\cite{gratzerknapp1} and \cite{gratzerknapp3}  as the pioneering sources. 

To prove Theorem~\ref{thmmain} on \rectext{}s, we coordinatize our lattices. Although our terminology is different, the coordinates we use are essentially the largest homomorphic images with respect to the 2-dimensional case of  \init{M.\ }Stern's join-homomorphisms in \cite{stern}, which were rediscovered in  \init{G.\,}Cz\'edli and \init{E.\,T.\ }Schmidt~\cite[Corollary 2]{czgschthowtoderive}. Note that the coordinatization used in this paper has nothing to do with the one used in \init{G.\,}Cz\'edli~\cite{czgcoord}. 

By a \emph{grid} we mean the direct product of two finite nontrivial (that is, non-singleton) chains.   Once we have coordinatization, it is natural to position the elements in a grid according to their coordinates. Of course, we have to prove that this plan is compatible with planarity.  This leads to a hierarchy of planar diagrams with useful properties. The emphasis is put on the properties of trajectories, because they are  powerful tools to understand slim rectangular lattices and their congruences.

Although we mostly deal with slim rectangular lattices in this paper, many of our statements can be extended to slim semimodular lattices in a straightforward but sometimes a bit technical way. Namely,  one can follow  \init{G.\,}Cz\'edli~\cite[Remark 8.5]{czgtrajcolor} or he can use Theorem~\ref{thmmain}. Because of space considerations, we do not undertake this task now.

\subsection*{Prerequisites} 
The reader is assumed to have some familiarity with lattices but not much. 
Although  widely known concepts  like semimodularity  are not defined here and a lot of specific statements and concepts are used from the recent literature, these less known constituents are explained here. Unless he wants to check the imported tools for correctness,  the reader hardly  has to look into the referenced literature while reading the present paper.

\section{\Rectexts}\label{mainresultsection}
Following \init{G.\,}Cz\'edli and \init{E.\,T.\ }Schmidt~\cite{czgschslim2}, a \emph{glued sum indecomposable lattice} is a finite non-chain lattice $L$ such that each $x\in L\setminus\set{0,1}$ is incomparable with some element of $L$. Such a lattice consists of at least 4 elements. Following \init{G.\ }Gr\"atzer and \init{E.\ }Knapp~\cite{gratzerknapp3}, a \emph{rectangular lattice} is a planar semimodular lattice $R$ such that $R$ has a planar diagram $D$ with the following properties:
\begin{enumeratei}
\item\label{rttlgpqa} $D\setminus\set{0,1}$ has exactly one double irreducible element on the left boundary chain of $D$; this element, called \emph{left corner}, is denoted by  $\cornl D$.
\item\label{rttlgpqb} $D\setminus\set{0,1}$ has exactly one double irreducible element,  $\cornr D$, on the right boundary chain of $D$. It is called the \emph{right corner} of $D$.
\item\label{rttlgpqc} These two elements are complementary, that is, $\cornl D\wedge \cornr D=0$ and  $\cornl D\vee \cornr D=0$.
\end{enumeratei}
Note that a rectangular lattice has at least four elements. 
Following \init{G.\,}Cz\'edli and \init{E.\,T.\ }Schmidt~\cite{czgschtJH}, a lattice $L$ is \emph{slim}, if it is finite and $\Jir L$, the (ordered) set of (non-zero) join-irreducible elements of $L$, is the union of two chains. 
It  follows from  \init{G.\,}Cz\'edli and \init{E.\,T.\ }Schmidt~\cite[page 693]{czgschslim2} that, for a \emph{slim}  semimodular lattice $L$,
\begin{equation}
\parbox{9cm}{$L$ is rectangular iff  $\Jir L$ is the union of two chains, $W_1$ and $W_2$, such that $w_1\wedge w_2=0$ for all $\pair{w_1}{w_2}\in W_1\times W_2$.}
\label{tpwcgTslm}
\end{equation}
We know from  \init{G.\ }Cz\'edli and \init{G.\ }Gr\"atzer~\cite[Ecercise 3.55]{czgggltsta} (which follows from  \eqref{tpwcgTslm}, \cite[Lemma 6.1(ii)]{czgschslim2}, and 
\init{G.\ }Cz\'edli and \init{G.\ }Gr\"atzer~\cite[Theorem 3-4.5]{czgggltsta}) that 
\begin{equation}
\parbox{8.6 cm}{
if one planar diagram of a semimodular lattice $L$ satisfies \eqref{rttlgpqa}--\eqref{rttlgpqc} above, then so do all planar diagrams of $L$.}
\label{1rctallrect}
\end{equation}
Let us emphasize that slim lattices, planar lattices, and rectangular lattices are finite by definition. Since a slim lattice is necessarily planar by \init{G.\,}Cz\'edli and \init{E.\,T.\ }Schmidt~\cite[Lemma 2.2]{czgschtJH}, we usually say ``slim'' rather than ``slim planar''.

\begin{definition} \label{dfrctext}
Let $L$  be a planar semimodular lattice. We say that a lattice $R$ is a \emph{\rectext{}} of $L$ if the following hold.
\begin{enumeratei}
\item\label{dfrctexta} $R$ is a rectangular lattice.
\item\label{dfrctextb} $L$ is a cover-preserving $\set{0,1}$-sublattice of $R$.
\item\label{dfrctextc} For every $x\in R$, if $x$ has a lower cover outside $L$, then $x$ has at most two lower covers in $R$.
\end{enumeratei}
\end{definition}

In Figure~\ref{fig-normext-notmin}, $R$ is a \rectext{} of $L$ but 
$\widehat R$ is not; no matter if we consider the pentagon-shaped grey-filled elements with the dotted edges or we omit them. This example witnesses that a \rectext{} of $L$ need not be a minimum-sized rectangular, cover-preserving extension of~$L$.

If $R_1$ and $R_2$ are extensions of a lattice $L$ and $\phi\colon R_1\to R_2$ is a lattice isomorphism whose restriction $\restrict\phi L$  to $L$ is the identity map, then $\phi$ is a \emph{relative isomorphism over} $L$.

\begin{theorem}\label{thmmain}
If $L$ is a planar semimodular lattice with more than two elements, then the following two statements hold.
\begin{enumeratei}
\item\label{thmmaina} $L$ has a \rectext{}. 
\item\label{thmmainb}  $L$ is slim iff it has a slim  \rectext{} iff all \rectext{}s of $L$ are slim.
\end{enumeratei}
Moreover, if $L$ is a glued sum indecomposable planar semimodular lattice, then even the
the following three statements also hold. 
\begin{enumeratei}
\setcounter{enumi}{2}
\item\label{thmmainc}  The \rectext{} of $L$ is unique up to  isomorphisms. 
\item\label{thmmaind} If in addition,  $L$ is slim, then its  \rectext{} is unique up to relative isomorphisms over $L$. 
\item\label{thmmaine} Furthermore,  if  $L$ is slim and  $\psi\colon L\to L'$ is a lattice isomorphism, $R$ is a \rectext{} of $L$, and $R'$ is that of $L'$, then $\psi$ extends to a lattice isomorphism $R\to R'$.
\end{enumeratei}
\end{theorem}

For a variant of this theorem in terms of diagrams,  
see Proposition~\ref{dgmpRopmain} later.
The two-element lattice cannot have a \rectext. 
Although a finite chain $C$ has a \rectext{} if $|C|\geq 3$, it is not unique up to relative automorphisms over $L$ in case $|C|\geq 5$.
Figure~\ref{fig-nonreliso}, where both $R_1$ and $R_2$ are \rectext{}s of $L$, shows that slimness cannot be removed from part  \eqref{thmmainc}. Figure~\ref{fig-nonindecomposable} shows that  glued sum indecomposability is also inevitable. In this figure, $L_1\cong L_2$ are isomorphic slim semimodular lattices but they are not glued sum indecomposable. Their diagrams are similar in the sense of \init{D.\ }Kelly and \init{I.\ }Rival  \cite{kellyrival}, so they are the same in $\adia$-sense, to be defined in Section~\ref{uniquediagramssection}. 
For $i\in\set{1,2}$, $R_i$ is a \rectext{} of $L_i$. However,  $R_1\ncong R_2$ since $|R_1|\neq |R_2|$.

\begin{figure}[htb]
\includegraphics[scale=1.0]{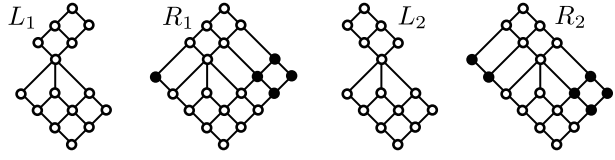}
\caption{The $L_i$ (black-filled) are isomorphic but the $R_i$ are not}\label{fig-nonindecomposable}
\end{figure}

Let $L_1$ be a sublattice of another lattice, $L_2$. We say that $L_2$ is a \emph{congruence-preserving extension} of $L_1$ if the restriction map $\Con {L_2}\to \Con{L_1}$ from the congruence lattice of $L_2$ to that of $L_1$ defined by $\balpha\mapsto \balpha\cap (L_1\times L_1)$ is a lattice isomorphism. We know from \init{G.\ }Gr\"atzer and \init{E.\ }Knapp~\cite[Theorem 7]{gratzerknapp3} that every planar  semimodular lattice has a rectangular congruence-preserving extension. Analyzing their proof, it appears that they construct a \rectext. Hence, using the uniqueness granted by Theorem~\ref{thmmain}, we obtain the following statement; note that it also follows from \init{G.\ }Cz\'edli~\cite[Lemmas 5.4 and 6.4]{czgrepres}.  

\begin{corollary}[{compare with \init{G.\ }Gr\"atzer and \init{E.\ }Knapp~\cite[Theorem 7]{gratzerknapp3}}]
\label{dkHjGwC}
If $L$ is a planar semimodular lattice, then its \rectext{} is a congruence-preserving extension of $L$.
\end{corollary}

\begin{remark} Omit the pentagon-shaped grey-filled elements and the dotted edges from Figure~\ref{fig-normext-notmin}. Then,  
 as opposed to $R$, the \rectext{} of $L$, $\widehat R$  is not a congruence-preserving extension of $L$,
because $\pair xy\in\con(0,x)$ holds in $\widehat R$ but fails in $L$.  Also, if we omit 1 and the rightmost coatom from this $L$, then the remaining planar semimodular lattice has two non-isomorphic minimum-sized cover-preserving extensions.    
\end{remark}

\begin{remark} Consider the lattices in Figure~\ref{fig-normext-notmin} together with  the pentagon-shaped grey-filled elements and the dotted edges. Then
$R$ is a \rectext{} of $L$, $|\widehat R|<|R|$, and  $\widehat R$  is  a congruence-preserving extension of $L$ since both $L$ and  $\widehat R$ are simple lattices. 
\end{remark}

These two remarks explain why we  deal with \rectext{}s rather than with minimum-sized ones or with congruence-preserving ones.

The fact that the construct in  \init{G.\ }Gr\"atzer and \init{E.\ }Knapp~\cite[Theorem 7]{gratzerknapp3} 
turns out to be a \rectext{} of $L$ does not imply Theorem~\ref{thmmain} in itself. First, because their proof does not say anything about uniqueness. Second, because our  definition of a \rectext{} does not make it evident that theirs is the only way to obtain such an extension. For example, they never insert a new element into the interior of their diagram but we have to prove that Definition~\ref{dfrctext} also excludes this  possibility.

For a given $n$, up to isomorphism, there are finitely many slim semimodular lattices of length $n$; their number is determined in 
\init{G.\ }Cz\'edli, \init{L.\ }Ozsv\'art, and  \init{B.\ }Udvari \cite{czgolub}. With the temporary notation  $f(n)=\max\{|L|: L$  is a slim semimodular lattice of length $n\}$, one may have the idea of proving Theorem~\ref{thmmain}\eqref{thmmainc}  by induction on $f(\length(L))-|L|$. Although such a proof seems to be possible and, probably, it would be somewhat shorter than the proof we are going to present here, our approach has two advantages. First, it gives an explicit formula for the \rectext{} rather than a recursive one; see Lemmas~\ref{samecoordin} and \ref{outlemmA}. 
 Second,  it is the present approach that leads us directly to a better understanding of slim semimodular lattices, as it is  witnessed by Sections~\ref{uniquediagramssection} and \ref{congruenceSection}. In particular, the
explicit description of a \rectext{} is heavily  used in the proof of Theorem~\ref{pcvWsWdG}.

\section{Preparations before the proof of Theorem~\ref{thmmain}}\label{preparesection}
For the reader's convenience, this section collects briefly the most important conventions, concepts, and tools needed in our proofs. Note that, with much more details, the majority of this section is covered by the  book chapter \init{G.\ }Cz\'edli and \init{G.\ }Gr\"atzer~\cite{czgggltsta}.
This paper is on \emph{planar} semimodular lattices. Unless otherwise stated, we always assume implicitly that a fixed planar diagram of the lattice under consideration is given. Some concepts, like ``left'' of ``right'', may depend on the diagram. However, the  choice of the diagram is irrelevant in the statements and proofs.  Later in Sections~\ref{uniquediagramssection}, \ref{congruenceSection}, and \ref{swingsection},  we focus explicitly on diagrams rather than lattices, and we apply lattice adjectives, like ``slim'' or ``semimodular'', to the corresponding diagrams as well. Also, if $D_i$ is a planar diagram of $L_i$  for $i\in \set{1,2}$, then we do not make a distinction between a map from $L_1$ to $L_2$ and the corresponding map from $D_1$ to $D_2$. This allows us to speak lattice isomorphisms between diagrams. Similarly, we can use the statements and concepts that are introduced in Section~\ref{proofssection} for lattices also for diagrams later. 

For a maximal chain $C$ of a planar lattice $L$, the set of elements $x\in L$ that are on the left of $C$ is the \emph{left side} of $C$, and it is denoted by $\lside( C)$. The \emph{right side} of $C$, $\rside(C)$, is defined similarly. Note that $C=\lside( C)\cap \rside(C)$. If $x\in \lside(C)\setminus C$, then $x$ is \emph{strictly on the left} of $C$; ``\emph{strictly on the right}'' is defined analogously.
Let us emphasize that, for an element $x$ and a maximal chain $C$, ``left'' and ``right'' is always understood in the wider sense that allows $u\in C$. We need some results from \init{D.\ }Kelly and \init{I.\ }Rival  \cite{kellyrival}; the most frequently used one is the following. 

\begin{lemma}[{\init{D.\ }Kelly and \init{I.\ }Rival  \cite[Lemma 1.2]{kellyrival}}]\label{kellyrivallemma}
Let $L$ be a finite planar lattice, and let $x\leq y\in L$. If $x$ and $y$ are on different sides of a maximal chain $C$ in $L$, then there exists an element $z\in C$ such that $x\leq z\leq y$.
\end{lemma}

Next, let $x$ and $y$ be elements of a finite planar lattice $L$, and assume that they are incomparable, written as $x\parallel y$. If $x\vee y$ has lower covers $x_1$ and $y_1$ such that $x\leq x_1\prec x\vee y$, $y\leq y_1\prec x\vee y$, and $x_1$ is on the left of $y_1$, then the element $x$ is \emph{on the left} of the element $y$. In notation, $x\rellambda y$. 
 If $x\rellambda y$, then we also say that $y$ is \emph{on the right} of $x$. Let us emphasize that whenever $\lambda$, that is ``left'', or ``right'' are used for two elements,  then the two elements in question are incomparable. That is, the notation $x\rellambda y$ implies $x\parallel y$.  Note the difference; while $\rellambda$ is an irreflexive relation for elements,  ``left'' and ``right'' are used in the wider sense if an element and a maximal chain are considered.

\begin{lemma}[{\init{D.\ }Kelly and \init{I.\ }Rival  \cite[Propositions 1.6 and 1.7]{kellyrival}}]\label{leftrightlemma} Let $L$ be finite planar lattice. If  $\,x,y\in L$ and $x\parallel y$, then the following hold.
\begin{enumeratei}
\item\label{leftrightlemmaa} $x\rellambda y$ if and only if  $x$ is on the left of some maximal chain through $y$ if and only if $x$ is on the left of all maximal chains through $y$.
\item\label{leftrightlemmab} Either $x\rellambda y$, or $y\rellambda x$.
\item\label{leftrightlemmac} If $z\in L$,  $x\rellambda y$, and   $y\rellambda  z$, then $x\rellambda z$.
\end{enumeratei}  
\end{lemma}

Let $L$ be a slim semimodular lattice.
According to  the general convention in the paper, a planar diagram of  $L$ is fixed. 
Let  $\inp_i=[x_i,y_i]$ be prime intervals, that is, edges in the diagram, for $i\in\set{1,2}$. These two edges are \emph{consecutive} if they are opposite sides of a covering square, that is, of a 4-cell in the diagram. Following \init{G.\ }Cz\'edli and \init{E.\,T.\ }Schmidt~\cite{czgschtJH}, an equivalence class of the transitive reflexive closure of the ``consecutive" relation is called a \emph{trajectory}. Recall from \cite{czgschtJH} that 
\begin{equation}
\parbox{9cm}{a trajectory begins with an edge on the left boundary chain $\leftb L$, it goes from left to
right, it cannot  branch out, and it terminates at an edge on the right boundary chain, $\rightb L$.}
\label{trjtrffrls}
\end{equation}
These boundary chains also important because  of 
\begin{equation}
\Jir L\subseteq \leftb L\cup\rightb L;
\label{ddHtBgRw}
\end{equation}
see \init{G.\,}Cz\'edli and \init{E.\,T.\ }Schmidt~\cite[Lemma 6]{czgschvisual}.

According to  \init{G.\ }Cz\'edli and \init{G.\ }Gr\"atzer~\cite{czgggresect}, there are three types of trajectories: an \emph{up-trajectory}, which goes up (possibly, in zero steps), a \emph{down-trajectory}, which goes down (possibly, in zero steps), and a \emph{hat-trajectory},
which goes up (at least one step), then turns to the lower right, and finally it
goes down (at least one step). Let 
$\inp_1=[x_1,y_1]$, $\inp_2=[x_2,y_2]$, and $\inp_3=[x_3,y_3]$ be three consecutive edges of a trajectory $T$, listed from left to right. If $y_1<y_2<y_3$, then $T$ \emph{goes upwards} at $\inp_2$.  Similarly,  $T$ \emph{goes downwards} at $\inp_2$  if $y_1>y_2>y_3$. The only third possibility is that 
$y_1<y_2>y_3$; then $T$ is a hat-trajectory and $\inp_2$ is called its \emph{top edge}. 
If $x_1$ and $y_1$ are on the left boundary chain, then we say that the trajectory containing    $\inp_1=[x_1,y_1]$ and  $\inp_2=[x_2,y_2]$ goes  upwards or downwards at $\inp_1$ if $y_1<y_2$ or $y_1>y_2$, respectively. 
Since there are only three types of trajectories,  if $\inp_1$ is on the left of $\inp_2$ in a trajectory $T$ of $L$, then 
\begin{equation}\parbox{7.5cm} {if $T$ goes upwards at $\inp_2$ then so it does at $\inp_1$, and if $T$ goes downwards at $\inp_1$ then so it does at $\inp_2$.} 
\label{trtrfR} 
\end{equation}

\section{Proving some lemmas and Theorem~\ref{thmmain}}\label{proofssection}

If  $C_1$ and $C_2$ are maximal chains of planar lattice $L$ such that $C_1\subseteq \lside(C_2)$, then $\rside(C_1)\cap\lside(C_2)$ is called  a \emph{region} of $L$. 
For a subset $X$ of $L$, we know from \init{G.\ }Cz\'edli and \init{G.\ }Gr\"atzer~\cite[Exercise 3.12]{czgggltsta} that the predicate ``$X$ is a region of $L$'' does not depend on the choice of the planar diagram. The following  lemma is  of separate interest.

\begin{lemma} \label{dkjTghNVc}
If $K$ is a cover-preserving $\set{0,1}$-sublattice of a slim semimodular lattice $L$, then $K$ is also a slim semimodular lattice, it is a region of $L$, and $K= \rside_L(\leftb K) \cap \lside_L(\rightb K)$.
\end{lemma}

\begin{proof} For $x\in L$, the \emph{left support} and the \emph{right support} of $x$, denoted by $\lsp(x)=\lsp_L(x)$ and $\rsp(x)=\rsp_L(x)$, are the largest element of $\leftb L\cap\ideal x$ and that of $\rightb L\cap\ideal x$, respectively. Since $\Jir L\subseteq \leftb L\cup\rightb L$ and $\leftb L$ and $\rightb L$ are chains, it is straightforward to see that,  for every $x\in L$, $y\in  [\lsp_L(x),x]$, and  $z\in  [\lsp_L(x),x]$, 
\begin{equation}
\parbox{9 cm}{
$x=\lsp_L(x)\vee \rsp_L(x)$,  $[\lsp_L(x),x]$ and $[\rsp_L(x),x]$ are\\ chains, $\lsp_L(y)=\lsp_L(x)$, and    $\rsp_L(z)=\rsp_L(x)$.}
\label{dkGHjT}
\end{equation}
Let $H=\rside_L(\leftb K) \cap \lside_L(\rightb K)$; it is the smallest region of $L$ that includes $K$. Consider an arbitrary element $x\in H$. Applying Lemma~\ref{kellyrivallemma} to $\lsp_L(x)\leq x$ and $\leftb K$, we obtain an element $a\in \leftb K$ such that $\lsp_L(x)\leq a\leq x$. Similarly, there is an element 
$b\in \rightb K$ such that $\rsp_L(x)\leq b\leq x$. Hence, $x=a\vee b\in K$ by \eqref{dkGHjT}. This shows that $K=H$ is a region, and it is a slim lattice since $\Jir K\subseteq \leftb K\cup\rightb K$. As a cover-preserving sublattice, $K$ inherits semimodularity.
\end{proof}

In the rest of this section, unless otherwise stated, we always assume that 
\begin{equation}
\parbox{7.5 cm}{
$L$ is a planar semimodular lattice of length $n\geq 2$ and  $R$ is a   \rectext{} of $L$.}
\label{ssmfTsCtwsta}
\end{equation} 
A planar diagram of $R$ is  fixed; it determines the diagram of $L$ as a subdiagram. Sometimes, we stipulate additional assumptions, including 
\begin{equation}
L\text{ is a glued sum indecomposable.}
\label{ssmfTsCtwstb}
\end{equation} 
Sometimes, for emphasis, we repeat \eqref{ssmfTsCtwsta} and \eqref{ssmfTsCtwstb}. By Lemma~\ref{dkjTghNVc}, 
\begin{equation}
L= \rside_R(\leftb L) \cap \lside_R(\rightb L)\text.
\label{LldkkLgdm}
\end{equation}

We know from 
\init{G.\ }Gr\"atzer and \init{E.\ }Knapp~\cite[Lemmas 3 and 4]{gratzerknapp3} (see also 
\init{G.\ }Cz\'edli and \init{G.\ }Gr\"atzer~\cite[Lemma 3-7.1]{czgggltsta}) that 
\begin{equation}
\text{the intervals }[0,\cornl R]\text{ and  }[0,\cornr R]\text{ are chains.}
\label{llsRChains}
\end{equation}
If $R$ is slim, then we also know from \eqref{tpwcgTslm},
\eqref{ddHtBgRw}, and \init{G.\ }Gr\"atzer and \init{E.\ }Knapp~\cite[Lemma 3]{gratzerknapp3},   see also 
\init{G.\ }Cz\'edli and \init{G.\ }Gr\"atzer~\cite[Exercises 3.51 and 3.52]{czgggltsta}, that  
\begin{equation}
\Jir R=\bigl(\llb R\cup\lrb R\bigr)\setminus\set0=\set{c_1,\dots,c_{\emel},d_1,\dots,d_{\emer}}\text.
\label{dkRtghWr}
\end{equation}

\begin{definition}\label{DFgjRd}
If \eqref{ssmfTsCtwsta} and $L$ is slim, then we agree in the following.
\begin{enumeratei}
\item\label{DFgjRda} Let $\llb R=[0,\cornl R]_R=\set{0=c_0\prec c_0\prec\dots\prec c_{\emel}}$ (lower left boundary) and $\lrb R=[0,\cornr R]_R=\set{0=d_0\prec d_0\prec\dots\prec d_{\emer}}$ (lower right boundary). Note that $c_{\emel}=\cornl D$ and $c_{\emer}=\cornr D$.
\item\label{DFgjRdb} For $x\in L$, the left and right \emph{join-coordinates} of $x$ are defined by 
$\lefth_L(x)=|\leftb L\cap\Jir L\cap\ideal x|$ and $\righth_L(x)=|\rightb L\cap\Jir L\cap\ideal x|$. The superscript ``n'' in the notation reminds us that they are numbers. 
It follows from \eqref{ddHtBgRw} that $x$ is determined by the pair $\pair{\lefth_L(x)}{\righth_L(x)}$ of its join coordinates; namely, we have that
\begin{equation}
x = c_{\lefth_L(x)}\vee d_{\righth_L(x)}\text.
\label{djdzBgRWv}
\end{equation}
\item\label{DFgjRdc} For $x\in R$, we obtain $\lefth_R(x)$ and $\righth_R(x)$ by substituting $R$ to $L$ above.
By  \eqref{dkRtghWr}, understanding  $\wedge$  in $\tuple{\mathbb N_0;\leq}$  , equivalently, we have  that 
\begin{equation*}
\phantom{mmmm}
\lefth_R(x)=\emel\wedge \height(\lsp_R(x)),\quad\righth_R(x)=\emer\wedge\height(\rsp_R(x))\text.
\end{equation*}
Note that, for $x,x',y,y'\in R$ with   $x\not> c_{\emel}$ and $y\not> d_{\emer}$, 
\begin{align}
&\lsp_R(x)=c_{ \lefth_R(x)},\quad \rsp_R(y)=d_{ \righth_R(y)},\label{dkjGHhPq} \\
&\begin{aligned}
\lefth_R(x')<\lefth_R(y') &\then \lsp_R(x')<\lsp_R(y'),\cr
\righth_R(x')<\righth_R(y') &\then \rsp_R(x')<\rsp_R(y')\text.
\end{aligned}  \label{moddkjGHhPq}
\end{align}
\end{enumeratei} 
\end{definition}

The conditions $x\not> c_{\emel}$ and $y\not> d_{\emer}$ right before  \eqref{dkjGHhPq} could be inconvenient at later applications. Hence, we are going to formulate a related condition, \eqref{dkjGHhPqvar} below. As a preparation to do so, the set of \emph{meet-irreducible elements} of $R$ distinct from 1 is denoted by $\Mir R$. For $x\in R$, $x\in \Mir R$ iff $x$ has exactly one cover. 
 \init{G.\ }Gr\"atzer and \init{E.\ }Knapp~\cite[Lemma 3]{gratzerknapp3} or \init{G.\ }Cz\'edli and \init{G.\ }Gr\"atzer~\cite[Exercise 3.52]{czgggltsta} yields that
\begin{equation}
\parbox{9.5cm}{if $1\neq x\in(\leftb R\setminus\llb R)\cup (\rightb R\setminus\lrb R)$, then $x\in\Mir R$.}
\label{jvcnTgjYq}
\end{equation}
This implies, see also \init{G.\ }Gr\"atzer and \init{E.\ }Knapp~\cite[Lemma 4]{gratzerknapp3}, that
\begin{equation}
[\cornl R,1]=\filter{c_{\emel}} \text{ and }
[\cornr R,1]=\filter{d_{\emer}}\text{ are chains.} 
\label{upPerChains}
\end{equation}
Clearly, for every $x\in R$,   $c_{\lefth_R(x)}\leq \lsp_R(x)$ and $d_{\righth_R(x)}\leq \rsp_R(x)$. Thus, 
\begin{equation*}
[\lsp_R(x),x]\subseteq [c_{\lefth_R(x)} ,x]\quad\text{and}\quad [\rsp_R(x),x]\subseteq [d_{\righth_R(x)} ,x]\text.
\label{dkzhThfzRb}
\end{equation*}
We conclude from \eqref{dkGHjT},
\eqref{dkjGHhPq}, and \eqref{upPerChains}  that  for all $x\in R$,  if $y\in [c_{\lefth_R(x)} ,x]$ and $z\in[d_{\righth_R(x)},x]$, then
\begin{equation}
\lefth_R(y)=\lefth_R(x)\quad\text{and}\quad
\righth_R(z)=\righth_R(x)\text.
\label{dkjGHhPqvar}
\end{equation}
\vajonkell{
\begin{remark}\label{dlPhGrGbRt}
For later use in subsequent sections, note that Definition~\ref{DFgjRd}\eqref{DFgjRdb}  is meaningful for every planar lattice  diagram $D$ of a finite lattice $L$, no matter if it is slim or not. So $\lefth_D(x)$ and $\righth_D(x)$ or, if $D$ is fixed,  $\lefth_L(x)$ and $\righth_L(x)$  are defined.
\end{remark}
}
The elements of $R$ on the left of  $\leftb L$ form a region 
\[\rside_R\bigl(\leftb R\bigr)\cap \lside_R\bigl(\leftb L\bigr)=  \lside_R\bigl(\leftb L\bigr);
\] 
it is called \emph{the region left to $L$}, and we denote it by $S$.

\begin{lemma}\label{SdstrL} Assume \eqref{ssmfTsCtwsta}. 
\begin{enumeratei}
\item\label{SdstrLa} The  region left to $L$, denoted by $S=\lside_R\bigl(\leftb L\bigr)$,  is a cover-preserving 
$\set{0,1}$-sublattice of $R$ and it is distributive. 
\item\label{SdstrLb} $R$ is slim iff $L$ is slim.
\end{enumeratei}
\end{lemma}

\begin{proof}  As a region of $R$, $S$ is a cover-preserving sublattice of $R$ by  \init{D.\ }Kelly and \init{I.\ }Rival  \cite[Proposition 1.4]{kellyrival}. 
The inclusion $\set{0_R,1_R}\subseteq S$ is obvious. 
As a cover-preserving sublattice, $S$ is semimodular. As a region of a planar diagram, $S$ is a planar lattice. We know from 
\init{G.\ }Cz\'edli and \init{G.\ }Gr\"atzer~\cite[Theorem 3-4.3]{czgggltsta}, see also \init{G.\ }Cz\'edli and \init{E.\,T.\ }Schmidt~\cite[Lemma 2.3]{czgschtJH}, that 
\begin{equation}
\parbox{7.5cm}{a finite semimodular lattice is slim if it contains no cover-preserving diamond sublattice $M_3$.}
\label{slimiffnodiamond}
\end{equation}
This property holds for $S$ by Definition~\ref{dfrctext}\eqref{dfrctextc}, so $S$ is slim. Recall from  \init{G.\ }Cz\'edli and \init{E.\,T.\ }Schmidt~\cite[Lemma 15]{czgschvisual} or
\init{G.\ }Cz\'edli and \init{G.\ }Gr\"atzer~\cite[Exercise 3.30]{czgggltsta} that
\begin{equation}
\parbox{8.1cm}{if no element of a slim semimodular lattice covers more than 2 elements, then the lattice is distributive.}
\label{tHnDstRbtV}
\end{equation}
By Definition~\ref{dfrctext}\eqref{dfrctextc} again, no element of $S$ covers three elements.  Hence,  $S$ is distributive by  \eqref{tHnDstRbtV} . This proves part \eqref{SdstrLb} .

By \eqref{slimiffnodiamond},  if $R$ is slim, then so is $L$. Suppose, for a contradiction, that $L$ is slim but $R$ is not. By \eqref{slimiffnodiamond}, some element $x\in R$ is the top of a cover-preserving diamond. By Definition~\ref{dfrctext}\eqref{dfrctextc}, none of the coatoms (that is, the atoms) of this diamond are outside $L$. Hence, they are in $L$, the whole diamond is $L$, which contradicts the slimness of $L$ by \eqref{slimiffnodiamond}.  This proves part \eqref{SdstrLb}
\end{proof}


In the following statement,   $R$ is slim by Lemma~\ref{SdstrL}\eqref{SdstrLb}.

\begin{lemma}\label{sldiHjNxY}
If \eqref{ssmfTsCtwsta}, $L$ is slim, and $x,y\in L$, then 
\begin{align}
x\rellambda y &\iff \bigl( \lefth_L(x)>\lefth_L(y)\text{ and }\righth_L(x)<\righth_L(y)\bigr), 
\label{sldiHjNxYa}\\
x\leq y  &\iff \bigl(\lefth_L(x)\leq \lefth_L(y)\text{ and }\righth_L(x)\leq\righth_L(y)\bigr)\text.
\label{sldiHjNxYb}
\end{align}
If we substitute $R$ to $L$, then \eqref{sldiHjNxYa} and \eqref{sldiHjNxYb} still hold.
\end{lemma}

\begin{proof} The $\Rightarrow$ part of \eqref{sldiHjNxYb} is evident. To show the converse implication, assume that $\lefth_L(x)\leq \lefth_L(y)$ and  $\righth_L(x)\leq\righth_L(y)$. For $z\in L$, let  $z'$ and $z''$ be the largest element of $\Jir L\cap\leftb L\cap  \ideal z$ and that of $\Jir L\cap\rightb L\cap \ideal z$, respectively. By the inequalities we have assumed, $x'\leq y'$ and $x''\leq y''$. Since  $x=x'\vee x''$ and $y=y'\vee y''$ by \eqref{ddHtBgRw}, we obtain that $x\leq y$. Thus, \eqref{sldiHjNxYb} holds.

To prove \eqref{sldiHjNxYa}, recall from  \init{G.\,}Cz\'edli~\cite[Lemma 3.15]{czgquasiplanar} that
\begin{equation} 
x\rellambda y  \iff \bigl( \lsp_L(x)>\lsp_L(y)\text{ and }\rsp_L(x)<\rsp_L(y)\bigr)\text. 
\label{MsldiHkkQY}
\end{equation}
Assume that $x\rellambda y$. Then  $\lsp_L(x)>\lsp_L(y)$ by \eqref{MsldiHkkQY}, and we obtain that $\lefth_L(x)\geq\lefth_L(y)$. Similarly, $\righth_L(x)\leq\righth_L(y)$. Both inequalities must be sharp, because otherwise \eqref{sldiHjNxYb} would imply that $x\nparallel y$. Therefore, the $\then$ implication of \eqref{sldiHjNxYa} follows. Conversely, assume that $\lefth_L(x)>\lefth_L(y)$ and $\righth_L(x)<\righth_L(y)$. Clearly, $\lsp_L(x)>\lsp_L(x)$ and $\rsp_L(x)<\rsp_L(x)$. Hence, $x\rellambda y$ by \eqref{MsldiHkkQY}, which gives the desired converse implication of \eqref{sldiHjNxYa}.
\end{proof}

In the  the following lemma,  the subscripts  come from ``left'' and ``right'' and so they are not numbers. Hence,  we  write $x_\tel$ and $x_\ter$ rather than $x_l$ and $x_r$.

\begin{lemma} \label{tChnQlM} 
Assume that  \eqref{ssmfTsCtwsta} holds,  $L$ is slim,  $T$ is a trajectory of $R$,  and that  $[x,y]\in T$.  Let $[x_\tel,y_\tel]$ and $[x_\ter,y_\ter]$ be the leftmost $($that is, the first) and the rigthmost edge of $T$, respectively. If $T$ goes upwards at $[x,y]$, then $\lsp_R(x)=x_\tel < y_\tel\leq\lsp_R(y)$. Similarly, if $T$ goes downwards  at $[x,y]$, then $\rsp_R(x)=x_\ter< y_\ter\leq\rsp_R(y)$.
\end{lemma}

\begin{proof}  By left-right symmetry, we can assume that $T$ goes upwards at $[x,y]$.  
The segment of $T$ from  $[x_\tel,y_\tel]$  to $[x,y]$ goes up by \eqref{trtrfR}. Combining this fact with \eqref{jvcnTgjYq}, 
 it follows  that the edge  $[x_\tel,y_\tel]$ belongs to $\llb R$ and that   $y=y_\tel \vee x$. Hence, $y_\tel\nleq x$. Thus,  we obtain that $x_\tel=\lsp_R(x) $ and  $y_\tel\leq \lsp_R(y)$.
\end{proof}

The following lemma is of separate interest.

\begin{lemma}\label{samecoordin} If \eqref{ssmfTsCtwsta}, \eqref{ssmfTsCtwstb}, $L$ is slim, and  $x\in L$,  then the pair of join-coordinates of $x$ is the same in $L$ as in $R$.
\end{lemma}

A maximal chain as $L$ in the 4-element Boolean lattice as $R$ indicates that this lemma would fail without assuming that $L$ is glued sum indecomposable.

\begin{proof} Since $L$ is glued sum indecomposable, 
\begin{equation}
\leftb L\cap\rightb L=\set{0,1}\text{ and }
|L|\geq 4\text.
\label{dkdgzNmD}
\end{equation}
By semimodularity, $|\leftb L|=\length(L)+1=n+1=|\rightb L|$. 
Let 
\begin{equation*}\parbox{6.5cm} 
{$\leftb L=\set{0=e_0\prec e_1\prec\dots\prec e_n=1}$ and $\rightb L=\set{0=f_0\prec f_1\prec\dots\prec f_n=1}$.}
\label{fjHzTnS} 
\end{equation*}
We claim that, for $i\in \set{1,\dots, n}$, 
\begin{equation}
e_i\in \Jir L \iff \pair{\lhr{e_i}} {\rhr{e_i}}  =
\pair{1+ \lhr{e_{i-1}}}  { \rhr{e_{i-1}}} \text.
\label{fiTbNWsp}
\end{equation}
First, to prove the ``$\Leftarrow$'' direction of \eqref{fiTbNWsp}, assume that $\lhr{e_i}=1+ \lhr{e_{i-1}}$ and  $\rhr{e_i}=\rhr{e_{i-1}}$. 
Suppose, for  a contradiction, that $e_i\notin \Jir L$, and let $y\in L\setminus\set{e_{i-1}}$ be a lower cover of $e_i$. Since $e_{i-1}$ is on the left boundary of $L$, $e_{i-1}\rellambda y$.  Hence, we obtain from \eqref{sldiHjNxYa}  that  $\rhr{e_{i-1}} < \rhr y$. On the other hand, $e_i>y$ and \eqref{sldiHjNxYb}  yield that $\rhr{e_{i-1}}= \rhr{e_{i}}\geq  \rhr y$, which contradicts the previous inequality. Thus, the ``$\Leftarrow$'' part of \eqref{fiTbNWsp} follows.

To prove the converse direction of \eqref{fiTbNWsp}, assume that $e_i\in \Jir L$.
By \eqref{sldiHjNxYb},  
\begin{equation}
\pair {\lhr{e_i}}  {\rhr{e_i}}  >  \pair{\lhr{e_{i-1}}} {\rhr{e_{i-1}}}
\label{dkjhTr}
\end{equation}
in the usual componentwise ordering ``$\leq$'' of $\set{0,1,\dots,n}^2$.  We claim that 
\begin{equation}
{\rhr{e_i}} = {\rhr{e_{i-1}}}\text.
\label{dkjJHpZzT}
\end{equation}
To prove this by contradiction, suppose ${\rhr{e_i}} > {\rhr{e_{i-1}}}$. Applying Lemma~\ref{kellyrivallemma} in $R$ to $\rsp_R(e_i) \leq e_i$ and the maximal chain $\rightb L$,  
we obtain an element $z\in\rightb L\subseteq L$ such that $\rsp_R(e_i)\leq z\leq e_i$. Combining \eqref{moddkjGHhPq} and  ${\rhr{e_i}} > {\rhr{e_{i-1}}}$, we have that $z\nleq e_{i-1}$. Hence, $e_{i-1}\prec e_i$ gives that $e_{i-1}\vee z=e_i\in \Jir L$. So we conclude that $z=e_i$, that is, $0\neq e_i\in\leftb L\cap\rightb L$. From \eqref{dkdgzNmD}, we obtain that $1=e_i=f_i$ and $i=n$. Since $1=e_i=e_n\in\Jir L$ has only one lower cover, we obtain that $e_{n-1}=f_{n-1}\in \leftb L\cap\rightb L$, which contradicts \eqref{dkdgzNmD}. This proves \eqref{dkjJHpZzT}.

Combining  \eqref{dkjhTr} and \eqref{dkjJHpZzT}, we obtain that $\lhr{e_i}>\lhr{e_{i-1}}$. Let $T$ denote the trajectory of $R$ that contains $[e_{i-1},e_i]$. In the moment, there are three possible ways how $T$ can be related to $[e_{i-1},e_i]$ but we want to exclude two of them. First, suppose that  $[e_{i-1},e_i]$  is the top edge of a hat-trajectory. Then $e_i$ has a lower cover to the left of $e_{i-1}\in \leftb L$, so outside $L$, and $e_i$ has at least three lower covers. This possibility is excluded by Definition~\ref{dfrctext}\eqref{dfrctextc}. Hence, $[e_{i-1},e_i]$ cannot be the top edge of a hat-trajectory. (Note, however, that $T$ can be a hat-trajectory whose top is above  $[e_{i-1},e_i]$ in a straightforward sense.)   Second, suppose that $T$ goes  downwards at $[e_{i-1},e_i]$. Then   $\rsp_R(e_{i-1})$ is meet-reducible, because it is  the bottom of the last edge of $T$ by  Lemma~\ref{tChnQlM} and $T$ arrives downwards at this last edge by \eqref{trtrfR}. So \eqref{jvcnTgjYq} yields that $\rsp_R(e_{i-1})<d_{\emer}$, and we have that $e_{i-1}\ngeq d_{\emer}$.
Hence, there is a unique $j<\emer$ such that $\rsp_R(e_{i-1})=d_j$, and \eqref{dkjGHhPq} gives that  $j=\lefth_R(e_{i-1})$. 
Since  $\lefth_R(e_i)$ is also $j$ by   \eqref{dkjJHpZzT}, $d_{j+1}\nleq e_i$, and we obtain that $\rsp_R(e_i)=d_j=\rsp_R(e_{i-1})$. This contradicts Lemma~\ref{tChnQlM}  and excludes the possibility that $T$ goes downwards at $[e_{i-1},e_i]$.
Therefore,  $T$ goes upwards at $[e_{i-1},e_i]$. Let $[u_\tel,v_\tel]$ be the first edge of $T$. We know from Lemma~\ref{tChnQlM}  that $u_\tel=\lsp_R(e_{i-1})$. The left-right dual of the argument used in the excluded previous case yields that 
$u_\tel =  \lsp_R(e_{i-1}) = \lefth_R(e_{i-1}) < \emel$. If $ \lefth_R(e_{i-1}) = \emel-1$, then the required equality $\lefth_R(e_i)=1 +  \lefth_R(e_{i-1})$ follows from $\lefth_R(e_i) >  \lefth_R(e_{i-1})$ and from the fact that $\lefth_R(x)\leq \emel$ for all $x\in R$. Thus, we can assume that $ \lefth_R(e_{i-1}) \leq \emel-2$.  If  $e_i > c_{\emel}$, then $e_i$ is meet-irreducible and $e_{i-1}\geq c_{\emel}$  by \eqref{jvcnTgjYq} and \eqref{upPerChains}, and so $\lefth_R{e_{i-1}}=\emel$, contradicting  $\lefth_R{e_{i-1}}\leq \emel-2$. Hence, $e_i \not> c_{\emel}$, 
and the desired equation   $ \lhr{e_i}=1+ \lhr{e_{i-1}}$ and \eqref{fiTbNWsp} will follow from \eqref{dkjGHhPq}  if we show that $v_\tel=\lsp_R(e_i)$.

\vajonkell{
For later reference, we summarize that we want to show that
\begin{equation}\parbox{7cm} 
{if $e_i\in\leftb L\cap\Jir L$, the trajectory $T$ through $[e_{i-1},e_i]$ goes upwards at $[e_{i-1},e_i]$, and  $[u_\tel,v_\tel]$ is the first edge of
$T$, then $v_\tel=\lsp_R(e_i)$. 
} \label{dkHjGsX} 
\end{equation}
}
Suppose, for a contradiction, that $v_\tel\neq\lsp_R(e_i)$. We have that $v_\tel\ < \lsp_R(e_i)$, since $v_\tel\leq \lsp_R(e_i)$ is clear by $v_\tel\leq e_i$.  Also, $\lsp_R(e_i)\nleq e_{i-1}$, since $\lsp_R(e_i)\geq v_l > u_\tel=\lsp_R(e_{i-1})$ and $\lsp_R(e_{i-1})$ is the largest element of $\leftb R\cap\ideal e_{i-1}$.
Since $u_\tel$, $v_\tel$, and $ \lsp_R(e_i)$ are on the leftmost chain $\leftb R$ of $R$, these elements belong to $S$, the region left to $L$,  defined before Lemma~\ref{SdstrL}.  By Lemma~\ref{SdstrL}\eqref{SdstrLb}, $R$ is a slim rectangular lattice.
Observe that $\lsp_R(e_i)>u_\tel\in S$ 
by Lemma~\ref{tChnQlM}, which excludes $\lsp_R(e_i)=0_S$. We obtain from $e_i\not>c_{\emel}$ that $\lsp_R(e_i)\in\llb R$. 
Hence, \eqref{dkRtghWr}  
yields that $\lsp_R(e_i)\in   \Jir R$, and we conclude that   $\lsp_R(e_i)\in   \Jir S$.  
Using $e_{i-1}\prec e_i$ and $e_i\geq v_\tel \nleq e_{i-1}$, we conclude that  $\lsp_R(e_i)\leq e_i=v_\tel\vee e_{i-1}$. Since $S$ is distributive by Lemma~\ref{SdstrL} and the elements in the previous inequality belong to $S$, we have that
\begin{equation}
\lsp_R(e_i) =\lsp_R(e_i)\wedge(v_\tel \vee e_{i-1} ) = (\lsp_R(e_i)\wedge v_\tel )\vee (\lsp_R(e_i)\wedge e_{i-1} )\text.
\label{dstrSng}
\end{equation}
Since $\lsp_R(e_i)\in   \Jir S$ equals one of the two joinands above and $\lsp_R(e_i)\nleq e_{i-1}$,  we obtain that 
$\lsp_R(e_i)\leq v_\tel$. This contradicts  $v_\tel\ < \lsp_R(e_i)$. In this way, we have shown that 
$v_\tel=\lsp_R(e_i)$. This proves 
\vajonkell{\eqref{dkHjGsX}  and}%
\eqref{fiTbNWsp}.

Next,  with reference to the notation in Definition~\ref{DFgjRd}\eqref{DFgjRda},  we claim that
{\begin{equation}\parbox{8cm} 
{$(\forall j \in\set{0,\dots,{\emel}})\  \bigl(\exists e_i\in\leftb L\bigr)\,\bigl(\lefth_R(e_i)=j\bigr) $.
}\label{dlGhBTmnd}
\end{equation}
}%
To prove \eqref{dlGhBTmnd}, let $j\in\set{0,\dots,\emel}$. We can assume that $j<\emel$, since otherwise we can let   $e_i:=e_n=1\in \leftb L$. Due to \eqref{dkjGHhPq}, it suffices to find an $e_i\in\leftb L$ such that $\lsp_R(e_i)=c_j$. If $c_j\in L$, then $c_j\in\leftb R$ implies $c_j\in\leftb L$, and we have that $c_j=\lsp_R(e_i)$ with  $e_i:=c_j$. Hence, we can assume that $c_j\notin L$. 
Consider the trajectory $T$ that contains $\inp_0=[x_0,y_0]:=[c_j,c_{j+1}]$. Let $\inp_0$,  $\inp_1=[x_1,y_1]$, $\inp_2=[x_2,y_2]$, \dots, $\inp_s=[x_s,y_s]$ be the edges that constitute $T$ in $R$, listed from left to right. Since $y_s\in \rightb R$, we conclude that $y_s$ is on the right of $\leftb L$; in notation,  $y_s\in \rside_R(\leftb L)$.
Since $y_0\in \leftb R$, $y_0\in\lside_R(\rightb L)$.  Thus, as opposed to $y_s$,   $y_0=c_j\notin\rside_R(\leftb L)$, because otherwise it would belong to $\rside_R(\leftb L)\cap\lside_R(\rightb L) $, which is $L$ by   Lemma~\ref{dkjTghNVc}. Therefore, there exists a unique integer $t\in\set{1,\dots, s}$ such that 
 $y_0, \dots, y_{t-1}$ are strictly on the left of $\leftb L$ but $y_t\in\rside(\leftb L)$. Since $y_0=c_j\in\Jir R$ by \eqref{dkRtghWr}, $T$ departs in upwards direction and $y_0\prec y_1$. None of $y_0, \dots,  y_{t-1}$ belongs to $L $, so none of $y_0, \dots,  y_t$ can have more than two lower covers by Definition~\ref{dfrctext}\eqref{dfrctextc}. Hence, none of $\inp_1, \dots,  \inp_t$ is the top edge of a hat trajectory, and the section of $T$ from $\inp_0$ to $\inp_t$ goes upwards. That is, $T$ goes upwards at $\inp_0,\dots, \inp_t$. Thus,  $y_0\prec y_1\prec\dots\prec y_t$. Applying Lemma~\ref{kellyrivallemma} to the maximal chain $\leftb L$ of $R$ and to the elements 
$y_{t-1}\prec y_t$, we obtain that $y_t\in\leftb L$. Therefore, since $c_j<c_{j+1}=y_0<y_t$ excludes that $y_t=0$, $y_t$ is of the form $y_t=e_{i+1}$ for some $i\in\set{0,1,\dots,n-1}$. Observe that $y_{t-1}\notin L$, $x_t$, and $e_i\in L$ are lower covers of $y_t$. However,  by Definition~\ref{dfrctext}\eqref{dfrctextc}, 
 $y_t$ has at most two lower covers. This implies that $x_t=e_i$, that is, $\inp_t=[x_t,y_t]=[e_i,e_{i+1}]$. 
Since $T$ is the trajectory through $\inp_t$, Lemma~\ref{tChnQlM}  implies that $c_j=x_0=\lsp_R(e_i)$. This proves   \eqref{dlGhBTmnd}.

Next, we claim that, for $x,y\in R$,
\begin{equation}
\parbox{9.0cm}
{if $x\prec y$, $\lefth_R(x)<\lefth_R(y)$, and $\righth_R(x)<\righth_R(y)$, then there are $u,v\in R$ such that $u\prec y$, $v\prec y$, $u\rellambda x$, and $x\rellambda v$.}
\label{hrMlCvr}
\end{equation}
To prove this, assume the first line of \eqref{hrMlCvr}. We conclude from  \eqref{moddkjGHhPq} that $\lsp_R(x)<\lsp_R(y)$ and $\rsp_R(x)<\lsp_R(y)$. 
We have that $c_{\emel}\nless y$, because otherwise  $c_{\emel}\leq x$ by 
\eqref{jvcnTgjYq} and \eqref{upPerChains}, and so $\lefth_R(x)=\emel=\lefth_R(y)$ would contradict our assumption. Similarly, $d_{\emer}\nless y$. 
%
%
First we show that $y\notin \leftb R$ and $y\notin \rightb R$. Suppose, for a contradiction, that $y\in\leftb R$. Then $y\in\llb R$ since $c_{\emel}\nless y$.  We know from \eqref{tpwcgTslm},   \eqref{dkRtghWr}, and Definition~\ref{DFgjRd}\eqref{DFgjRda} that 
\begin{equation}
\text{for all } \pair ij\in\set{0,\dots,\emel}\times\set{0,\dots,\emer},\quad  c_i\wedge d_j=0\text.
\label{dkhNmxWn}
\end{equation}
In particular, $y\wedge d_j$ for all $j\in\set{0,\dots, \emer}$. But $y\neq 0$, so $y\ngeq d_i$ for $i\in\set{1,\dots,\emer}$, and we obtain that $\rsp_R(y)=0$. This contradicts $\rsp_R(x)<\rsp_R(y)$, and we conclude that $y\notin \leftb R$.  Similarly, $y\notin \rightb R$. 
We know from \eqref{dkGHjT}  that $[\lsp_R(y),y]$ is a chain. This chain is nontrivial, because $y\notin \leftb R$. 
Thus, we can pick a unique element $u$ such that $\lsp_R(y)\leq u\prec y$. 
By \eqref{dkGHjT},  $\lsp_R(u)=\lsp_R(y)>\lsp_R(x)$.  We claim that $\rsp_R(u)<\rsp_R(x)$. 
Suppose, for a contradiction, that $\rsp_R(u)\geq \rsp_R(x)$. Combining this inequality with 
$\lsp_R(u)>\lsp_R(x)$ and \eqref{sldiHjNxYb}, we obtain that $x<u$. This is a contradiction since both $x$ and $u$ are lower covers of $y$. Hence,  
\eqref{MsldiHkkQY} applies and  $u\rellambda x$. By left-right symmetry, we also have a $v\in R$ with  $x\rellambda v$. This proves \eqref{hrMlCvr}.

The next step is to show that, for $x\in L$, 
\begin{equation}
\lsp_L(x)= \lsp_L\bigl(\lsp_L(x)\bigr)  \quad\text{and}\quad 
\lsp_R(x)=\lsp_R\bigl(\lsp_L(x)\bigr)\text.
\label{dkhgNTr}
\end{equation}
The first equation is a consequence of \eqref{dkGHjT}.  To prove the second, we can assume that $c_{\emel}\nless x$, because otherwise  $\lsp_R(x)=x=\lsp_L(x)=\lsp_R(\lsp_L(x))$ by \eqref{jvcnTgjYq} and \eqref{upPerChains}. Let $u=\lsp_L(x)$, $v=\lsp_R(u)$, and $w=\lsp_R(x)$. Since $x\geq u\geq v\in\leftb R$, we have $v\leq w$. Applying Lemma~\ref{kellyrivallemma} to $w\leq x$ and the maximal chain $\leftb L$, we obtain an element $t\in \leftb L$ such that $w\leq t\leq x$. By the definition of $v$, we have that $t\leq v$. By transitivity, $w\leq v$, so $v=w$. This proves \eqref{dkhgNTr}.

Now, we are in the position to complete the proof of Lemma~\ref{samecoordin}. 
Let $x\in L$. By left-right symmetry, it suffices to show that $\lefth_R(x)=\lefth_L(x)$. However, by  
\eqref{dkhgNTr}, it is sufficient to show that
\begin{equation}
\text{for}\quad x=e_k\in\leftb L,\quad\lefth_R(e_k)=\lefth_L(e_k)\text.
\label{dljHPTn}
\end{equation}
First, we assume that $c_{\emel}\nless e_k$.
Let $t=\lefth_R(e_k)$; by  \eqref{dkjGHhPq}, this means that $\lsp_R(e_k)=c_t$. Consider the chain
$ H:=\leftb L\cap  \ideal{e_k} =\set{e_k\succ e_{k-1}\succ \dots \succ e_0=0}$.   When we walk down along this chain, at each step from $e_i$ to $e_{i-1}$, \eqref{sldiHjNxYb} yields that at least one of the join-coordinates $\lefth_R(e_i)$ and $\righth_R(e_i)$ decreases. By the definition of $\lefth_L(e_k)$,  it suffices to show that $\lefth_R(e_i)$ decreases iff  $e_i\in\Jir L$, and it can decrease by at most 1.
Therefore, by \eqref{fiTbNWsp}, 
it suffices to show that, for $i\in\set{1,\dots, k}$, 
\begin{align}
&\text{if }\lefth_R(e_i)>\lefth_R(e_{i-1})\text{, then } \lefth_R(e_i)-\lefth_R(e_{i-1})=1\text{, and}\label{dijhGrWa}
\\
&\text{if }\lefth_R(e_i)>\lefth_R(e_{i-1}) \text{, then } \righth_R(e_i)=\righth_R(e_{i-1})\text.\label{dijhGrWb}
\end{align}
Suppose, for a contradiction, that \eqref{dijhGrWb} fails. Since $ \righth_R(e_i)>\righth_R(e_{i-1})$ by \eqref{sldiHjNxYb},
\eqref{hrMlCvr} yields $u,v\in R$ such that $u\prec e_i$, $v\prec e_i$, $u\rellambda e_{i-1}$, and $e_{i-1}\rellambda v$. Since  $e_{i-1}\in\leftb L$, $u$ is strictly on the left of $\leftb L$, and so $u\notin L$. This contradicts Definition~\ref{dfrctext}\eqref{dfrctextc}, proving \eqref{dijhGrWb}.
The proof of \eqref{dijhGrWa} is even shorter. By  \eqref{sldiHjNxYb}, for each $e_j\in\leftb L$,  either $e_j \geq e_i$ and  $\lefth_R(e_j)\geq \lefth_R(e_i)$, or $e_j\leq e_{i-1}$ and
$\lefth_R(e_j)\leq \lefth_R(e_{i-1})$. So if the gap $ \lefth_R(e_i)-\lefth_R(e_{i-1})>1$, then   \eqref{dlGhBTmnd} fails. Hence,   \eqref{dijhGrWa} holds, and so does \eqref{dljHPTn} if $c_{\emel}\nless e_k$.

Second, we assume that $c_{\emel}<e_k$. Since $\filter{c_{\emel}}$ is a chain by \eqref{upPerChains}, there is a unique element $y$ in this chain such that 
$c_{\emel} \leq y\prec e_k$.
Clearly, $\lefth_R(e_k)=\lefth_R(y)=\emel$. 
Let $t$ be the smallest subscript such that $c_{\emel}<e_t$; note that $0<t\leq k$. Since $c_{\emel}$, $e_t$ and $e_{t-1}$ belongs to $S$, which is distributive by Lemma~\ref{SdstrL}\eqref{SdstrLa}, so does $c_{\emel}\wedge e_{t-1}$. By distributivity, $c_{\emel}\wedge e_{t-1} \prec c_{\emel}$. Hence 
\eqref{llsRChains} and \eqref{dkRtghWr} give that $c_{\emel}\wedge e_{t-1}=c_{\emel-1}$. So $c_{\emel-1} \leq  e_{t-1}$. By the definition of $t$, $c_{\emel}\nleq  e_{t-1}$. Hence, $\lefth_R(e_{t-1})=\emel-1$. Since $c_{\emel}\nless e_{t-1}$, \eqref{dljHPTn} is applicable and we have that 
$\lefth_L(e_{t-1})= \lefth_R(e_{t-1})=\emel-1$.
Hence, for the validity of  \eqref{dljHPTn} for $e_k$, we only have to show that $|\set{e_t,\dots,e_k}\cap\Jir L|=1$. This will follow from the following observation:
\begin{equation}
e_t\in\Jir L\text{ but for all }s,\text{ if }t<s\leq k\text{, then }e_s\notin\Jir L\text.
\label{fhgztbvnRwW}
\end{equation}
Note that $t=k$ is possible; if so, then the second part above vacuously holds. Since 
$\lefth_R(y)=\emel>\emel-1=   \lefth_R(e_{t-1})$,  Lemma~\ref{sldiHjNxY} gives that $y\rellambda e_{t-1}$ or $y  > e_{t-1}$. However,  $y  > e_{t-1}$ is impossible since both elements are lower covers of $e_t$. Hence,  $y\rellambda e_{t-1}\in\leftb L$, which implies that $y\notin L$. By Definition~\ref{dfrctext}\eqref{dfrctextc}, $e_t$ has only two lower covers in $R$. Since one of them, $y$, is outside $L$, $e_t\in \Jir L$, as required.  To prove the second half of \eqref{fhgztbvnRwW}, assume that $t<k$ and that $t<s\leq k$. We want to show that $e_s$ is join-reducible in $L$. Since the join-reducibility of $1=e_n$ in $L$ follows prompt from \eqref{dkdgzNmD}, we can assume that $e_s\neq 1$. Since $[e_t,e_k]\subseteq \filter{c_{\emel}}$ and, by \eqref{upPerChains}, $\filter{c_{\emel}}$ is a chain, we obtain that the interval $[e_t,e_k]$ is a chain in $R$. On the other hand,  $\leftb L$ contains $e_t$ and $e_k$, and it is a maximal chain in $R$. Hence, $[e_t,e_k]\subseteq \leftb L$. Since $e_s\notin \Jir R$ by \eqref{dkRtghWr}, $e_s$ has a lower cover $z\in R$ such that $z\neq  e_{s-1}\in \filter{c_{\emel}}\subseteq    \leftb R$. Hence, $e_{s-1}\rellambda z$, and 
the left-right dual of Lemma~\ref{leftrightlemma}\eqref{leftrightlemmaa} gives that $z\in\rside_R(\leftb L)$. 
 We claim that $z\in\lside_R( \rightb L )$ and then, by \eqref{LldkkLgdm}, $z\in L$.  Suppose, for a contradiction, that this is not the case and $z$ is strictly on the right of $\rightb L$. Since $e_s\in \leftb L$ belongs to $\lside_R(\rightb L)$ and $z\prec e_s$, 
Lemma~\ref{kellyrivallemma} yields that $e_s\in\rightb L$. Hence, $e_s=1$ by \eqref{dkdgzNmD}, but this possibility has previously been excluded. This shows that $z\in L$ is another lower cover of $e_s$. Therefore,  $e_s$ is join-reducible in $L$, as required. 
This proves \eqref{dljHPTn} and Lemma~\ref{samecoordin}.
\end{proof}

Next, we still assume \eqref{ssmfTsCtwsta}, \eqref{ssmfTsCtwstb}, and that $L$ is slim. We define the following sets of coordinate pairs; the acronyms come from ``Internal'', ``Left'', ``Right'', and ``All'' Coordinate Pairs,  respectively. 
\begin{align*}
\incoord L&:= \set{\pair{\lefth_L (x)}{\righth_L(x)}: x\in L },\cr
\incoord R&:= \set{\pair{\lefth_R (x)}{\righth_R(x)}: x\in L },\cr
\leftcoord R&:= \set{\pair{\lefth_R (x)}{\righth_R(x)}: x\in R\text{ is strictly on the left of  }\leftb L},\cr
\rightcoord R&:= \set{\pair{\lefth_R (x)}{\righth_R(x)}: x\in R\text{ is strictly on the right of  }\rightb L},\cr
\allcoord R&:= \incoord R \cup  \leftcoord R \cup \rightcoord R
\text.
\end{align*}
We know from \eqref{sldiHjNxYa} and \eqref{sldiHjNxYb} that  
\begin{equation}
\text
{these sets describe $R$ and, in an appropriate sense, its diagram.}
\label{dhdttjrjkL}
\end{equation}
As an important step towards the uniqueness of $R$, the following lemma states that these sets do not depend on $R$. As the case $L$ is a chain witnesses, the following lemma would fail without assuming \eqref{ssmfTsCtwstb}.

\begin{lemma}\label{outlemmA}
Assume \eqref{ssmfTsCtwsta}, \eqref{ssmfTsCtwstb},   and that $L$ is slim.
With the notation given in Definition~\ref{DFgjRd} and $\nggy:=\set{0,\dots,\emel}\times\set{0,\dots,\emer}$, the following 
hold.
{\allowdisplaybreaks{
\begin{align}
&\emel=\max\set{\lefth_L(x): x\in L}, \quad \emer=\max\set{\righth_L(x): x\in L},
\label{jGjdTbZa}\\
&\incoord R=\incoord L,
\label{jGjdTbZb}\\
&\begin{aligned}
\leftcoord R=\set{&\pair ij: \pair ij\in\nggy\setminus \incoord L\text{ and } \exists x\in\leftb L \cr   
& \text{such that }i > \lefth_L(x)\text{ and }  j=\righth_L(x)},
\end{aligned} 
\label{jGjdTbZc}\\
&\begin{aligned}
\rightcoord R=\set{&\pair ij: \pair ij\in \nggy\setminus \incoord L\text{ and } \exists x\in\rightb L \cr   
& \text{such that }j > \righth_L(x)\text{ and }  i=\lefth_L(x)}\text.
\end{aligned}
\label{jGjdTbZd}
\end{align}
}%
}%
Also, $\leftcoord R$ and $\rightcoord R$ are given in terms of $\incoord L$ as follows:
{\allowdisplaybreaks{
\begin{align}
&\begin{aligned}
\leftcoord R =\set{\pair ij\in\nggy\setminus \incoord L:\exists i'\text{ such that }
\pair{i'}{j}\in\incoord L,\phantom{\}}  \cr
 i'<i \text{, and for every }  \pair{i''}{j''}\in\incoord L,\text{ }  i''> i' \then j'' \geq{} j},
\end{aligned} 
\label{jGjdTbZe}\\
&\begin{aligned}
\leftcoord R =\set{\pair ij\in\nggy\setminus \incoord L:\exists j'\text{ such that }
\pair{i}{j'}\in\incoord L,\phantom{\}}   \cr
 j'<j \text{, and for every }  \pair{i''}{j''}\in\incoord L,\text{ }  j''> j' \then i'' \geq{} i}\text.
\end{aligned} 
\label{jGjdTbZf}
\end{align}
}%
}%
The componentwise ordering, $\pair{i_1}{j_1} \leq \pair{i_2}{j_2}$ iff $i_1\leq i_2$ and $j_1\leq j_2$, turns 
$\allcoord R = \incoord R \cup  \leftcoord R \cup \rightcoord R$ into a lattice, which depends only on the fixed diagram of $L$. 
Actually, 
\begin{equation}
\allcoord R\text{ only depends on }\incoord L\text.
\label{dksdghNcX}
\end{equation}
Furthermore, the \emph{``coordinatization maps''}  $\gamma\colon L\to \incoord L$ defined by $x\mapsto \pair{\lefth_L(x)}{\righth_L(x)}$ and $\delta\colon R\to \allcoord R$ defined by $x\mapsto \pair{\lefth_R(x)}{\righth_R(x)}$ are lattice isomorphisms.
\end{lemma}

\begin{proof}
Using Lemma~\ref{samecoordin} and $1=1_R\in L$, we obtain that $\max\set{\lefth_L(x): x\in L}=\lefth_L(1)=\lefth_R(1)=\emel$. The other half of  \eqref{jGjdTbZa} follows similarly. 
\eqref{jGjdTbZb}  also follows from Lemma~\ref{samecoordin}. 
In the rest of the proof, \eqref{jGjdTbZb}  and Lemma~\ref{samecoordin} allow  us to write $\incoord R$,    $\lefth_R$ and $\righth_R$ instead of  $\incoord L$, $\lefth_L$ and $\righth_L$, and vice versa,   respectively,   without further  warning.

Assume that $\pair ij\in \leftcoord R$, and that $\pair ij=\pair{\lefth_R (y)}{\righth_R(y)}$ for some $y\in R$ strictly on the left of $\leftb L$. Applying Lemma~\ref{kellyrivallemma} to $\rsp_R(y)\leq y$ and $\leftb L$, we obtain an element $x\in\leftb L\cap [\rsp_R(y),  y]$.  By \eqref{dkjGHhPqvar}, $\righth_R(x)=\righth_R(y)=j$. We know that $x\neq y$, because $x\in L$ but $y\notin L$. Hence, $x<y$, and \eqref{sldiHjNxYb} gives that $i=\lefth_R(y)>\lefth_R(x)$. This proves the ``$\subseteq$'' part of \eqref{jGjdTbZc}.

In order to prove the converse inclusion, assume that $x\in \leftb L$, $\pair ij \in \nggy$, $\pair ij\notin\incoord R=\incoord L$, $i> \lefth_R(x)$, and $j=\righth_R(x)$. Let $k=\lefth_R(x)$. 
In the distributive lattice $S$ from Lemma~\ref{SdstrL}, let $y=c_i\vee x\in S$.
Next, we show that
\begin{equation}
\lefth_R(y)=i  \text{ and, if } i<\emel, \text{ } \lsp_R(y)=c_i\text.
\label{mkrvNd}
\end{equation}
Since \eqref{mkrvNd} is obvious if $i=\emel$, we can assume that $i<\emel$. Since $c_i\leq y$, we  obtain that $\lsp_R(y)\geq c_i$. Suppose, for a contradiction, that $\lsp_R(y) >  c_i$. Then  $c_{i+1}\leq y$. Since 
$c_{i+1}$ is join-irreducible in $R$ by \eqref{dkRtghWr} and $c_{i+1}\neq 0_S=0_R$, we have that $c_{i+1}\in\Jir S$.  Using distributivity in the standard way as in \eqref{dstrSng} and taking $c_{i+1}\nleq c_i$ and $c_{i+1}\leq y=c_i\vee x$ into account, we obtain that $c_{i+1}\leq x$. Hence, $k=\lefth_R(x) \geq i+1$, either because  $i+1\leq\emel$ and $x>c_{\emel}$, or because  $x \not>c_{\emel}$ and \eqref{dkjGHhPq} applies. This contradicts $i>k$,  proving the second equation in \eqref{mkrvNd}. The first equation in \eqref{mkrvNd} follows from \eqref{dkjGHhPq}.

Observe that, for every $z\in R$, 
\begin{equation}
\text{the intervals }
[c_{\lefth_R(z)},z]\text{ and } [d_{\righth_R(z)},z]\text{ are chains.}
\label{stLRchains}
\end{equation}
If $\lefth_R(z)=\emel$, then $z\geq c_{\emel}=\cornl R$, and the first interval  is a chain by  \eqref{upPerChains}. If $\lefth_R(z)\neq\emel$, then  $z\not>c_{\emel}$, and the first interval  is a chain by  \eqref{dkGHjT} and \eqref{dkjGHhPq}. Similarly, the second interval is also a chain, proving  \eqref{stLRchains}.

Next, we prove that 
\begin{equation}
\righth_R(y) = j\text.
\label{mkrJdBr}
\end{equation}
If $\righth_R(x)=j=\emer$, then $y\geq x\geq d_{\emer}$ yields that $j=\emer=\righth_R(y)$ as required. Hence, we can assume that $j<\emer$.
From $y\geq x$ and \eqref{sldiHjNxYb}, we obtain that $\righth_R(y)\geq j$. Suppose, for a contradiction, that $t:=\righth_R(y) >  j$. 
By \eqref{stLRchains}, we can let
$[d_t,y]=\set{y_0:=y\succ y_1\succ \dots \succ y_s=d_t}$. Since $y=y_0\in S$, there is a largest number $q\in\set{0,\dots,s}$ such that 
$\set{y_0,\dots,y_q}\subseteq S$. 
The situation is roughly visualized in Figure~\ref{fig-prfDtls1}, where only a part of $R$ is depicted and the black-filled elements belong to $\leftb L$. (Note, however, that a targeted contradiction cannot be satisfactorily depicted.)

We claim that  $q<s$. Suppose, for a new  contradiction, that $q=s$. Then $y_s=d_t\in S\cap\rightb R$. So $d_t$ is on the left of $\leftb L$ and, belonging to $\rightb R$, it is also on the right of $\rightb L$.   If $d_t\notin \leftb L$, then we can pick an $e\in\leftb L$ with $d_t\parallel e$, because $\leftb L$ is a maximal chain. By 
Lemma~\ref{leftrightlemma}\eqref{leftrightlemmaa}, $d_t\rellambda e$, that is, $e$ is on the right of $d_t$. By the left-right dual of Lemma~\ref{leftrightlemma}\eqref{leftrightlemmaa}, $e$ is on the right of $\rightb R$. But no element of $R$ can be strictly on the right of $\rightb R$, so $e\in \rightb R$. This is a contradiction, because $d_t\parallel e$,  $d_t$ also belongs to $\rightb R$, and $\rightb R$ is a chain. This shows that $d_t\in\leftb L$. A similar argument gives that  
$d_t\in \rightb L$. That is, $d_t\in \leftb L\cap \rightb L$, and the glued sum indecomposability of $L$ implies that $d_t\in\set{0,1}$. The inequality $t>j$ excludes that $d_t=0$, so $d_t=1_L=1_R$. Hence,  \eqref{dkRtghWr} and
\eqref{dkhNmxWn} yield that $R$ is a chain, which contradicts its rectangularity.
Thus, $q<s$.

\begin{figure}[htb]
\includegraphics[scale=1.0]{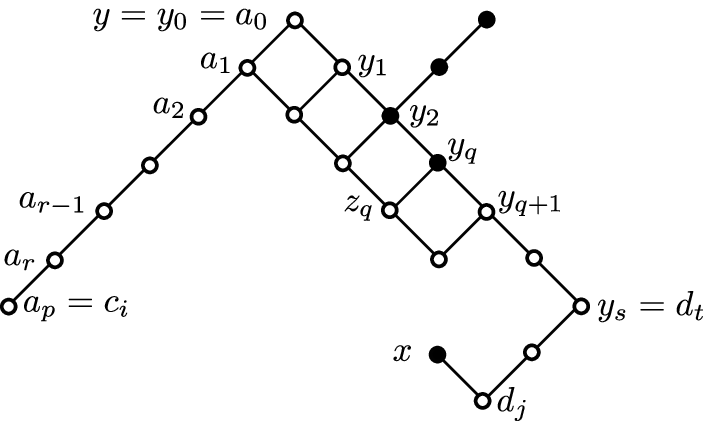}
\caption{If \eqref{mkrJdBr} fails}\label{fig-prfDtls1}
\end{figure}

Note that $q<s$ implies that   $s\geq 1$ and that    $y_1$ and $y_{q+1}$ will make sense  later. Since $\lefth_R(y)=i$ by \eqref{mkrvNd}, \eqref{stLRchains} yields that $[c_i,y]$ is a chain. Since  $y\geq d_t$ but, by  \eqref{dkhNmxWn}, $c_i\ngeq d_t$, we obtain that  
 $[c_i,y]$ is a nontrivial chain. Denote its element as follows:
\[[c_i,y] = \set{y=a_0\succ a_1\succ\dots\succ a_p=c_i},
\]
where $p\geq 1$.  By \eqref{dkjGHhPqvar}, $\lefth_R(a_1)=\lefth_R(y)=i$. By \eqref{djdzBgRWv},  $y=c_i\vee d_t$. Since $y=c_i\vee d_t\leq a_1\vee y_1\leq y$, we obtain that $a_1\neq y_1$. Thus, as two distinct lower covers of $y$, $a_1$ and $y_1$ are incomparable.  Observe that $c_i\leq y_1$ is impossible because otherwise $y=c_i\vee d_t\leq y_1 < y$. Hence, $\lefth_R(y_1)<i=\lefth_R(a_1)$.  Combining this inequality with $a_1\parallel y_1$ and \eqref{sldiHjNxYa} and using Lemma~\ref{leftrightlemma}\eqref{leftrightlemmab}, we obtain that $a_1\rellambda y_1$.

Next, we assert that $a_1\in S$. Suppose, for a new contradiction, that $a_1\notin S$. 
This means that $a_1$ is strictly on the right of $\leftb L$. Since $a_1\prec y$, Lemma~\ref{kellyrivallemma} excludes that $y$ is strictly on the left of $\leftb L$. However, $y\in S$ is on the left of $\leftb L$, so $y\in\leftb L$.   We know that $a_p=c_i\in\leftb R$ belongs to $S$, whence there exists a smallest $r\in\set{2,\dots,r}$ such that $\set{a_p, a_{p-1},\dots, a_r}\subseteq S$.  Since $a_{r-1}$ is not in $S$, it strictly is on the right of $\leftb L$.  But $a_r$ is on the left of $\leftb L$ and $a_r\prec a_{r-1}$. Lemma~\ref{kellyrivallemma} implies easily that $a_r\in\leftb L$. The interval $[a_r, a_0]$ is a chain since so is $[a_p,a_0]=[c_i,y]$ by 
\eqref{stLRchains}. Since $a_r,a_0\in \leftb L$ and $\leftb L\cap  [a_r, a_0]$ is a maximal chain in the interval $[a_r, a_0]=\set{a_r\prec \dots\prec a_1\prec a_0}$, it follows that $\set{a_r, \dots, a_1,a_0}\subseteq \leftb L\subseteq S$. This contradicts   $a_1\notin S$. Thus, $a_1\in S$. 

Since $S$ is meet-closed, $z_q:=a_1\wedge y_q\in S$. 
The distributivity of $S$, see Lemma~\ref{SdstrL}, yields that $z_q\preceq y_q$.
If $z_q=y_q$, then we have $y>a_1 =  a_1\vee z_q = a_1\vee y_q\geq c_i\vee d_t$, which is a contradiction since  $y=c_i\vee d_t$ by  \eqref{djdzBgRWv}. Hence,  
$z_q\prec y_q$.  We also know that $y_{q+1}\prec y_q$. By the choice of $q$, $y_{q+1}\notin S$, so $y_{q+1}$ is strictly on the right of $\leftb L$. But the element $y_q\in S$
is on the left of $\leftb L$, and we conclude from Lemma~\ref{kellyrivallemma}  easily again that $y_q\in \leftb L$. Since $q<s$ and $y_q>y_s$, we know that $y_q\neq 0$. Therefore, $\leftb L$ contains a unique element $e$ such that $e\prec y_q$. Since $y_{q+1}\notin S$, we obtain that $e\neq y_{q+1}\neq z_q$. 
Suppose, for a new contradiction, that $e=z_q$. Since $x$ and $z_q=e$ both belong to $\leftb L$, they are comparable. If $x> z_q$, then $x\geq y_q\in \leftb L$, so \eqref{dkjGHhPqvar} and \eqref{sldiHjNxYb} imply that $j=\righth_R(x)\geq \righth_R(y_q)=\righth_R(y)=t$, which is a contradiction excluding that $e=z_q$. 
Consequently, the set $\set{z_q,e,y_{q+1}}$, which consists of  distinct lower covers of $y_q$,   is a three-element antichain.  
Hence, as opposed to $e$, $z_q$ does not belong to the chain $\leftb L$. If $z_q\rellambda e$, then $z_q$ is strictly on the left of $\leftb L$, so $z_q\notin L$, which contradicts Definition~\ref{dfrctext}\eqref{dfrctextc}. Hence, by Lemma~\ref{leftrightlemma}\eqref{leftrightlemmab}, $e\rellambda z_q$. However, then  $z_q$ is strictly on the right of $\leftb L$ by the left-right dual of Lemma~\ref{leftrightlemma}\eqref{leftrightlemmaa}, which contradicts $z_q\in S$. That is, $e\neq z_q$ also leads to a contradiction. This proves \eqref{mkrJdBr}.

It follows from  \eqref{mkrvNd} and \eqref{mkrJdBr} that $\pair{\lefth_R(y)}{\righth_R(y)}=\pair ij\notin \incoord L$. Hence, $y\notin L$, which gives that $y\notin\leftb L$. On the other hand, $y\in S$ yields that $y$ is strictly on the left of $\leftb L$. Therefore, $\pair ij = \pair{\lefth_R(y)}{\righth_R(y)} \in   \leftcoord R$.
This implies the ``$\supseteq$'' part of \eqref{jGjdTbZc}.  Thus, \eqref{jGjdTbZc} holds, and so does \eqref{jGjdTbZd} by left-right symmetry.

Next, we deal with \eqref{jGjdTbZe}. 
The pair $\pair{i'}{j}$ in \eqref{jGjdTbZe} corresponds to the coordinate pair $\pair{\lefth_L(x)}{\righth_L(x)}$ for some element $x\in L$.   By \eqref{sldiHjNxYa}, the condition that for every $ \pair{i''}{j''}\in\incoord L$, $  i''> i' \then j'' \geq{} j$ says that no element of $L$ is to the left of $x$, that is, 
 this  $x$ belongs to $\leftb L$.  Therefore, the right-hand side of the equation in \eqref{jGjdTbZe} is the same as that in \eqref{jGjdTbZc}. Hence,  \eqref{jGjdTbZe} follows from \eqref{jGjdTbZc}. Similarly, \eqref{jGjdTbZd} implies \eqref{jGjdTbZf}. In this way, we have proved the equations  \eqref{jGjdTbZa}--\eqref{jGjdTbZf}.

It follows from  \eqref{jGjdTbZb},  \eqref{jGjdTbZe},  and \eqref{jGjdTbZf} that $\allcoord R$ depends only on the fixed diagram of $L$, and it only depends on $\incoord L$.  Finally, Lemma~\ref{sldiHjNxY} and \eqref{sldiHjNxYb} imply that $\gamma$ and $\delta$ are isomorphisms. This completes the proof of  Lemma~\ref{outlemmA}.
\end{proof}

Let $L$ be a planar semimodular lattice. According to  \init{G.\ }Gr\"atzer and \init{E.\ }Knapp~\cite{gratzerknapp1}, a \emph{full slimming} (sublattice) $L'$ of $L$ is obtained from a planar diagram of $L$ by omitting all elements from the interiors of intervals of length 2 as long as there are elements to omit in this way. Note that $L'$, as a subset of $L$, is not unique; this is witnessed by $L=M_3$. However, the full slimming sublattice becomes unique if the planar diagram of $L$ is fixed.
 In \cite{gratzerknapp1}, the elements we omit are called ``eyes''. Note that $L'$ is a slim semimodular lattice.
Note also that when we omit an eye from the lattice, then we also omit this eye (which is a doubly irreducible element)  from the diagram with the two edges from the eye. The converse procedure, when we put the omitted elements back, is called an \emph{anti-slimming}. 
An element $x\in L$ is \emph{reducible} if it is join-reducible or meet-reducible, that is, if $x\in (L\setminus \Jir L)\cup (L\setminus \Mir L)$. In other words, if  it is not  doubly irreducible.
It follows obviously from the slimming procedure that if $L'$ is a full slimming sublattice of $L$, then 
\begin{equation}
L' \text{ contains every reducible element of }L\text.
\label{nrxrDmrD}
\end{equation}
Although we know from \init{G.\,}Cz\'edli and \init{E.\,T.\ }Schmidt~\cite[Lemma 4.1]{czgschslim2} that $L$ determines $L'$ up to isomorphisms, we need a stronger statement here. 
By \init{G.\ }Gr\"atzer and \init{E.\ }Knapp~\cite[Lemma 8]{gratzerknapp1}, an element in a slim semimodular lattice  can have at most two covers. Therefore, every 4-cell can be described by its bottom element. To capture the situation that $L'$ is a full slimming (sublattice) of a planar semimodular lattice $L$, we define \emph{the numerical companion map} $\unu=\inu L{L'}$ associated with the full slimming sublattice $L'$ as follows. It is the map $\inu L{L'}\colon L'\to  \mathbb N_0:=\set{0,1,2,\dots}$ defined by
\begin{equation}
\inu L{L'}(x)= \begin{cases}
n,&\text{if } x\text{ is the bottom of a 4-cell that has }n\text{ eyes in }L,\cr
0,&\text{otherwise.}
\label{slimmingMAP}
\end{cases}
\end{equation}
Let $L_i'$ be a full slimming sublattice of a planar semimodular  lattice $L_i$, for $i\in\set{1,2}$, and let $\phi\colon L_1'\to L_2'$ be an isomorphism. We say that $\phi$ is an \emph{$\unu$-preserving isomorphism} if $\inu{L_1}{L_1'}= \inu{L_2}{L_2'}\circ\phi$. (We compose maps from right to left.) 
The map $\inu L{L'}$ exactly describes how to get $L$ back from $L'$ by anti-slimming.  Hence,  obviously, 
\begin{equation}
\parbox{6.5cm}{every $\unu$-preserving $L_1'\to L_2'$ isomorphism extends to an $L_1\to L_2$  isomorphism.}
\label{sdTZnDsghpq}
\end{equation}
The restriction of a map $\kappa$ to a set $A$ is denoted by $\restrict\kappa A$.

\begin{lemma} \label{slimcompmml} For $i\in\set{1,2}$, let $L_i'$ be a full slimming sublattice of a planar semimodular lattice $L_i$.
\begin{enumeratei}
\item\label{slimcompmmla} $L_1$ is glued sum indecomposable iff so is $L_1'$. 
\item\label{slimcompmmlb} $L_1$ is rectangular iff so is $L_1'$. $($This is  \init{G.\ }Cz\'edli~\cite[(6.4)]{czgrepres}.$)$
\item\label{slimcompmmlc} If $\phi\colon L_1\to L_2$ is an isomorphism, then there exists an automorphism $\pi$ of $L_1$ such that the restriction $\restrict{(\phi\circ\pi)}{L_1'}$ is a $\unu$-preserving $L_1'\to L_2'$ isomorphism and, in addition, $\pi(x)=x$  for every reducible  $x\in L_1$.
\item\label{slimcompmmld} Any two full slimming sublattices of a planar semimodular lattice are isomorphic. 
\end{enumeratei}
\end{lemma}

\begin{proof}  To prove part \eqref{slimcompmmla}, assume that $L_1$ is glued sum indecomposable and that $x\in L_1'\setminus\set{0,1}$. There is an element $y\in L_1$ such that $x\parallel y$. We can assume that $y\notin L_1'$, since otherwise there is nothing to do. Then $y$ is an ``eye'', so there are $a,b\in L_1'$ such that $\set{a\wedge b,a,b,a\vee b}$ is a covering square in $L_1$ and $y\in[a\wedge b,a\vee b]$ is to the left of $a$ and to the right of $b$. If $x\nparallel a$ and $x\nparallel b$, then either $x\leq a\wedge b\leq y$, or $x\geq a\vee b\geq y$, because the rest of cases would contradict $a\parallel b$. But this contradicts $x\parallel y$, proving that $L_1'$ is slim. The converse direction is trivial, because if $L_1'$ is glued sum indecomposable, then its elements outside $\set{0,1}$ are incomparable with appropriate elements of $L_1'$ while the eyes are incomparable with the corners of the covering squares they were removed from to obtain $L_1'$. This proves \eqref{slimcompmmla}.

Part \eqref{slimcompmmlb} has already been proved in \init{G.\ }Cz\'edli~\cite{czgrepres}.

We assume that, for $i\in\set{1,2}$, a planar diagram of $L_i$ is fixed and that we form the  full slimming sublattice $L_i'$ according to this diagram. 
We prove \eqref{slimcompmmlc} by induction on $|L_1|$. If $L_1$ is slim, then the statement is trivial, because $L_1'=L_1$, $L_2'=L_2$, $\pi$ is the identity map on $L_1$, and both numerical companion maps are the constant zero maps. 
 Assume that $L_1$ is not slim. Then there are $a_1 < b_1 \in L_1$ with images $a_2=\phi(a_1)$ and $b_2=\phi(b_1)$ such that,
for $i\in\set{1,2}$,  $[a_i, b_i]$ is
an interval of length two and it contains a doubly irreducible element $x_i$ in its interior such that $x_i\notin L_i'$.  Let $y_1=\phi^{-1}(x_2)$; it is a doubly irreducible element of $L_1$ in  $[a_1,b_1]$. 
Clearly, there is an automorphism $\pi_0$ of $L_1$ such that $\pi_0(x_1)=y_1$, $\pi_0(y_1)=x_1$, and $\pi_0(z)=z$ for $z\notin\set{x_1,y_1}$. As we require in case of our automorphisms, every reducible element is a fixed point of   $\pi_0$.

Observe that $(\phi\circ\pi_0)(x_1)= \phi(\pi_0(x_1))= \phi(y_1)=x_2$. Since $x_i$ is doubly irreducible, $L_i^\ast:=L_i\setminus\set{x_i}$ is a sublattice of $L_i$ and  $\phi\ast:=\restrict{ (\phi\circ\pi_0) }{L_1^\ast}$ is an $L_1^\ast\to L_2^\ast$ isomorphism. Note that $L_i'$ is also a full slimming sublattice of $L_i^\ast$. By the induction hypothesis, $L_1^\ast$ has an automorphism $\pi^\ast$ such that  $\restrict{(\phi^\ast\circ \pi^\ast)}{L_1'}$ is an $\unu$-preserving $L_1'\to L_2'$ isomorphism and, in addition, $\pi^\ast(z)=z$ for every reducible element  $z$ of $L_1^\ast $. 
In particular,
\begin{equation}
\inu{L_2^\ast}{L_2'} \circ  (\restrict{(\phi^\ast\circ \pi^\ast)}{L_1'}) =  \inu{L_1^\ast}{L_1'} \text.
\label{zldhl}
\end{equation}
Let $\nszgpi\colon L_1\to L_1$ the only automorphism that extends $\pi^\ast$. That is, $\nszgpi(x_1)=x_1$ and, for $z\neq x_1$, $\nszgpi(z)=\pi^\ast(z)$. We define $\pi:=\pi_0\circ\nszgpi$, and we claim that it has the properties required in Lemma~\ref{slimcompmml}\eqref{slimcompmmlc}. 
If $z$ is a reducible element of $L_1$, then $z\notin\set{ x_1,y_1}$, since $x_1$ and $y_1$ are doubly irreducible. Hence, $z\in  L_1^\ast$. Furthermore, $z$ is also reducible in  $L_1^\ast$, because it is only $a_1$ and $b_1\in L_1$ that loose one of their  upper and lower covers, respectively, when passing from $L_1$ to $L_1^\ast$, but they  still have at least two upper and lower covers, respectively, in $L_1^\ast$. Hence, $z$ is a fixed point of $\pi^\ast$ by the induction hypothesis, and we obtain that $
\pi(z)= (\pi_0\circ\nszgpi)(z) = \pi_0(\nszgpi(z)) = 
\pi_0(\pi^\ast(z))=\pi_0(z)=z$, as Lemma~\ref{slimcompmml}\eqref{slimcompmmlc} requires. Next, we show  that
\begin{equation}
\restrict{(\phi\circ\pi)}{L_1'} =
\restrict{(\phi^\ast \circ \pi^\ast  )}{L_1'},   
\label{zhlglpnkSt}
\end{equation}
Let $z\in L_1'$. Since $x_1\notin L_1'$, $z\neq x_1$. We compute as follows.
\begin{align*}
(\phi\circ\pi)(z)&= (\phi\circ\pi_0\circ\nszgpi)(z)=  (\phi\circ\pi_0)(\nszgpi(z))=  (\phi\circ\pi_0)(\pi^\ast(z))\cr
&= \restrict{(\phi\circ\pi_0)}{L_1^\ast}(\pi^\ast(z)) = \phi^\ast(\pi^\ast(z)) = (\phi^\ast\circ \pi^\ast)(z)\text.
\end{align*}
This proves \eqref{zhlglpnkSt}. In particular, this also gives that $\restrict{(\phi\circ\pi)}{L_1'}$ is an isomorphism from $L_1'$ to $L_2'$. We have to prove that it is $\unu$-preserving, that is,
\begin{equation}
\inu{L_2}{L_2'} \circ  (\restrict{(\phi \circ \pi )}{L_1'}) \overset ?=  \inu{L_1}{L_1'} \text.
\label{prscrssq}
\end{equation}
Before proving \eqref{prscrssq}, observe that, for $z\in L_i'$ and $i\in\set{1,2}$, 
\begin{equation}
\inu{L_i}{L_i'}(z)=
\begin{cases}
\inu{L_i^\ast}{L_i'}(z),&\text{ if }z\neq a_i.\cr
1+\inu{L_i^\ast}{L_i'}(z),&\text{ if }z = a_i\text.
\end{cases}
\label{ktmnTNfqC}
\end{equation}
Hence $z=a_1$, which is in $L_1'$ by \eqref{nrxrDmrD}, and $z\in L_1'\setminus\set{a_1}$ need separate treatments. 
First, since $a_1$ is reducible and   $\pi^\ast$, $\pi_0$, and $\pi$ keep it fixed,
\begin{align}
&
(\restrict{(\phi \circ \pi )}{L_1'})(a_1)=(\phi \circ \pi ) (a_1) 
 = \phi ( \pi( a_1) )= \phi ( a_1)=a_2, 
\label{mRtwDsrpa} \\
&\begin{aligned}
&(\restrict{(\phi^\ast  \circ \pi^\ast  )}{L_1'})(a_1)=(\phi^\ast  \circ \pi^\ast  ) (a_1) 
 = \phi^\ast  ( \pi^\ast ( a_1) )
= \phi^\ast  ( a_1) \cr 
&\kern 16.5 pt =\restrict{(\phi\circ\pi_0)}{L_1^\ast}(a_1)
=(\phi\circ\pi_0)(a_1)=\phi(\pi_0(a_1))=\phi(a_1) =a_2\text.
\end{aligned}
\label{mRtwDsrpb}
\end{align}
Hence, we can compute as follows.
\begin{align*}
\bigl(\inu{L_2}{L_2'} &\circ  (\restrict{(\phi \circ \pi )}{L_1'})\bigr) (a_1) = 
\inu{L_2}{L_2'} \bigl(  (\restrict{(\phi \circ \pi)}{L_1'})(a_1)  \bigr) 
\overset{\eqref{mRtwDsrpa} }=
\inu{L_2}{L_2'}( a_2) \cr
&\overset{\eqref {ktmnTNfqC} }=
1+\inu{L_2^\ast}{L_2'}( a_2) 
\overset{\eqref {mRtwDsrpb} }=
1+\inu{L_2^\ast}{L_2'}\bigl( (\restrict{(\phi^\ast  \circ \pi^\ast  )}{L_1'}) (a_1)  \bigr) \cr
&\overset{\eqref {zldhl} }=
1+\inu{L_1^\ast}{L_1'}(a_1) 
\overset{\eqref {ktmnTNfqC} }= \inu{L_1}{L_1'}(a_1) \text.
\end{align*}
This shows that \eqref{prscrssq} 
 holds for the element $a_1$. Second, 
assume that $z\in L_1'\setminus\set{a_1}$. Since the map in \eqref{mRtwDsrpb} is a bijection, $(\restrict{(\phi^\ast  \circ \pi^\ast )}{L_1'})(z) \neq a_2$, and we can compute as follows.
\begin{align*}
\bigl(\inu{L_2}{L_2'} &\circ  (\restrict{(\phi \circ \pi )}{L_1'})\bigr) (z) 
\overset{\eqref  {zhlglpnkSt}  }= 
\inu{L_2}{L_2'} \bigl(  (\restrict{(\phi^\ast  \circ \pi^\ast )}{L_1'})(z)  \bigr) 
\cr
&
\overset{\eqref {ktmnTNfqC} }=
\inu{L_2^\ast}{L_2'} \bigl(  (\restrict{(\phi^\ast  \circ \pi^\ast )}{L_1'})(z)  \bigr) 
\overset{\eqref {zldhl} }=
\inu{L_1^\ast}{L_1'} (z)
\overset{\eqref {ktmnTNfqC} }=
\inu{L_1}{L_1'} (z)
\text.
\end{align*}
Therefore, \eqref{prscrssq} holds and $\restrict{(\phi\circ\pi)}{L_1'}$ is $\unu$-preserving. This  completes  the proof of Lemma~\ref{slimcompmml}. 
\end{proof}

\begin{definition}[{\init{D.\ }Kelly and \init{I.\ }Rival ~\cite[p.~ 640]{kellyrival}}]
\label{dkrmMpdF}
For planar lattice diagrams $D_1$ and $D_2$, a 
bijection $\phi\colon D_1\to D_2$ is a \emph{similarity map} if it is a lattice isomorphism and, for all $x,y,z\in D_1$ with $y\prec x$ and $z\prec x$, $y$ is to the left of $z$ iff $\phi(y)$ is to the left of $\phi(z)$. If there is such a map, then $D_1$ is \emph{similar} to~$D_2$.
\end{definition}

Note that similarity turns out to be a self-dual condition; see  \init{G.\ }Cz\'edli and \init{G.\ }Gr\"atzer~\cite[Exercise 3.9]{czgggltsta}. Furthermore, if $D_1$ and $D_2$ are planar diagrams of slim (but not necessarily semimodular) lattices and a bijective map 
 $\phi\colon D_1\to D_2$ is a lattice isomorphism, then  
\begin{equation}
\text{$\phi$ is a similarity map iff it preserves the left boundary chain,}
\label{kkfpPjRw}
\end{equation}
that is, $\phi(\leftb{D_1})=\leftb{D_2}$; see \cite[Theorem 3-4.6]{czgggltsta}.  
A map between two lattices can be considered as a map between (the vertex sets) of their diagrams.  For a diagram $D$, its mirror image over a vertical axis is denoted by $\refl D$. We say that the planar diagrams of a planar lattice $L$ are unique up to \emph{left-right similarity} if for any two diagrams $D_1$ and $D_2$ of $L$, $D_1$ is similar to $D_1$ or it is similar to $\refl{D_1}$. For a statement stronger than the following one, see  \init{G.\ }Cz\'edli and \init{G.\ }Gr\"atzer~\cite[Theorem 3-4.5]{czgggltsta}.

\begin{proposition}[{\init{G.\,}Cz\'edli and \init{E.\,T.\ }Schmidt~\cite[Lemma 4.7]{czgschslim2}}]\label{slmUq} 
Assume that $L_1$ and $L_2$ are glued sum indecomposable slim semimodular lattices with planar diagrams $D_1$ and $D_2$, respectively. If $\phi\colon L_1\to L_2$ is a lattice isomorphism, then $\phi\colon D_1\to D_2$ or $\phi\colon D_1\to \refl{D_2}$ is a similarity map. Consequently, the planar diagrams of a glued sum indecomposable slim lattice are unique up to  left-right similarity.
\end{proposition}

Now, we are in the position to complete the present section briefly.

\begin{proof}[Proof of Theorem~\ref{thmmain}]
 Let $L$ be a planar semimodular lattice. 
The existence of a \rectext{} $R$ of $L$ follows from  \init{G.\ }Gr\"atzer and \init{E.\ }Knapp~\cite[Proof of Theorem 7]{gratzerknapp3}, and it also follows from 
 \init{G.\ }Cz\'edli~\cite[Lemma 6.4]{czgrepres}. Thus, part \eqref{thmmaina} of the theorem holds.
 
To prove part \eqref{thmmainb}, let $R$ be an arbitrary \rectext{} of $L$. 
Based on \eqref{slimiffnodiamond}, it suffices to show that $R$ has a cover-preserving diamond sublattice $M_3$ iff so has $L$.  The ``if'' part is evident since $L$ is a cover-preserving sublattice of $L$. Conversely, assume that $M_3$ is a cover-preserving sublattice of $R$. It follows from Definition~\ref{dfrctext}\eqref{dfrctextc} that none of its three atoms is in $R\setminus L$. Hence, all atoms of $M_3$ belong to $L$. Since $M_3$ is generated by its atoms, $M_3$ is a cover-preserving sublattice of $L$. This proves part \eqref{thmmainb} of the theorem.

To prove \eqref{thmmaine}, assume that $L$ and $L'$  are glued sum indecomposable slim planar semimodular lattices with fixed planar diagrams and that $\psi\colon L\to L'$ is an isomorphism. By reflecting one of the diagrams over a vertical axis if necessary, 
Proposition~\ref{slmUq} allows us to assume that $\psi$ is a
similarity map between the respective diagrams. Hence, $\psi(\leftb{L})= \leftb{L'}$ and $\psi(\rightb{L})= \rightb{L'}$. Also, $\emel=\emel'$ and $\emer=\emer'$.
We know from (the last sentence of) Lemma~\ref{outlemmA}
that  $\incoord L$ and $\vincoord L$ are lattices with respect to the componentwise ordering.  The same lemma says that 
$\gamma\colon L\to \incoord L$   is a lattice isomorphism, and so is $\gamma'\colon L'\to \vincoord L$, defined analogously by 
$x\mapsto \pair{\lefth_{L'}(x)}{\righth_{L'}(x)}$. Since $\psi$ is a similarity map, it commutes with the maps $\lsp$ and $\rsp$, and so it preserves the left and right join-coordinates. 
Thus, for $x\in L$, $\gamma(x)=\gamma'(\psi(x))$, that is, $\gamma=\gamma'\circ \psi$, and we also conclude that $\incoord L=\vincoord L$. Hence, \eqref{dksdghNcX}  yields that $\allcoord L=\vallcoord L$.
Since $\gamma$, $\gamma'$, and $\psi$ are isomorphisms, the equality $\gamma=\gamma'\circ \psi$ implies that $\psi=\gamma'^{-1}\circ\gamma$.
Consider the isomorphism $\delta\colon R\to \allcoord R$ from  Lemma~\ref{outlemmA} and the isomorphism $\delta'\colon R'\to \vallcoord R$ defined  by $x\mapsto \pair{\lefth_{R'}(x)}{\righth_{R'}(x)}$ analogously. It follows from Lemma~\ref{samecoordin} that $\delta$ and $\delta'$ extend $\gamma$ and $\gamma'$, respectively. Since $\delta'$ extends $\gamma'$ and they are bijections, $\delta'^{-1}$ extends $\gamma'^{-1}$. 
The equality $\allcoord R=\vallcoord R$ allows us to define a lattice isomorphism $\psi^\ast\colon R\to R'$ by 
$\psi^\ast:=\delta'^{-1}\circ\delta$. Since $\delta$ and $\delta'^{-1}$ extend $\gamma$ and $\gamma'^{-1}$, we conclude that $\psi^\ast$ extends $\gamma'^{-1}\circ\gamma$, which is $\psi$. This proves part \eqref{thmmaine} of the theorem. 

Part \eqref{thmmaind} follows from part \eqref{thmmaine} trivially.

Finally, to prove  \eqref{thmmainc}, let $L$ be a glued sum indecomposable planar semimodular lattice, and let $R_1$ and $R_2$ be \rectext{}s of $L$. Let $R_1'$, and $R_2'$ denote their full slimmings, respectively 
(with respect to their fixed planar diagrams, of course).
These full slimmings are rectangular lattices by Lemma~\ref{slimcompmml}\eqref{slimcompmmlb}.
When we delete  all eyes, one by one, from $R_i$ to obtain $R_i'$, we also delete all eyes from its cover-preserving sublattice, $L$. So, this sublattice changes to a full slimming sublattice $L_i'$ of $L$, for $i\in\set{1,2}$.  Since the deletion of eyes does not spoil the validity of Definition~\ref{dfrctext}\eqref{dfrctextc}, we conclude that $R_i'$ is a \rectext{} of $L_i'$, for $i\in\set{1,2}$.

Applying Lemma~\ref{slimcompmml}\eqref{slimcompmmlc} to the identity map $L\to L$, we obtain an automorphism $\pi$ of $L$ such that
$\restrict{\pi}{L_1'}$ is   a $\unu$-preserving isomorphism $\restrict{\pi}{L_1'} \colon L_1'\to L_2'$.
Thus, 
$\inu{L}{L_1'}=\inu{L}{L_2'}\circ\restrict{\pi}{L_1'} $.
If $B\cong M_3$ is a cover-preserving diamond sublattice of $R_i$, then all the three coatoms of $B$ belong to $L$ by  Definition~\ref{dfrctext}\eqref{dfrctextc}, and so do all elements of $B$, including $0_B$. Hence, by the definition of antislimming, if $B\cong M_3$ is a cover-preserving diamond sublattice of $R_i$, then $0_B\in L_i'$. These facts imply that, for $x\in R_i'\setminus L_i'$ and $y\in L_i'$, we have that  $\inu{R_i}{R_i'}(x)=0$ and    $\inu{R_i}{R_i'}(y)=\inu{L}{L_i'}(y) $. By the already proved part \eqref{thmmaine} of Theorem~\ref{thmmain}, $\restrict{\pi}{L_1'}$ extends to an isomorphism $\phi\colon R_1'\to R_2'$.  For $x\in R_1'\setminus L_1'$,  we have that $\phi(x)\in R_2'\setminus L_2'$, and 
the already established facts imply that
\[(\inu{R_2}{R_2'} \circ\phi)(x) =  \inu{R_2}{R_2'} (\phi(x))=0= \inu{R_1}{R_1'} (x)\text.
\]
On the other hand, for $y\in L_1'$, we have that 
\begin{align*}
(\inu{R_2}{R_2'} \circ\phi)(y) &=  \inu{R_2}{R_2'} (\phi(y)) = \inu{R_2}{R_2'} (\restrict{\pi}{L_1'}(y)) = \inu{L}{L_2'} (\restrict{\pi}{L_1'}(y))\cr
 &= (\inu{L}{L_2'} \circ \restrict{\pi}{L_1'})(y)
 = \inu{L}{L_1'} (y) = \inu{R_1}{R_1'} (y) \text.
\end{align*}
The two displayed equations show that $\inu{R_2}{R_2'} \circ\phi= \inu{R_1}{R_1'}$, which means that  $\phi\colon R_1'\to R_2'$ is a $\unu$-preserving isomorphism. By \eqref{sdTZnDsghpq}, it extends to an $R_1\to R_2$ isomorphism. Consequently, the \rectext{} of $L$ is unique up to isomorphims, which completes the proof of Theorem~\ref{thmmain}.
\end{proof}

\section{A hierarchy of planar semimodular lattice diagrams}\label{uniquediagramssection}
Our experience with planar semimodular lattices makes it reasonable  to develop a hierarchy of diagram classes for planar semimodular lattices. In this section,  
we do so. Several properties of our diagrams and their trajectories will be studied at various levels of this hierarchy. 
In particular, we are interested in what sense our diagrams are unique.
The power of this approach is demonstrated in Section~\ref{swingsection}, where we give a proof of \init{G.\ }Gr\"atzer's Swing Lemma.
Let us repeat that, unless otherwise explicitly stated, our lattices are still assumed to be finite planar semimodular lattices and the diagrams are planar diagrams of these lattices. We are going to define diagram classes $\adia$, $\bdia$, $\cdia$, and $\ddia$; the term ``hierarchy'' is explained by the inclusions $\adia\supset\bdia \supset  \cdia \supset  \ddia$. A small part of this section is just an overview of earlier results in the present setting.

\subsection{Diagrams and uniqueness in Kelly and Rival's sense}
Let $\adia$ be the class of planar diagrams of planar semimodular lattices. 
We recall some well-known concepts from, say, \init{G.\ }Cz\'edli and \init{G.\ }Gr\"atzer~\cite[Definition 3-3.5 and Lemma 3-4.2]{czgggltsta}. 
An element $x$ of a lattice $L$ is a \emph{narrows} if $x\nparallel y$ for all $y\in L$. 
The \emph{glued sum}  $L_1\glsum L_2$  of finite lattices $L_1$ and $L_2$ is a particular case of their (Hall--Dilworth) gluing: we put   $L_2$ atop $L_1$ and identify the singleton filter $\set{1_{L_1}}$ with the singleton ideal $\set{0_{L_2}}$.  Chains and lattices with at least two elements are called \emph{nontrivial}. Remember that a glued sum indecomposable lattice consists of at least four elements by definition. A folklore result says that a finite lattice $L$ and, consequently, any of its diagrams $D$ can uniquely be decomposed as
\begin{equation}
L=L_1\glsum \dots \glsum L_t\quad\text{and}\quad D=D_1\glsum \dots \glsum D_t
\label{cnNcDpsT}
\end{equation}
where $t\in\mathbb N_0:=\set{0,1,2,\dots}$ and, for every $i\in\set{1,\dots,t}$, either $L_i$ is a glued sum indecomposable lattice, or it is a maximal nontrivial (chain) interval that consists of narrows. By definition, the empty sum yields the one element lattice. 
This decomposition makes it meaningful to speak of the \emph{glued sum indecomposable components} of  $L$ or $D$. 
We say that the planar diagrams of a planar lattice $L$ are \emph{unique up to sectional left-right similarity} if for  every  $L_i$ from the canonical decomposition \eqref{cnNcDpsT}, the  planar diagrams of  $L_i$ are unique up to left-right similarity. 
The uniqueness properties of $\adia$, that is, the ``natural isomorphism'' concept in $\adia$,  are explored by the following statement.

\begin{proposition}\label{kTrTzmyY}
If $L$ is a planar semimodular lattice, then its  planar diagrams are unique up to sectional left-right similarity. They are unique even up to left-right similarity if, in addition, $L$ is glued sum indecomposable.
\end{proposition}

\begin{proof}  First, assume that $L$ is glued sum indecomposable. Let $D_1\in \adia$ and $D_2\in\adia$ by diagrams of $L$.  For $i\in\set{1,2}$, by deleting eyes as long as possible, we obtain a subdiagram $D_i'$ of $D_i$ such that $D_i'$ determines a full slimming sublattice $L_i'$ of $L$. By Lemma~\ref{slimcompmml}\eqref{slimcompmmla}, the $L_i$ are glued sum indecomposable. 
Applying Lemma~\ref{slimcompmml}\eqref{slimcompmmlc}
to the identity map $\id_L\colon L\to L$, we obtain an automorphism $\pi$ of $L$ such that $\restrict \pi{L_1'}\colon L'_1\to L'_2$ is an $\unu$-preserving lattice isomorphism. We let $\kappa:=\restrict \pi{L_1'}$ , and we consider it as  a $D_1'\to D_2'$ map. Also, let  $\refl{\kappa}:=\restrict \pi{L_1'}$, which is treated as a $D_1'\to \refl{D_2'}$ map.    
By Proposition~\ref{slmUq}, $\kappa$ or $\refl{\kappa}$ is a similarity map. We can assume that $\kappa\colon D_1'\to D_2'$  is a similarity map, because in the other case we could work with $\refl{D_2}$, whose full slimming subdiagram is $\refl{D_2'}$. Next, we define a map $\psi\colon D_1\to D_2$ as follows.
First, if  $x\in D_1'$, then  $\psi(x):=\kappa(x)$.  
Second, let $y\in D_1\setminus D_1'$. By \eqref{nrxrDmrD}, $y$ is a doubly irreducible element; its unique lower cover is denoted by $y^-$.  It follows obviously from the slimming procedure that $y^-\in D_1'$ and that $\inu{D_1}{D_1'}(y^-)\geq 1$. 
Listing them from left to right, let $y$ be the $i$-th cover of $y^-$ in $D_1$; note that $1< i < 2+\inu{D_1}{D_1'}(y^-)$, because $y^-$ has exactly $2+\inu{D_1}{D_1'}(y^-)$ covers in $D_1$.
Since $\kappa$ is $\unu$-preserving, $\inu{D_2}{D_2'}(\kappa(y^-)) = \inu{D_1}{D_1'}(y^-)$. So, $\kappa(y^-)$ has the same number of covers as $y^-$. Hence, we can define $\psi(y)$ as the $i$-th cover of $\kappa(y^-)$, counting from left to right. To sum up,
$\psi\colon D_1\to D_2$ is defined by 
\begin{equation}
\psi(z)=\begin{cases}
\kappa(z),&\text{if }z\in D_1',\cr
\text{the }i\text{-th cover of }\kappa(z^-),&
\text{if }z\notin D_1'\text{ is the }i\text{-th cover of }z^-\text.
\end{cases}\label{dlghrnMycq}
\end{equation}
We claim that $\psi\colon D_1\to D_2$ is a similarity map.
Clearly, $\psi$ is an order isomorphism, because so is $\kappa$. Hence, it is a lattice isomorphism.
To prove that $\psi$ is a similarity map, assume that $a,b,c\in D_1$, $a\prec b$, $a\prec c$,  $b\neq c$, and $b$ is to the left of $c$. Since  Definition~\ref{dkrmMpdF}   gives a selfdual condition (see the paragraph after the definition), it suffices to show that $\psi(b)$ is to the left of $\psi(c)$.  Having at least two covers, $a$ belongs to $D_1'$ by \eqref{nrxrDmrD}. If $b,c\in D_1'$, then $\psi(b)= \kappa(b)$ is to the left of $\psi(c)=\kappa(c)$, because $\kappa$ is a similarity map. Hence, the second line of \eqref{dlghrnMycq} implies that $\psi(b)$ is to the left of $\psi(c)$ even if $\set{b,c}\nsubseteq D_1'$.  Therefore, $\psi$ is a similarity map. This proves the second half of the proposition. 

Based on \eqref{cnNcDpsT}, the first half follows from the second.
\end{proof}

As a preparation for later use, we formulate the following lemma. 

\begin{lemma}\label{autslrctglexchbrs}
Let $L$ and $L'$ be  
slim rectangular lattices with  fixed diagrams $D, D'\in\adia$, respectively, and let $\phi\colon L\to L'$ be a lattice isomorphism. Then either $\phi(\leftb D)=\leftb {D'}$ and $\phi(\rightb {D})=\rightb{D'}$, or $\phi(\leftb D)=\rightb{D'}$ and $\phi(\rightb D)=\leftb {D'}$.
\end{lemma}

Although $\phi$ is also a $D\to D'$ map, it is not so obvious that it preserves the ``to the left of'' relation or its inverse. Hence, this lemma  seems not to follow from Proposition~\ref{kTrTzmyY} immediately.

\begin{proof}[Proof of Lemma~\ref{autslrctglexchbrs}]
With self-explanatory notation, \eqref{dkRtghWr}  yields that
\begin{equation}
\begin{aligned}
\Jir D&=\bigl(\llb D\cup\lrb D\bigr)\setminus\set0=\set{c_1,\dots,c_{\emel},d_1,\dots,d_{\emer}}, \cr
\Jir {D'}&=\bigl(\llb {D'}\cup\lrb {D'}\bigr)\setminus\set0=\set{c'_1,\dots,c'_{\emel'},d_1',\dots,d'_{\emer'}},
\end{aligned}
\label{skdTrzGbnWrGkR}
\end{equation}
where, as in Definition~\ref{DFgjRd}\eqref{DFgjRda}, $c_0\prec\dots\prec c_{\emel}$, $d_0\prec\dots\prec d_{\emer}$, 
 $c'_0\prec\dots\prec c'_{\emel'}$, and $d'_0\prec\dots\prec d'_{\emer'}$.  
Since $\filter{c_{\emel}}$ and $\filter{d_{\emer}}$  are chains by \eqref{upPerChains},
\begin{align}
\begin{aligned}
\leftb D&=\ideal {c_{\emel}}\cup\filter{c_{\emel}}, \quad \kern4.5pt\rightb D=\ideal {d_{\emer}}\cup\filter{d_{\emer}},\cr
\leftb {D'}&=\ideal {c'_{\emel'}}\cup\filter{c'_{\emel'}}, \quad \rightb{D'}=\ideal {d'_{\emer'}}\cup\filter{d'_{\emer'}}\text.
\end{aligned}
\label{dkQdszGtWkX}
\end{align}
We know from \init{G.\,}Cz\'edli and \init{E.\,T.\ }Schmidt~\cite[(2.14)]{czgschslim2} that, with the exceptions of $c_{\emel}$, $d_{\emer}$, $c'_{\emel'}$ and $d'_{\emer'}$, the elements given in \eqref{skdTrzGbnWrGkR} are meet-reducible. 
Thus, each of $D$ and $D'$ has exactly two doubly irreducible elements, and they are  $c_{\emel}, d_{\emer}\in D$ and  $c'_{\emel'},d'_{\emer'}\in D'$, respectively. Hence, $\set{\phi(c_{\emel}),\phi(d_{\emer})} =\set{c'_{\emel'},d'_{\emer'}}$. Thus, Lemma~\ref{autslrctglexchbrs}
 follows from \eqref{dkQdszGtWkX}.
\end{proof}

\subsection{Diagrams with normal slopes on their boundaries}
Although the title of this  subsection  does not define the class $\bdia$ of diagrams, it reveals a property, to be defined soon, of diagrams in $\bdia$. 
In the rest of the whole section, we often consider the plane as $\mathbb C$, the field of complex numbers. However, a comment is useful at this point. When dealing with diagrams, they are on the blackboard or in a page of an article or a book. All in these cases, the direction ``up'' is fixed, but $0\in \mathbb C$ and (to the right of 0)  $1\in\mathbb C$ are not necessarily. In other words, the position of the origin and the unit distance is our choice.  Let 
\begin{align*}
\bepsilon&=\cos(\pi/4)+i\,\sin(\pi/4) = \sqrt 2/2 + i\,\sqrt 2/2,\cr
\bepsilon^3&=\cos(3\pi/4)+i\,\sin(3\pi/4) = -\sqrt 2/2 + i\,\sqrt 2/2\text.
\end{align*}
We use these 8th roots of 1 to coordinatize the location of vertices of diagrams in $\bdia$, which we want to define.

For finite sequences $\vec x=\tuple{x_1,\dots,x_j}$ and $\vec y=\tuple{y_1,\dots, y_k}$, we can \emph{glue} these two sequences to obtain a new sequence
$\vec x\glsum \vec y :=\tuple{x_1,\dots, x_j,y_1,\dots, y_k}$; we can also glue more then two sequences. For $D \in \adia$, let 
\[\emel(D)=\max\set{\lefth_{D'}(x): x\in D'}\text{ and }
\emer(D)=\max\set{\righth_{D'}(x): x\in D'}, 
\]
where $D'$ is the full slimming subdiagram of $D$. (That is, $D$ determines a unique full slimming sublattice $L'$ of the lattice $L$ defined by $D$, and $D'$ consists of the vertices that represent the elements of $L'$.) Let $\chain n$ denote the \emph{chain of length} $n$; it consists of $n+1$ elements.
The superscripts ft and gh below come from ``left'' and ``right'', respectively.

%

\begin{definition}\label{dhGhTzM} \   
\begin{enumerate}[\quad\upshape(A)]
\item\label{dhGhTzMA} A planar diagram $D$ of a \emph{glued sum indecomposable}
finite planar semimodular lattice $L$ belongs to $\bdia$ if  there exist a complex number $\bdelta\in\mathbb C$ and  sequences 
\begin{equation}
\vrbal= \tuple{\rbal1, \dots, \rbal{\emel(D)}} \text{\quad and\quad}
\vrjobb= \tuple{\rjobb1,\dots,\rjobb{\emer(D)}} 
\label{rGblrjbBb}
\end{equation}
of positive real numbers such that the following conditions hold.
\begin{enumeratei}
\item\label{dhGhTzMAa} $L$ is a planar semimodular lattice. The full slimming subdiagram of $D$ and the corresponding sublattice of $L$ are denoted by $D'$ and $L'$, respectively. 
\item\label{dhGhTzMAb} For every $x\in L'$, the corresponding vertex of $D'$ is 
\begin{equation}
\bdelta +   \bepsilon^3\cdot \sum_{j=1}^{\dsty{ \lefth_{D'}(x)}}\rbal j  +   \bepsilon\cdot \sum_{j=1}^{\dsty{ \righth_{D'}(x)}}\rjobb j  \in\mathbb C\text.
\label{bdltrrvdGr}
\end{equation}
\item\label{dhGhTzMAc} We know that for each ``eye'' $x\in L\setminus L'$, there exists a unique $4$-cell $U$ of $D'$ whose interior contains $x$; the condition is that the eyes in the interior of $U$ should belong to the (not drawn) line segment connecting the left corner and the right corner of $U$ and, furthermore, these eyes should divide this line segment into equal-sized parts.
\end{enumeratei}
In this case, we say that $D$ is \emph{determined by} $\tuple{\bdelta, \vrbal,\vrjobb}$. We also say that $\tuple{\bdelta, \vrbal,\vrjobb}$ is the \emph{complex coordinate triplet} of $D\in\bdia$.
\item\label{dhGhTzMB} For a \emph{chain} $C=\set{0=c_0\prec c_1\prec\dots\prec c_n=1}$ 
of length $n\in\mathbb N_0$, a planar diagram $D$ of $C$ belongs to $\bdia$ if there exists a $\bdelta\in\mathbb C$ such that  one of the following three possibilities holds.
\begin{enumeratei}
\item There is a sequence $\vrbal=\tuple{\rbal1,\dots,\rbal n}$ of positive real numbers such that, for $j\in\set{0,\dots,n}$, the vertex representing $c_j$ is
$\bdelta +   \bepsilon^3\cdot(\rbal1+\dots+\rbal j)$.
In this case we let $\emel(D):=n$, $\emer(D):=0$, and let $\vrjobb$ be the empty sequence.
\item There is a sequence $\vrjobb=\tuple{\rjobb1,\dots,\rjobb n}$ of positive real numbers such that, for $j\in\set{0,\dots,n}$, the vertex representing $c_j$ is
$\bdelta +   \bepsilon \cdot(\rjobb1+\dots+\rjobb j)$.
In this case we let $\emel(D):=0$, $\emer(D):=n$, and let $\vrbal$ be the empty sequence.
\item There are positive integers $j$ and $k$ with $n=j+k$, sequences $\vrbal=\tuple{\rbal1,\dots,\rbal j}$ and $\vrjobb=\tuple{\rjobb1,\dots,\rjobb k}$ of positive real numbers such that $D$ is a cover-preserving $\set{0,1}$-subdiagram of the diagram $E\in\bdia$ of $\chain j\times \chain k$ determined by $\tuple{\bdelta,\vrbal,\vrjobb}$. Then $\tuple{\bdelta,\vrbal,\vrjobb}$ is said to be the complex coordinate triplet of $D$. However, this vector does not determine $D$, which can be any of the ``zigzags'' from $0_E$ up to $1_E$.
\end{enumeratei}
\item\label{dhGhTzMC} If $D\in\adia$ is a diagram of an \emph{arbitrary planar semimodular lattice} $L$ and they are decomposed as in \eqref{cnNcDpsT}, then we say that $D$ belongs to $\bdia$ if so do  its components, $D_1,\dots, D_t$.
With the self-explanatory notation, the complex coordinate triplet of $D$ is
$\tuple{\bdelta,\vrbal,\vrjobb}$ defined as
\begin{equation}\tuple{\bdelta^{(1)},{\vrbal}^{(1)}\glsum\dots\glsum {\vrbal}^{(t)}, {\vrjobb}^{(1)}\glsum\dots\glsum {\vrjobb}^{(t)}    }\text.
\label{dsdglcmpvTr}
\end{equation}
We define $\emel(D)$ and $\emer(D)$ as the number of components of $\vrbal$ and that of $\vrjobb$, respectively.
\item\label{dhGhTzMD}
We say that   $\tuple{\bdelta, \vrbal,\vrjobb}$ is a \emph{triplet  compatible with} $L$, if 
$\vrbal$ and $\vrjobb$ are finite sequences of positive real numbers,  $\bdelta\in\mathbb R$, and there exists a planar diagram $D$ of $L$ $($that is, $D$ is in $\adia$ but not necessarily in $\bdia)$ such we can obtain $\vrbal$ and $\vrjobb$ by determining the values  $\emel(D_j)$ and $\emer(D_j)$ in \eqref{rGblrjbBb}  for the canonical components of $L$ and using \eqref{dsdglcmpvTr}.
\item\label{dhGhTzME} 
For $D\in\bdia$, we say that $D$ is \emph{collinear} if $0\in\set{\emel(D),\emer(D)}$. Otherwise, $D$ is  \emph{non-collinear}. 
\end{enumerate}
\end{definition}

For example, $D_{t-1}$ and $D_t$ in Figure~\ref{fig-forkext-bdia}, which happen to be slim rectangular diagrams,  belong to $\bdia$ but not to $\cdia$, to be defined soon. In these diagrams,  $\vrbal$ and $\vrjobb$ are indicated. 
No matter if the pentagon-shaped grey-filled elements are considered, the diagram of $L$ in Figure~\ref{fig-normext-notmin} 
is also in $\bdia$; this lattice is neither slim, nor rectangular, $\vrbal=\tuple{1,1,1}$ and  $\vrjobb=\tuple{1,2,1,1}$. 
There are also many earlier examples, including 
\init{G.\,}Cz\'edli~\cite[Figure 7]{czg-mtx}, \cite[$M$ in Figure 3]{czgrepres} , \cite[$D$ in Figures 2, 3]{czgtrajcolor}, which belong to $\bdia\setminus\cdia$.
The examples in $\cdia$, to be mentioned later, are also in $\bdia$. 
\nothing{
For example,  the diagrams 
\init{G.\,}Cz\'edli~\cite[Figures 1, 2, 3, 4,  5]{czg-mtx},  \cite[Figures 2, 4, 5]{czgrepres}, \cite[Figures 1, 8, 9]{czgtrajcolor}  belong to   $\cdia$. } Our examples are non-collinear, since only nontrivial  chains have collinear diagrams in $\bdia$. However, the chain $\chain n$ with $n\geq 2$ also has non-collinear diagrams in $\bdia$.

\begin{remark}\label{whynotonlyA}
One may ask why we need parts \eqref{dhGhTzMB} and \eqref{dhGhTzMC} of definition  Definition~\ref{dhGhTzM} and why we do not apply part \eqref{dhGhTzMA} and \eqref{bdltrrvdGr} 
without assuming  glued sum indecomposability. For the answer, see Remark~\ref{thatswhynotA} later.
\end{remark}

Assume that $D\in\bdia$ is a diagram  of a glued sum indecomposable planar semimodular lattice $L$ with complex coordinate triplet $\tuple{\bdelta,\vrbal,\vrjobb}$ and that the full slimming subdiagram of $D$ is $D'$. If we change $\tuple{\bdelta,\vrbal,\vrjobb}$  to some 
 $\tuple{\bdelta^\ast,\vrvbal,\vrvjobb}$ where
\begin{equation} \vrvbal  = \tuple{\rvbal1, \dots,\rvbal{\emel(D)} }\quad\text{and}\quad  
\vrvjobb = \tuple{\rvjobb1, \dots  ,\rvjobb{\emer(D)}},
\label{rbTljbBb}
\end{equation}
then \eqref{bdltrrvdGr}, in which $\lefth_{D'}(x)$ and $\righth_{D'}(x)$ are still understood in the full slimming of the \emph{original} diagram, defines another diagram $D^\ast$ of $L$. 
We say that $D^\ast$ is obtained from $D$ by \emph{rescaling}. We can rescale a diagram $D\in\bdia$ of a chain similarly, keeping  $\emel(D)$ and $\emer(D)$ unchanged. Finally, if  $D\in\bdia$ and we rescale its  components in the canonical decomposition 
\eqref{cnNcDpsT}, then we obtain another diagram of the same lattice, and we say that it is obtained from $D$ by  \emph{piecewise rescaling}. Also, we can reflect some of the $D_j$ in  \eqref{cnNcDpsT} over a vertical axis. (Of course, we may have to move several $D_j$'s to the left or to the right in order not to ``tear'' the glued sum.) We say that the new diagram is obtained by \emph{component-flipping}. Finally, \emph{parallel shifting} means that we change $\bdelta$ in \eqref{bdltrrvdGr}. Obviously, $\bdia$ is closed with respect to component-flipping. 
Since the compatibility of a triplet does not depend on the magnitudes of its real number components, if  $\tuple{\bdelta,\vrbal,\vrjobb}$ is the complex coordinate triplet of $D\in\bdia$, then \eqref{rbTljbBb} gives a triplet compatible with $L$. Hence, Theorem~\ref{pcvWsWdG}\eqref{pcvWsWdGa} below implies that $\bdia$ is also closed with respect to piecewise rescaling;  this is not obvious, because we have to shows that  rescaling does not ruin planarity.

\begin{theorem}\label{pcvWsWdG}
For a planar semimodular lattice $L$,  the following hold.
\begin{enumeratei} 
\item\label{pcvWsWdGa} If $\tuple{\bdelta,\vrbal,\vrjobb}$ is a triplet compatible with $L$, then this triplet determines a diagram $D\in\bdia$ of $L$. That is, every triplet compatible with $L$ is the complex coordinate triplet of a unique diagram of $L$ in  $\bdia$.
\item\label{pcvWsWdGb} In particular, $L$ has a diagram in $\bdia$.
\item\label{pcvWsWdGc} The diagram of $L$ in $\bdia$ is unique  up to component-flipping, parallel shifting,  and  piecewise rescaling.
\end{enumeratei}
\end{theorem}

\begin{figure}[htb]
\includegraphics[scale=1.0]{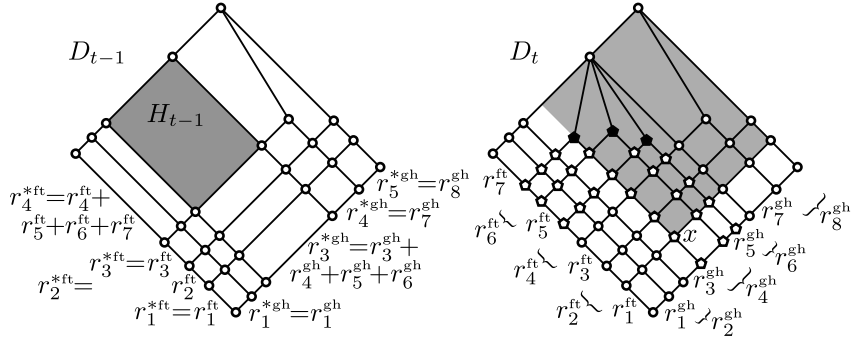}
\caption{$D_t$ is a 3-fold multifork extension of $D_{t-1}$ at $H_{t-1}$}\label{fig-forkext-bdia}
\end{figure}

Before proving this theorem, we have to recall a construction from \init{G.\ }Cz\'edli~\cite{czgtrajcolor}. Let $D$ be a planar diagram of a slim semimodular lattice $L$. A 4-cell $H$ of $D$ is \emph{distributive} if the ideal $\ideal{1_H}$ is a distributive lattice. 
To obtain a  \emph{multifork extension}  $D'$ of $D$ at the 4-cell $H$, we have to perform two steps.  As the first step, we insert $k$ new lower covers of $1_H$ into the interior of $H$. For
$\tuple{D,D',H,k}=\tuple{D_t,D_{t-1}, H_{t-1},3}$, the situation is exemplified in Figure~\ref{fig-forkext-bdia}, where $H=H_{t-1}$ is the grey 4-cell on the left and the new lower covers of  $1_H$ are the black-filled pentagon-shaped elements on the right.  (Except for $D=D_{t-1}$, $D'=D_t$, and $H=H_{t-1}$, the reader is advised to disregard the labels in the figure at present.)  In the second step,  we proceed downwards by inserting new elements (the empty-filled pentagon-shaped ones in the figure)  into the 4-cells of the two trajectories through 
$H$, and we obtain $D'$ in this way. We say that $D'$ and $L'$ are obtained by a ($k$-fold)  \emph{multifork extension} at the 4-cell $H$ from $D$ and from $L$, respectively.  
The maximal elements in $L'\setminus L$ or, equivalently, the new meet-irreducible elements, are called the \emph{source elements} of the fork  extension. (They are the black-filled pentagon-shaped elements in the figure.)  For more details, the reader might want but need not resort to   \cite[Definition 3.1]{czgtrajcolor}. Note that this construction also makes sense for slim semimodular lattices without rectangularity.

The importance of this construction is given by the following lemma. Remember that a \emph{grid} is the direct product of two finite chains. 

\begin{lemma}[{\init{G.\ }Cz\'edli~\cite[Theorem 3.7]{czgtrajcolor}}]\label{oldmfrkLma} 
If $D\in\adia$ is a slim rectangular diagram, then there exist a $t\in\mathbb N_0=\set{0,1,2,\dots}$,
\begin{equation}
\parbox{8.0 cm}
{a sequence of diagrams $D_0\subseteq D_1\subseteq \dots \subseteq D_t=D$, and  distributive $4$-cells $H_j$ of $D_j$ for
$j=0,1,\dots,t-1$
}\label{olddkmfFtHv}
\end{equation}
such that  $D_0, \dots,D_{t-1}\in \adia$, $D_0$ is a grid, and  that  $D_{j+1}$ is obtained from $D_j$ by a multifork extension at $H_j$, for   $j=0,1,\dots,t-1$.
\end{lemma}

The sequence in \eqref{olddkmfFtHv} is not unique, since  the order of multifork extensions is unique in general. However, $t$ is uniquely determined, because it is clearly the number of elements with more than two lower covers. 
Now, we  tailor Lemma~\ref{oldmfrkLma}  to our needs as follows.

\begin{lemma}\label{mfrkLma} 
Let $L$ be a slim rectangular lattice, and let $t$ be the number of its elements with more than two lower covers. If  
$\tuple{\bdelta,\vrbal,\vrjobb}$ is a triplet compatible with $L$,  then it is the complex coordinate triplet of a unique diagram $D$
of $L$ in $\bdia$ and, furthermore, there 
exist 
\begin{equation}
\parbox{8.0 cm}
{a sequence of diagrams $D_0\subseteq D_1\subseteq \dots \subseteq D_t=D$, and  distributive $4$-cells $H_j$ of $D_j$ for
$j=0,1,\dots,t-1$
}\label{dkmfFtHv}
\end{equation}
such that  $D_0, \dots,D_{t-1}\in \bdia$, $D_0$ is a grid, and  that  $D_{j+1}$ is obtained from $D_j$ by a multifork extension at $H_j$, for   $j=0,1,\dots,t-1$.
\end{lemma}

Sometimes, we refer to \eqref{dkmfFtHv} as a \emph{multifork construction sequence} of $D$.
Before proving Lemma~\ref{mfrkLma}, we formulate an auxiliary statement.

\begin{lemma}\label{dkhgTbFmk}\
\begin{enumeratei}
\item\label{dkhgTbFmka}
Let $x$ be an element of a slim rectangular lattice $L$.  If $\ideal x$ is distributive, then it is a grid $($= direct product of two chains$)$ or a chain. 
\item\label{dkhgTbFmkb}
A distributive rectangular lattice is a grid.
\end{enumeratei}
\end{lemma}

\begin{proof}  To prove part \eqref{dkhgTbFmka}, assume that $\ideal x$ is not a chain. Since $\Jir{\ideal x}\subseteq \Jir L$,   $\Jir{\ideal x}$ satisfies the condition given in \eqref{tpwcgTslm}. Hence, there is a grid $G$ such that the ordered sets $\Jir G$ and $\Jir{\ideal x}$ are isomorphic. By the classical structure theory of finite distributive lattices, see \init{G.\ }Gr\"atzer~\cite[Corollary 108]{r:Gr-LTFound}, $\ideal x\cong G$, as required. This proves part \eqref{dkhgTbFmka}. Part \eqref{dkhgTbFmkb} follows from part \eqref{dkhgTbFmka}, applied to $x=1$, and \eqref{slimiffnodiamond}.
\end{proof}

\nothing
{$\{$ 
$D_{t-1}$ $D_t$   $H_{t-1}$\\
$\rbal1$ $\rjobb1$ 
$\rbal2$ $\rjobb2$ $\rbal3$ $\rjobb3$ $\rbal4$ $\rjobb4$ 
$\rbal5$ $\rjobb5$ 
$\rbal6$ $\rjobb6$ $\rbal7$ $\rjobb7$ $\rbal8$ $\rjobb8$ 
\\
$\rvbal1=\rbal1$   $\rvbal2=\rbal2$  $\rvbal3=\rbal3$   $\rvbal4=\rbal4+\rbal5+\rbal6+\rbal7$ 
\\$\rvjobb1=\rjobb1$ $\rvjobb2=\rjobb2$      $\rvjobb3=\rjobb3+\rjobb4+\rjobb5+\rjobb6$
$\rvjobb4=\rjobb7$ $\rvjobb5=\rjobb8$ 
}

\begin{proof}[Proof of Lemma~\ref{mfrkLma}] We prove the lemma by induction on $t$. If $t=0$, then $L$ is a grid by Lemma~\ref{dkhgTbFmk}\eqref{dkhgTbFmkb} and the statement is trivial. Assume that $t>0$ and the lemma holds for $t-1$.  By  Lemma~\ref{oldmfrkLma}, there exist a slim rectangular lattice $L'$ with exactly $t-1$ of its elements having more than two lower covers,  a fixed diagram $D_0'\in \adia$ of $L'$, a distributive covering square  (equivalently, a distributive 4-cell in $D_0'$) $H_{t-1}$ of $L'$, and $k\in\mathbb N=\set{1,2,\dots}$ such that $L$ is obtained from $L'$ by a $k$-fold multifork  extension at $H_{t-1}$. With respect to $D_0'$, let $i=\lefth_{L'}(1_{H_{t-1}})$ and $j=\righth_{L'}(1_{H_{t-1}})$. Define
\begin{align*}
\vrvbal&=\tuple{ \rbal1,\dots,\rbal{i-1}, \rbal i+\dots+\rbal{i+k},\rbal{i+k+1},\dots,\rbal{k+\emel(D'_0)} }\text{ and}\cr
\vrvjobb&=\tuple{ \rjobb1,\dots,\rjobb{j-1}, \rjobb j+\dots+\rjobb{j+k},\rjobb{j+k+1}, \dots,\rjobb{k+\emer(D'_0)}}\text{.}
\end{align*}
Since $D'_0$ witnesses that  $\tuple{\bdelta,\vrvbal , \vrvjobb}$ is a triplet compatible with $L'$, the induction hypothesis applies to this triplet and $L'$. Therefore, there exists a diagram $D'\in \bdia$ of $L'$ whose compatible coordinate triplet  is $\tuple{\bdelta,\vrvbal , \vrvjobb}$  such that \eqref{dkmfFtHv} holds with $t-1$, $D'$, and $L'$ instead of $t$, $D$, and $L$; see Figure~\ref{fig-forkext-bdia} for an illustration with $\tuple{i,j,k}=\tuple{4,4,3}$. (In the figure, $H_{t-1}$ is the grey covering square on the left.) The ideal $\ideal{1_{H_{t-1}}}$ in $D'$ is a distributive lattice, so  it is a grid. Clearly, if we insert a $k$-multifork at $H_{t-1}$ according to $\tuple{\rbal i,\dots,\rbal{i+k} }$ and $\tuple{\rjobb j,\dots,\rjobb{j+k}}$ as in the figure, then we obtain a \emph{planar} diagram $D$, which belongs to $\bdia$. The definition of $\vrvbal$ and $\vrvjobb$ imply that   $\tuple{\bdelta,\vrbal , \vrjobb}$ is the complex coordinate triplet of $D$. 
This completes the induction step and proves the lemma.
\end{proof}

\begin{proof}[Proof of Theorem~\ref{pcvWsWdG}]  Part \eqref{pcvWsWdGc}  follows from  Proposition~\ref{kTrTzmyY}.  
Part \eqref{pcvWsWdGb}  is an obvious consequence of part \eqref{pcvWsWdGa}, so we only focus on part  part \eqref{pcvWsWdGa}.

It is straightforward to see that if part \eqref{pcvWsWdGa} holds for all the $L_i$ in the canonical glued sum decomposition   \eqref{cnNcDpsT} of $L$, then it also holds for $L$.   Part \eqref{pcvWsWdGa} is evident if  $L_i$ is a chain.  Part \eqref{pcvWsWdGa}
follows from Lemma~\ref{mfrkLma} if $L_i$ is a slim rectangular lattice.  So, it suffices to show the validity of part \eqref{pcvWsWdGa} if $L_i$ is a glued sum indecomposable planar semimodular lattice.  To ease the notation, we write $L$ rather than $L_i$.  Actually, since 
the application of 
Definition~\ref{dhGhTzM}\eqref{dhGhTzMAc} cannot destroy planarity, we can assume that $L$ is a \emph{slim} semimodular lattice. Let $\tuple{\bdelta,\vrbal,\vrjobb}$ be a triplet compatible with $L$. 
Theorem~\ref{thmmain} allows us to consider the \rectext{} $L'$ of $L$. Since this is only the question of the diagram-dependent values $\emel$ and $\emer$, it follows from Lemma~\ref{samecoordin} and Proposition~\ref{kTrTzmyY} that $\tuple{\bdelta,\vrbal,\vrjobb}$ is compatible with $L'$. Thus, Lemma~\ref{mfrkLma} gives us a diagram $D'\in \bdia$ of $L'$ such that  $\tuple{\bdelta,\vrbal,\vrjobb}$ is the complex coordinate triplet of $D'$. We conclude from Lemma~\ref{samecoordin} that the elements of $L$ in $D'$ are exactly in the appropriate places that  \eqref{bdltrrvdGr} demands for $L$. These elements form a subdiagram $D$. By Lemma~\ref{dkjTghNVc}, $D$ is a region of $D'$. As a region of a planar diagram, $D$ is also planar. It is clear, again by  Lemma~\ref{samecoordin}, that  $\tuple{\bdelta,\vrbal,\vrjobb}$ is the complex coordinate triplet of $D$. In particular, $D\in\bdia$.
\end{proof}

Although $\cdia$ is not yet defined, the diagrams in $\jdia j$, $j\in\set{1,2}$, of a \emph{rectangular} lattice are particularly easy to draw. Hence,  we formulate the following remark, which follows from Lemma~\ref{samecoordin}.
Note, however, that \eqref{bdltrrvdGr} allows us to draw a diagram directly, without drawing its \rectext{}.

\begin{remark}\label{kdTZcWZ} For $j\in\set{1,2}$, a diagram $D\in \jdia j$ of a planar semimodular lattice $L$ with more than two elements can be constructed as follows.
\begin{enumeratei}
\item Take a \rectext{} $R$ of $L$.
\item Find a diagram $E\in\jdia j$ of $R$.
\item Remove the vertices corresponding to $R\setminus L$ and the edges not in $L$.
\end{enumeratei}
\end{remark}

As a counterpart of this remark, we formulate the following statement here, even if $\cdia$ is not yet defined. (We need this statement before introducing $\cdia$, and its validity for $\bdia$ will trivially imply that it holds for $\cdia$.)
We say that $E$ is a \emph{\rectext{} diagram} of a planar semimodular diagram $D$ if $E$ is a planar diagram of a \rectext{} of the lattice determined by $D$ and  we can obtain $D$ from $E$ by omitting some vertices and edges. The  equation $E_1=E_2$ below is understood in the sense that the two diagrams consist of the same complex numbers as vertices and the same edges.
Note that a glued sum indecomposable lattice cannot have a collinear diagram; see 
Definition~\ref{dhGhTzM}\eqref{dhGhTzME}.

\begin{proposition}\label{dgmpRopmain}
If $j\in\set{0,1,2}$,     $D\in\jdia j$, and $|D|\geq 3$, then  the following  assertions hold.
\begin{enumeratei}
\item\label{dgmpRopmaina} If $j=0$, then $D$ has a \rectext{} diagram in $\adia$.
\item\label{dgmpRopmainb} If $j\in\set{1,2}$ and $D$ is non-collinear,  then  $D$ has a \rectext{} diagram in $\jdia j$.
\item\label{dgmpRopmainc} Assume, in addition, that $D$ is glued sum indecomposable. Let  $E_1\in\jdia j$ and $E_2\in\jdia j$ be  \rectext{} diagrams of $D$. If $j$ is in $\set{1,2}$, then $E_1=E_2$. If $j=0$, then $E_1$ is similar to $E_2$.
\end{enumeratei}
\end{proposition}

Besides that $\ddia$ has not been defined yet, 
Remark~\ref{remnotInScope} will explain  why 
$j=3$ is not allowed above.

\begin{proof}[Proof of Proposition~\ref{dgmpRopmain}\eqref{dgmpRopmainc}] 
We can assume that $D$ is slim; then its \rectext{} are also slim Theorem~\ref{thmmain}\eqref{thmmainb}. The reason is that if $D$ is not slim, then we can work with its full slimming subdiagram $D'$, and we can put the eyes back later, in the \rectext{}. 
For \emph{lattices}, the ambiguity of the full slimming can cause some difficulties, see 
Lemma~\ref{slimcompmml}. However, for \emph{diagrams}, the full slimming is uniquely determined and cannot cause any problem; see also 
Definition~\ref{dhGhTzM}\eqref{dhGhTzMAc}

Part \eqref{dgmpRopmainc} for   
 $j\in\set{1,2}$ follows from  Lemma~\ref{samecoordin}, \eqref{dksdghNcX}, and \eqref{bdltrrvdGr}. Next, we deal with part \eqref{dgmpRopmainc} for $j=0$. So let $D\in \adia$ and let $E_1,E_2\in\adia$ be \rectext{} diagrams of $D$. Let $L$ be the lattice determined by $D$. By \eqref{dksdghNcX}, objects like $\leftcoord {E_k}$ and $\allcoord {E_k}$ will be understood with respect to $D$.  For $k\in\set{1,2}$, take the  coordinatization map $\delta_k\colon E_k\to  \allcoord{E_k}$, given in the last sentence of Lemma~\ref{outlemmA}. Since it is a lattice isomorphism, so is the map $\delta_2^{-1}\circ\delta_1\colon E_1\to E_2$, which we denote by  $\eta$. It follows from \eqref{jGjdTbZe} that 
$\eta (\leftcoord{E_1})  \subseteq  \leftcoord{E_2}$.  Hence, $\eta(\leftb{E_1})\subseteq \leftcoord{E_2}$.  
The glued sum indecomposability of $L$ yields that $\rightb {E_2}\cap \leftcoord{E_2}=\set{0,1}$.  This excludes that $\eta(\leftb{E_1})= \rightb{E_2}$. Thus, Lemma~\ref{autslrctglexchbrs} yields that $\eta(\leftb{E_1})= \leftb{E_2}$.  Consequently,  \eqref{kkfpPjRw} implies that $E_1$ is similar to $E_2$, as required, completing the proof of part \eqref{dgmpRopmainc}.
\end{proof}

\begin{figure}[htb]
\includegraphics[scale=1.0]{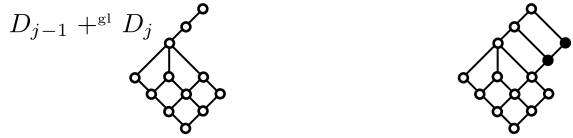}
\caption{Getting rid of a collinear chain $D_j$}\label{fig-getridchain}
\end{figure}

\begin{proof}[Outline for Proposition~\ref{dgmpRopmain}\eqref{dgmpRopmaina}-\eqref{dgmpRopmainb}] 
As opposed to part \eqref{dgmpRopmainc}, we will not use parts \eqref{dgmpRopmainb} and \eqref{dgmpRopmainb} in the paper. Hence, and also because of space considerations, we only give the main ideas.
Consider the canonical decomposition 
$D=D_1\glsum\dots\glsum D_t$; see \eqref{cnNcDpsT}.
In the simplest case, we can take
a \rectext{} $E_j$ of $D_j$ for every $j$; either by following the argument in the proof of  Proposition~\ref{dgmpRopmain}\eqref{dgmpRopmainc} for $j=0$, see also \eqref{dhdttjrjkL}, or trivially for chain components. Then Figure~\ref{fig-nonindecomposable} indicates how to continue by successively replacing the glued sum of two consecutive rectangular diagrams by their \rectext{}. However, there are less simple cases, where some $D_j$ are collinear or $|D_j|=2$. Then we can exploit the fact that  $D_j\neq D$, and  so at least one of  $D_{j-1}$ and $D_{j+1}$ exists and it is glued sum indecomposable. If, say, $D_{j-1}$ is glued sum indecomposable, then $D_{j-1}\glsum D_j$, see on the left of Figure~\ref{fig-getridchain}, can be replaced by the diagram on the right of the same figure.

The straightforward but tedious details proving that our method yields a \rectext{} diagram of $D$ are omitted. 
\end{proof}

\subsection{Equidistant diagrams with normal slopes on their boundaries}
We define a subclass $\cdia$ of $\bdia$ as follows
\begin{definition}\label{dkGkGn}
A diagram $D\in \bdia$ belongs to 
$\cdia$ if its complex coordinate triplet is of the form
\begin{equation}
\tuple{\bdelta,\vrbal,\vrjobb}=\tuple{\bdelta,\tuple{r,\dots,r},\tuple{r,\dots,r}}
\label{dTfchnMr}
\end{equation}
for a positive constant  $r\in\mathbb R$. 
``\emph{Rescaling}" in $\cdia$ means to change $r$.
\end{definition}

From Theorem~\ref{pcvWsWdG}, we clearly obtain the following statement.

\begin{corollary}\label{pcvhjRdG}
Every planar semimodular lattice has a  diagram in $\cdia$, which is unique  up to  rescaling in $\cdia$, parallel shifting, and component-flipping.
\end{corollary}

The diagrams in Figures~\ref{fig-nonreliso}, \ref{fig-nonindecomposable}, and     $\widehat R$ in Figure~\ref{fig-normext-notmin}, and, for example, the diagrams in  
\init{G.\,}Cz\'edli~\cite[Figures 1, 2, 3, 4,  5]{czg-mtx},  \cite[Figures 2, 4, 5]{czgrepres}, and \cite[Figures 1, 8, 9]{czgtrajcolor}  belong to   $\cdia$. Furthermore, the fact that the diagrams in \init{G.\ }Cz\'edli~\cite[Figure 4]{czgquasiplanar} 
belong to $\cdia$ is more than an esthetic issue; it is an integral part of the proof of \cite[Lemma 3.9]{czgquasiplanar}.
Generally, for a planar semimodular lattice, we use a diagram outside $\cdia$ only in the following two cases: a diagram is extended or a subdiagram is taken, or if there are many eyes in the interior of a covering square. (In the first but not the second case,   $\bdia$ is recommended.)

\subsection{Uniqueness without compromise}
The "up" direction in our plane (blackboard, page of an article, etc.)  is usually fixed. Hence, for a diagram $D\in \cdia$, the parameters $\bdelta$ and $r$ in \eqref{dTfchnMr} does not effect the geometric shape and the orientation of $D$. So, we can choose $\pair\bdelta r=\pair 01$.  As we will see soon, this means that we choose  the complex plain $\mathbb C$ so that  $0_D$ is placed at  $0\in \mathbb C$ 
and the leftmost atom of $D$ is placed at $\bepsilon^3$.  
However,   reflecting some of the $D_j$ in the canonical decomposition \eqref{cnNcDpsT} over a vertical axis may effect the geometric shape of $D$, and we  want to get rid of this possibility. To achieve this goal, we need some preparation.

Let $D$ be a planar diagram of a slim semimodular lattice. Recall from 
 \init{G.\,}Cz\'edli and \init{E.\,T.\ }Schmidt~\cite{czgschperm} that the \emph{Jordan--H\"older permutation} $\pi_D$, which was associated with $D$ first by 
\init{H.\ }Abels~\cite{abels} and
\init{R.\,P.\ }Stanley~\cite{stanley}, can be defined as follows. Let 
\begin{equation*}\parbox{6.5cm} 
{$\leftb D=\set{0=e_0\prec e_1\prec\dots\prec e_n=1}$ and $\rightb D=\set{0=f_0\prec f_1\prec\dots\prec f_n=1}$,}
\label{fjHzTnS} 
\end{equation*}
and let $S_n$ denote the symmetric group consisting of all $\set{1,\dots,n}\to\set{1,\dots,n}$ permutations. We define $\pi_D\in S_n$ by the rule
\[\pi_D(i)=j\iff [e_{i-1},e_i]\text{ and }[f_{j-1},f_j]\text{ belong to the same trajectory.}
\]
Obviously, for slim semimodular lattices diagrams $D_1$ and $D_2$, 
\begin{equation}
\text{if }D_1\text{ is similar to }D_2 \text{, then } \pi_{D_1}=\pi_{D_2}\text.
\label{simdiagrsameperm}
\end{equation}
For $\sigma,\tau\in S_n$, $\sigma$ \emph{lexicographically precedes} $\tau$, in notation $\sigma\lexleq \tau$, if   
\begin{equation}
\tuple{\sigma(1),\dots,\sigma(n)} \leq \tuple{\tau(1),\dots,\tau(n)} 
\label{ptlexeqdef}
\end{equation}
in the lexicographic order. Although \eqref{ptlexeqdef} is meaningful for all slim semimodular diagrams,  Section~\ref{proofssection} does not work for chains. 
For example, the diagrams in $\cdia$ of a chain cannot be distinguished by means of join-coordinates. Hence, chain components in the canonical decomposition \eqref{cnNcDpsT} would lead to difficulties. Therefore, we assume glued sum indecomposability here. So  let 
 $D_j'\in \adia$ be the full slimming diagram of $D_j\in \adia$ for $j\in\set{1,2}$ such that $D_1'$ is similar to $D_2'$ and, in addition, let the $D_j$ be glued sum indecomposable. 
Note that if $\height(x)=\height(y)$ and $x\neq y$, then $x\parallel y$ and, by  Lemma~\ref{leftrightlemma}\eqref{leftrightlemmab},
either $x\rellambda y$, or $y\rellambda x$. Hence, we can consider  the unique (repetition free) list 
$\tuple{x_1^{(j)},x_2^{(j)},\dots,x_k^{(j)}}$  of elements of $D_j'$ such that, for all $1\leq s<t\leq k$, either $\height(x_s^{(j)})<\height(x_t^{(j)})$, or $\height(x_s^{(j)})=\height(x_t^{(j)})$ and $x_s^{(j)}\rellambda x_t^{(j)}$.
Denoting the similarity map $D_1'\to D_2'$ by $\phi$, note that
\begin{equation}
\phi\text{ preserves the list, that is, }
\phi(x_s^{(1)})=x_s^{(2)} \text{ for } \forall s\in\set{1,\dots,k}\text.
\label{simmapPreslist}
\end{equation}
We say that      $D_1\antilexleq D_2$ if the $k$-tuple
$\tuple{\inu{D_1}{D_1'}(x_1^{(1)}),\dots \inu{D_1}{D_1'}(x_k^{(1)})}$  equals or lexicographically precedes  
$\tuple{\inu{D_2}{D_2'}(x_1^{(2)}),\dots \inu{D_2}{D_2'}(x_k^{(2)})}$.  Let us emphasize that $D_1\antilexleq D_2$ only makes sense if the full slimming sublattice of $D_1$ is similar to that of $D_2$. 
The upper integer part of a real number $x$ is denoted by $\lceil x\rceil$; for example, $\lceil \sqrt 2 \rceil = 2 =\lceil 2\rceil$.
Now we are in the position to define a class $\ddia\subset \cdia$ of diagrams as follows.

\begin{definition}\label{djNmvYx} 
Let  $D\in\cdia$ be a diagram, and let $L$ denote the  planar semimodular lattice it determines. Let $D'$ and $L'$ denote the full slimming subdiagram of $D$ and the corresponding full slimming sublattice of $L$, respectively. Then $D$ belongs to $\ddia$ if one of the conditions \eqref{djNmvYxA}, \eqref{djNmvYxA}, and \eqref{djNmvYxA} below~holds.
\begin{enumerate}[\quad\upshape(A)]
\item\label{djNmvYxA} $D$ is glued sum indecomposable and the following three conditions hold.
\begin{enumeratei}
\item\label{djNmvYxAa}  The complex coordinate triplet of $D$ is  $\tuple{0,\tuple{1,\dots,1},\tuple{1\dots,1}}$.
\item\label{djNmvYxAb} For every diagram 
$E'\in \adia$ of $L'$, $\pi_{D'}\lexleq \pi_{E'}$.
\item\label{djNmvYxAc} For every diagram $E\in \adia$ of $L$, if the full slimming of  $E$ is similar to $D$, then  $E\antilexleq D$.
\end{enumeratei}
\item\label{djNmvYxB} $D$ is a chain $D=\set{0=d_0\prec \dots\prec d_n=1}$ and, for $j\in\set{0,\dots, n}$,  
\[d_j=\begin{cases}
j\bepsilon^3,&\text{if }j\leq \lceil n/2 \rceil,\cr
\lceil n/2 \rceil\bepsilon^3 + (j- \lceil n/2 \rceil)\bepsilon,&\text{if }j > \lceil n/2\text.
\end{cases}
\]
\item\label{djNmvYxD} The canonical decomposition \eqref{cnNcDpsT} consists of more than one components, that is, $t>1$, and, for every $j\in\set{1,\dots,t}$, an appropriate parallel shift (that is, changing the first component of the complex coordinate triplet) turns $D_j$ into a diagram in $\ddia$.
\end{enumerate}
\end{definition}

For example, the  diagrams in Figures~\ref{fig-nonreliso}, \ref{fig-nonindecomposable}, and \ref{fig-oterr} are in $\ddia$; see also Figure~\ref{fig-centgravleft}. 
 Observe that in \eqref{djNmvYxAb} and \eqref{djNmvYxAc} of Definition~\ref{djNmvYx}, $E'$ and $E$ range in $\adia$ rather than only in $\cdia$. 
Of course, there could be other definitions to make the following proposition valid. Our vague idea  is that ``at low level'', we want more elements on the left than on the right.

\begin{figure}[htb]
\includegraphics[scale=1.0]{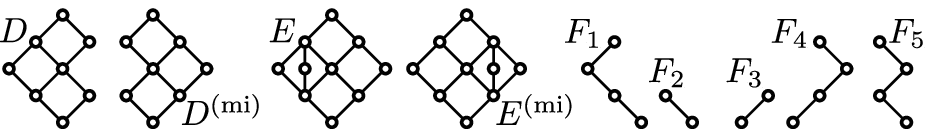}
\caption{ $D,E, F_1,F_2 \in \ddia$ but $\refl D, \refl E, F_3,F_4,F_5\not\in \ddia$ }\label{fig-centgravleft}
\end{figure}

\begin{proposition}\label{dqHRT}
Every planar semimodular lattice $L$ has a \emph{unique} diagram $D$ in $\ddia$. The uniqueness means that if $D^\ast,D^\natural\in\ddia$ are diagrams of $L$, then their vertex sets are exactly the same subsets of $\,\mathbb C$, and their edge sets are also  the same sets of straight line segments in the complex plane. 
\end{proposition}

\begin{proof} 
It suffices to deal with the existence statement, because the uniqueness part is evident. 
Let $L'$ be a full slimming sublattice of $L$. We obtain from  Corollary~\ref{pcvhjRdG}  that $L'$ has a diagram $D'\in \cdia$. We can assume that $L$ and, consequently, $L'$ are glued sum indecomposable. 
After rescaling in $\cdia$ and parallel shifting if necessary, we can assume that 
\begin{equation}
\text{the complex coordinate triplet of }D'\text{ is }\tuple{0,\tuple{1,\dots,1},\tuple{1,\dots,1}}\text.
\label{dkzThjT}
\end{equation}
Of course, the same holds for $\refl{D'}$, 
obtained from $D'$ by reflecting it over  the ``imaginary'' axis $\set{ri:r\in \mathbb R}$.
It follows from Proposition~\ref{kTrTzmyY} that every diagram $E'\in \adia$ of $L'$ is similar to $D'$ or $\refl{D'}$. 
Hence, by \eqref{simdiagrsameperm},
all the permutations we have to consider belong to $\set{\pi_{D'},\pi_{\refl{D'}}}$. Thus, $D'$ or $\refl{D'}$ belongs to $\ddia$, depending on $\pi_{D'}\lexleq \pi_{\refl{D'}}$ or $\pi_{\refl{D'}}\lexleq   \pi_{D'}$, because both represent $L'$ and satisfy \eqref{dkzThjT}. Let, say, $D'\in \ddia$. 

Since $D'$ has finitely many 4-cells, and 
the positions of the eyes in a given 4-cell   are determined by Definition~\ref{dhGhTzM}\eqref{dhGhTzMAc}, we conclude that  there are only finitely many antislimmings $D_1,\dots,D_k$ of $D'$ in $\cdia$ that define $L$.  By changing the subscripts is necessary, we can assume that $D_j\antilexleq D_k$ holds for all $j\in\set{1,\dots, k}$. We assert that $D_k\in\ddia$. To prove this,  consider an arbitrary diagram $E\in\adia$ of $L$ such that its full slimming subdiagram $E'$  is similar to $D'$. We have to show that $E\antilexleq D_k$.
Let $\phi\colon D'\to E'$ be similarity map, and define a map $g\colon D'\to\mathbb N_0$ as $g=\inu{E}{E'}\circ \phi$; see \eqref{slimmingMAP}. For each $4$-cell $H$ of $D'$, let us add $g(0_H)$  eyes into the interior of $H$,  keeping Definition~\ref{djNmvYx}\eqref{djNmvYxAc} in mind. In this way, we obtain a diagram $D\in\ddia$, which is an antislimming of $D'$. Since $g=\inu D{D'}$ obviously holds, the similarity map $\phi$ is an $\unu$-preserving isomorphism. Applying \eqref{sdTZnDsghpq} to the lattices our diagrams determine, it follows that $E$ and $D$ define isomorphic lattices. Hence, $D\in \cdia$ defines $L$, and we obtain that $D=D_j$ for some $j\in\set{1,\dots, k}$. 
Since $\phi$ is $\unu$-preserving and it preserves the list of  \eqref{simmapPreslist},
$E\antilexleq D_k\iff D_j\antilexleq D_k$. Therefore, by the choice of $D_k$, $E\antilexleq D_k$, as required.
\end{proof}

\begin{figure}[htb]
\includegraphics[scale=1.0]{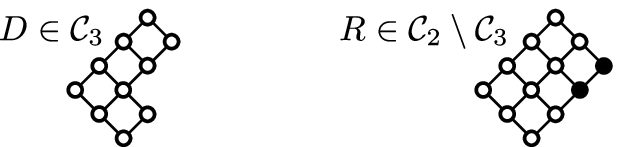}
\caption{In $\ddia$, $D$ has no \rectext{} diagram}\label{fig-notinthescope}
\end{figure}

Consider $D$ and $R$ in Figure~\ref{fig-notinthescope}. By Proposition~\ref{dgmpRopmain}\eqref{dgmpRopmainc}, $R$   is the only \rectext{} of $D$ in  $\cdia$. Hence, we obtain the following remark.

\begin{remark}\label{remnotInScope}  Part \eqref{dgmpRopmainb} of 
Proposition~\ref{dgmpRopmain} fails for $j=3$. 
\end{remark}

\section{A toolkit for diagrams in $\bdia$}\label{toolkitsection}
For $x=x_1+x_2i$ and $y=y_1+y_2i$ in $\mathbb C$, where $x_1,x_2,y_1,y_2\in\mathbb R$,   we say that $x$ is \emph{geometrically below} $y$ if $x_2\leq y_2$. 
In addition to Theorem~\ref{pcvWsWdG}\eqref{pcvWsWdGc}, the following statement also  indicates well the advantage of $\bdia$ over $\adia$; note that this statement  would fail without assuming slimness.

\begin{corollary} \label{nslplQcor}
Let $D\in\bdia$ be a slim semimodular diagram. For distinct $x,y\in D$, we have $x<y$ iff $x$ is geometrically below $y$ and the slope of the line through $x$ and $y$ is in the interval $[\pi/4,3\pi/4]$ $($that is, between $45^\circ$ and $135^\circ)$.
\end{corollary}

\begin{proof} First, we deal with the case where $D$ is glued sum indecomposable.  Let $x\neq y\in D$, and denote the line through $x$ and $y$ by $\ell$.  Assume that $x<y$.  Since $\lefth_D$ and $\righth_D$ are monotone, we obtain from \eqref{bdltrrvdGr} that $y-x=s_1\bepsilon^3+s_2\bepsilon\in \mathbb C$ with nonnegative $r_1,r_2\in \mathbb R$. This implies that the slope of $\ell$ is in  $[\pi/4,3\pi/4]$. Conversely, assume that  the slope of $\ell$ is in  $[\pi/4,3\pi/4]$. Again, we can write the complex number $y-x$ in the form  $y-x=t_1\bepsilon^3+t_2\bepsilon\in \mathbb C$ with $t_1,t_2\in \mathbb R$. The  assumption on the slope of $\ell$ implies that $t_1$ and $t_2$ are nonnegative. Thus, we can extract from \eqref{bdltrrvdGr} that $\lefth_D(x)\leq \lefth_D(y)$ and $\righth_D(x)\leq \righth_D(y)$. Hence, $x\leq y$ by \eqref{sldiHjNxYb}. 
So, Corollary~\ref{nslplQcor} holds for the glued sum indecomposable case, which easily implies its validity for the general case. 
\end{proof}

In view of Remark~\ref{kdTZcWZ} and the simplicity of the constructive step described in Definition~\ref{dhGhTzM}\eqref{dhGhTzMAc}, we will mainly focus on \emph{slim rectangular} diagrams. 
Let $D\in \bdia$, and let $[u,v]$ or, in other words, $u\prec v$ be an edge of the diagram $D$. If the angle this edge makes with a horizontal line is $\pi/4$ ($45^\circ$) or  $3\pi/4$ ($135^\circ$), then we say that the edge is of \emph{normal slope}.
If this angle is strictly between $\pi/4$ and  $3\pi/4$, then the edge is \emph{precipitous} or, in other words, it is of \emph{high slope}. The following observation shows that  edges of ``low slopes" do not occur. The 
\emph{boundary}   and the
\emph{interior}   of a diagram $D$ are $\bound D:= \leftb D\cup \rightb D$ and  $D\setminus \bound D$, respectively. 
Remember that $\Mir D$, the set of \emph{meet-irreducible elements}, is $\{x\in D:x$ has exactly one cover$\}$.

\begin{observation}\label{dksRptG}
 Let $D\in\bdia$ be a slim rectangular lattice diagram. If $u\prec v$ in $D$, then exactly one of the following two possibilities holds:
\begin{enumeratei}
\item the edge $[u,v]$ is of normal slope and $u\in \bound D \cup (D\setminus \Mir D)$;
\item the edge $[u,v]$ precipitous,  $u\in\Mir D$,  $u$ is  in $D\setminus \bound D$, the interior of $D$,  and $v$ has at least three lower covers.
\end{enumeratei}
\end{observation}

\begin{proof}[Proof of Observation~\ref{dksRptG}]  Take a multifork construction sequence \eqref{dkmfFtHv} . Reversing the passage from $D_{t-1}$ to $D_t$, that is, omitting the last ``multifork'', we see that $D_{t-1}\in \bdia$. And so on, all the $D_j$ belong to $\bdia$. Since $D_0$ is distributive, it is a \emph{grid}, that is, the direct product of two nontrivial chains. 
Hence,  the statement obviously holds for $D_0$. Finally, it is easy to see that if Observation~\ref{dksRptG} holds for $D_j$, then so it does for $D_{j+1}$.
\end{proof}

The following observation follows by a trivial induction based on Lemma~\ref{mfrkLma}.
The case $y\in\Jir D$, equivalently, $y\in \llb D\cup\lrb D$,  is not considered in it. 

\begin{observation}\label{dklftrght}
If $D\in\bdia$ is a slim rectangular lattice diagram, $x\prec y\in D$, and $y\notin\Jir D$,  then 
the following three conditions are equivalent.
\begin{enumeratei}
\item\label{dklftrghta} The edge $[x,y]$ is of slope $\pi$ $($respectively, $3\pi/4)$.
\item\label{dklftrghtb} $x$ is the leftmost $($respectively, rightmost$)$ lower cover of $y$.
\end{enumeratei}
\end{observation}

Let $u$ be a trajectory of a slim semimodular lattice diagram such that its edges, from left to right, are listed as 
$[x_0,y_0]$, $[x_1,y_1]$, \dots, $[x_k,y_k]$. For $a\nparallel b$, let $\mixint ab$ denote $[a,b]$ if $a\leq b$, and let it denote $[b,a]$ if $b\leq a$. That is, $\mixint ab=[a\wedge b,a\vee b]$.
The \emph{lower border} of $u$ is the set $\set{\mixint{x_{j-1}}{x_j}:1\leq j\leq k}$ of edges. Similarly, the \emph{upper border} of $u$ is $\set{\mixint{y_{j-1}}{y_j}:1\leq j\leq k}$.   

\begin{corollary}\label{dkzTbR}
Let $D\in \bdia$ be a diagram of a slim semimodular lattice $L$. If $T$ is trajectory of $D$, then 
every edge of its lower border is of normal slope.
\end{corollary}

\begin{proof} Clearly, we can assume that $D$ is glued sum indecomposable. 
 Let $j$ be in $\set{1,\dots, k}$. First, assume that, in addition, $L$ is rectangular. By \eqref{1rctallrect}, so is $D$.  We  assume that $y_{j-1}<y_j$,  because otherwise we can work in $\refl D$. Thus, $T$ goes upwards at $[x_{j-1},y_{j-1}]$. Hence,  $x_{j-1}<x_j$,  $\mixint{x_{j-1}}{x_j}=[x_{j-1},{x_j}]$, and $x_{j-1}=x_j\wedge y_{j-1}\notin \Mir L$. Therefore, Observation~\ref{dksRptG} yields that $\mixint{x_{j-1}}{x_j}$ is of normal slope. Second, we do not assume that $D$ is rectangular. Then, by Lemma~\ref{dkjTghNVc} and Proposition~\ref{dgmpRopmain}, $D$ is a region of a unique slim rectangular diagram $E\in\bdia$, and $T$ is a section from $\leftb D$ to $\rightb D$ of a trajectory $T'$ of $E$. 
Since $\mixint{x_{j-1}}{x_j}$ is on the lower border of  $T'$, it is of normal slope in $E$. By Remark~\ref{kdTZcWZ}, it is of the same slope in $D$.
\end{proof}

For a  $4$-cell $H$, we say that $H$ is a \emph{$4$-cell with normal slopes} if each of the four sides of $H$ is of normal slope.

\begin{corollary} \label{edgesOfdistCellNSlope}
If $H$ is a distributive $4$-cell of a diagram $D\in \bdia$, then $H$ is of normal slopes and, moreover, every edge in $\ideal{1_H}$ is of normal slope. 
\end{corollary} 

\begin{proof}  Its a folklore result, see 
the Introduction in \init{G.\ }Gr\"atzer and \init{E.\ }Knapp \cite{gratzerknapp1} or see \init{G.\,}Cz\'edli and \init{E.\,T.\ }Schmidt~\cite[Lemmas 2 and 16]{czgschvisual}, that 
\begin{equation}
\parbox{6cm}{no element of  a planar distributive lattice covers more than two elements.}
\label{ifdistno3covered}
\end{equation}
Hence, the corollary follows from Observation~\ref{dksRptG}.
\end{proof}

\begin{figure}[htb]
\includegraphics[scale=1.0]{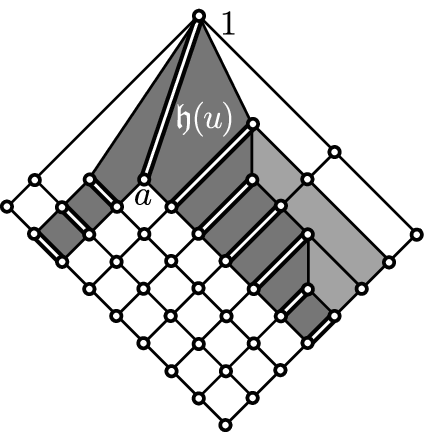}
\caption{Territories: $\terr u$ and $\oterr u$}\label{fig-oterr}
\end{figure}

\begin{remark}  \label{thatswhynotA} The answer to the question in  Remark~\ref{whynotonlyA} is the following. If we applied \eqref{bdltrrvdGr}  to a chain $C$, then 
no edge in the diagram $C'\in \bdia$ of $C$ would be of normal slope. Hence, for $j\in\set{1,2}$, Proposition~\ref{dgmpRopmain} would 
fail by the following argument, and we wanted to avoid this failure.
\end{remark}

\begin{proof}[Proof of Remark~\ref{thatswhynotA}]
Suppose, for a contradiction that $D\in\bdia$ is a \rectext{} diagram of $C'$. Take an edge $a\prec b$ in $C'$, it is not of normal slope. Hence, by Observation~\ref{dksRptG}, $b$ has at least three lower covers but only $a$ of them belongs to $C$. This contradicts Definition~\ref{dfrctext}\eqref{dfrctextc}.
\end{proof}

For a slim $D\in\adia$, the set of trajectories of $D$ is denoted by $\Traj D$.

\begin{definition} \label{dkGhRwQ}
Let $D\in\bdia$ be a slim rectangular diagram, and let $u$ be a trajectory of $D$.  
\begin{enumeratei}
\item \label{dkGhRwQa}
The \emph{top edge} of a trajectory $u$, denoted by $\inh(u)$, belongs to $u$ and is defined by the property that $y\leq 1_{\inh(u)}$ for all $[x,y]\in u$; see \eqref{trtrfR}.
\item \label{dkGhRwQb} For a region or a 4-cell $A$, the territory of $A$ is denoted by $\terr A$. It is a closed polygon in the plane.
\item \label{dkGhRwQc}
Similarly, the \emph{territory}  $u$, denoted by $\terr u$  is the closed polygon of the plane covered by the squares of $u$. An example is given in  Figure~\ref{fig-oterr}, where $u$ is the hat-trajectory through $\inh(u)=[a,1]$ and it consists of the double (thick) edges; the territory of $u$ is the dark grey  area.
\item \label{dkGhRwQd}With reference to the multifork construction sequence \eqref{dkmfFtHv}, for each $x\in D$, there is a smallest  $j$ such that $x\in D_j$.   We denote this smallest $j$ by $\birthno(x)$; the acronym comes from ``year of birth". For an interval  $\ing$, $\birthno(\ing)=\max\set{\birthno(0_\ing),\birthno(1_\ing)}$ is the  smallest $j$ such that $\ing$ is an edge of $D_j$. 
For $x,y\in D$, $x$ is \emph{younger} than $y$ if $\birthno(x)>\birthno(y)$, and similar terminology applies for intervals. 
Note that an interval $\ing$ can contain elements younger than $\ing$ itself.
\item \label{dkGhRwQe}
For $u\in\Traj D$, we define $\birthno(u)$ as  $\birthno(\inh(u)))$. 
The trajectory of $D_{\birthno(u)}$ that contains $\inh(u)$ is denoted by $\birthtr(u)$; now the acronym comes from ``birth trajectory''.     Clearly, $u$ is a straight trajectory iff $\birthno(u)=0$. Also, $u$ is a hat-trajectory iff $\birthno(u)>0$.
\item \label{dkGhRwQf} For $\birthno(u)\leq j\leq t$, the trajectory of $D_j$ through the edge $\inh(u)$ is denoted by $\anc uj$, and it is called an \emph{ancestor} of $u$. (Observation~\ref{cdsrMsdGs}\eqref{cdsrMsdGsd} will show that $\anc uj$ exists.) In particular,  we have that $\anc u{\birthno(u)}=\birthtr(u)$.
\item \label{dkGhRwQg} 
The \emph{original territory} of $u$, denoted by $\oterr u$, is the  territory of $\birthtr(u)$   in $D_{\birthno(u)}$. 
For example, in   Figure~\ref{fig-oterr}, $\oterr u$ is the grey area (dark grey and light grey together). 
The \emph{original upper border} of $u$ is the upper border of $\birthtr(u)$ in  $D_{\birthno(u)}$; it is a broken line consisting of several (possibly, one) straight line segments in the plane. Similarly, the \emph{original lower border} of $u$ is the lower border of $\birthtr(u)$ in  $D_{\birthno(u)}$.
\item \label{dkGhRwQh} 
The \emph{halo square} of $u$ is the $4$-cell  $H_{\birthno(u)-1}$   of  $D_{\birthno(u)-1}$ into which the multifork giving birth to $u$ is inserted. 
\end{enumeratei}
\end{definition}

By a straight line segment \emph{compatible with  a diagram} or, if the diagram is understood, a \emph{compatible straight line segment} we mean a straight line segment composed from consecutive edges $[x_0,x_1]$, $[x_1,x_2]$, \dots, $[x_{k-1},x_k]$ of the same slope.  In particular, every edge is a compatible straight line segment  in the diagram. When we pass from $D_j$ to $D_{j+1}$ in \eqref{dkmfFtHv}, then every edge of $D_j$ either remains an edge of $D_{j+1}$, or it is divided into several new edges by new vertices. 
A 4-cell is formed from two top edges and two bottom edges.  Observe that, by \eqref{ifdistno3covered}, the halo square $H_j$ will not remain distributive in $D_{j+1}$. Hence, the top edges of $H_j$ do not belong to the trajectory through a top edge of $H_k$ for $k>j$. However,  Corollary~\ref{edgesOfdistCellNSlope} applies to $H_j$  when we consider it  in $D_j$. To summarize the present paragraph,  we conclude the following statement; its  part \eqref{cdsrMsdGsd} follows from part \eqref{cdsrMsdGsc}.

\begin{observation}\label{cdsrMsdGs}
If $D\in\bdia$ is a slim rectangular diagram and $u\in\Traj D$, then the following hold.
\begin{enumeratei}
\item\label{cdsrMsdGsa} If $0\leq j<k\leq t$, then every compatible straight line segment of $D_j$ is also a compatible straight line segment of $D_k$ and, in particular, of $D$.
\item\label{cdsrMsdGsb} The sides of the planar polygon $\oterr u$ are  compatible straight line segments of $D$. In particular, the upper border and the lower border of  $\oterr u$ consist of compatible straight line segments of  of $D$.
\item\label{cdsrMsdGsc} With reference to \eqref{dkmfFtHv}, let $j<t$ and $j < k\leq t$. The upper edges of the halo square $H_j$ are of normal slopes, and they are also edges of $D_k$ and, in particular, of $D$. Furthermore, denoting $1_{H_j}$ by $1_j$, 
the edges of the form $[x,1_j]$ are the same in $D_k$ and, in particular, in $D$ as in $D_{j+1}$. That is, $\set{x\in D_{j+1}: x\prec 1_j}=\set{x\in D_k: x\prec 1_j}=\set{x\in D: x\prec 1_j}$.
\item\label{cdsrMsdGsd} For $\birthno(u)\leq j\leq t$, $\anc uj$ exists and  $\inh(\anc uj)=\inh(u)$. In particular, $\inh(\birthtr(u))=\inh(u)$ in $D$.
\end{enumeratei}
\end{observation}

As a straightforward consequence of Corollary~\ref{dkzTbR}, we have 

\begin{remark} \label{sZtGwRrgT}
If $u\in \Traj D$ for a slim rectangular $D\in\bdia$, then the lower border $B$ of $u$ and the original lower border of $u$ are the  same (straight or broken) lines in the plane and they consists of compatible straight line segments. Furthermore, for all $j\in\set{\birthno(u),\dots, t}$, the lower border of $\anc uj$  is also $B$. 
\end{remark}

\begin{proof} A trivial induction based on Lemma~\ref{mfrkLma}.
\end{proof}

As an illustration for the following lemma, see Figure~\ref{fig-oterr}.

\begin{lemma}\label{trajorigtershape}
Let $D\in\bdia$ be a slim rectangular diagram, and let $u$ be a trajectory of $D$. If $u$ is a straight trajectory, then its original territory, denoted by $\oterr u$, is a rectangle whose sides are compatible straight line segments with normal slopes.
If $u$ is a hat-trajectory, then the polygon $\oterr u$ is bordered by one or two precipitous edges belonging to its upper border and  containing  $1_{\inh(u)}$ as an endpoint,  and compatible straight  line segments of normal slopes. 
\end{lemma}

\begin{proof}  Clearly, all edges of $D_0$ in  \eqref{dkmfFtHv}, are of normal slope.  Hence, the first part of the lemma follows, because $\birthno(u)=0$, provided $u$ is a straight trajectory. Next, assume that $u$ is a hat-trajectory, that is, $\birthno(u)>0$.
Since the halo square $H_{\birthno(u)-1}$ of $u$ is a distributive $4$-cell of  $D_{\birthno(u)-1}$,  \eqref{ifdistno3covered} implies that no element of the ideal $\ideal 1_{\inh(u)}$ can have more than 2 lower covers in $D_{\birthno(u)-1}$. Hence, the rest of the lemma follows from Observation~\ref{dksRptG}
\end{proof}

As a useful supplement to Observation~\ref{dksRptG}, we formulate the following.

\begin{observation}\label{dkhGrmnT}
If $D\in\bdia$ is a slim rectangular lattice diagram, $x,y\in D$, and $x\prec y$, then 
the following three conditions are equivalent.
\begin{enumeratei}
\item\label{dkhGrmnTa} The edge $[x,y]$ is precipitous.
\item\label{dkhGrmnTb} $y$ has at least three lower covers and $x$ is neither the leftmost, nor the rigthmost of them.
\item\label{dkhGrmnTc} The trajectory $u$ containing  $[x,y]$ is a hat-trajectory and $[x,y]$ is $\inh(u)$.
\end{enumeratei}
\end{observation}

\begin{proof} A trivial induction based on Lemma~\ref{mfrkLma}.
\end{proof}

\section{A pictorial version of the  Trajectory Coloring Theorem}
\label{congruenceSection}
The set of prime intervals of a finite lattice $M$ is denoted by $\Prin M$. (An interval $[x,y]$ is \emph{prime} if $x\prec y$.) For a quasiordering (reflexive and transitive relation)  $\gamma$, $x\leq_\gamma y$ stands for $\pair xy\in\gamma$. 
\begin{definition}[{\init{G.\ }Cz\'edli~\cite[page 317]{czgrepres}}] 
A \emph{quasi-colored lattice} is a finite lattice $M$  with a  surjective \mapping{} $\gamma$, called \emph{quasi-coloring}, from $\Prin M$ onto a 
quasiordered set $(A;\nu)$ such that $\gamma$ satisfies the following two properties:
\begin{enumerate}[xxxxx]
\item[\colgammacon] if $\gamma(\inp)\geq_\nu \gamma(\inq)$, then  $\con(\inp)\geq\con(\inq)$,
\item[\colcongamma] if $\con(\inp)\geq\con(\inq)$, then $\gamma(\inp)\geq_\nu\gamma(\inq)$.
\end{enumerate}
If, in addition, $(A;\nu)$ is an ordered  set,  then $\gamma$ is called a \emph{coloring}; this concept is due to \init{G.\ }Gr\"atzer and \init{E.\ }Knapp~\cite{gratzerknapp1}. 
\end{definition}

For $u\in\Traj D$, the top edge  $\inh(u)$ was defined in Definition~\ref{dkGhRwQ}\eqref{dkGhRwQa}.

\begin{definition}[{\init{G.\ }Cz\'edli~\cite[Definitions 4.3 and 7.1]{czgtrajcolor}}]
\label{tRajtaudef}
Let $D$ be a slim rectangular diagram. 
\begin{enumeratei} 
\item\label{tRajtaudefa} On the set $\Traj D$ of all trajectories of $D$, we define a relation $\sigma$ as follows. For  $u,v\in \Traj D$, we let 
$\pair uv\in\sigma$ if{f} $u$ is a hat-trajectory, $1_{\inh(u)}\leq 1_{\inh(v)}$, but $0_{\inh(u)}\not\leq 0_{\inh(v)}$. 
\item\label{tRajtaudefb} For $u,v\in\Traj D$, we let $\pair uv\in\Theta$ if{f} $u=v$, or both $u$ and $v$ are hat trajectories such that $1_{\inh(u)}= 1_{\inh(v)}$. The quotient set $\Traj D/\Theta$ of $\Traj D$ by the equivalence $\Theta$ is denoted $\Straj D$. Its elements are denoted by $\tblokk u$, where $u\in\Traj D$.
\item\label{tRajtaudefc} On the set $\Straj D$, we define a relation $\widehat\sigma$ as follows. For  $\tblokk u$ and $\tblokk v$ in $\Straj D$, we let 
$\pair{\tblokk u }{ \tblokk v}\in\widehat\sigma$ if{f} $\tblokk u\neq\tblokk v$ and   there exist $u',v'\in\Traj D$ such that $\pair u{u'},\pair v{v'}\in\Theta$ and  $\pair{u'}{v'}\in\sigma$.  
\item\label{tRajtaudefd} We let $\wtau=\preogen(\widehat\sigma)$, the reflexive transitive closure of $\widehat\sigma$ on   $\Straj D$. 
\item\label{tRajtaudefe} The \emph{trajectory coloring} of $D$ is the coloring $\widehat\xi$ from $\Prin D$ onto the \poset{} 
 $\tuple{\Straj D;\wtau}$, defined by the rule 
that  $\widehat\xi(\inp)$ is the $\Theta$-block of the unique trajectory containing $\inp$.
\end{enumeratei}
\end{definition}

We recall the following result, which 
carries a lot of information on the congruence lattice of a slim rectangular lattice. (By \cite[Remark 8.5]{czgtrajcolor}, the case of slim semimodular lattices reduces to the slim  rectangular case.) 
Note that the original version of the proposition below assumes slightly less, $D\in \adia$.

\begin{proposition}[{\init{G.\ }Cz\'edli~\cite[Theorem 7.3(i)]{czgtrajcolor}}]
\label{tRqCLthm}
If $L$ is a slim rectangular lattice with a diagram  $D\in\bdia$,  then 
 $\tuple{\Straj D; \wtau}$  is an \poset{} and it is isomorphic to $\tuple{\Jir{\Con L};\leq}$. Furthermore, $\widehat\xi$ in Definition~\ref{tRajtaudef}\eqref{tRajtaudefe} is a coloring.
\end{proposition}

The fact that the key relation $\wtau$ is defined as a transitive (and reflexive) closure is probably inevitable. However, the complicated definition of $\widehat\sigma$, whose reflexive transitive closure is taken, makes Proposition~\ref{tRqCLthm} a bit difficult to use. Hence, we introduce the following concept. For  $u,v\in\Traj D$, we say that $u$ is a \emph{descendant} of $v$, in notation $u\lessdesc v$,   if $\birthno(u)>\birthno(v)$ and the halo square of $u$, as a geometric quadrangle,   is within the original territory  $\oterr v$ of $v$.  Note that ``descendant'' is an irreflexive relation. 
Note also that, as opposed to ``in'' for containment, in geometric sense we always use the preposition ``within''. That is, ``$A$ is within $B$'' means that $A$ and $B$ are geometric polygons (closed subsets of the complex plane that contain their inner points) such that $A$ is a subset of $B$. For a point $x$, if $x\in B$, then we also say that $x$ is within $B$ to express that $B$ is a polygon.

We are now in the position to formulate the main achievement of the present section. Since it looks quite technical in itself, let us emphasize that the following theorem is to be used together with Proposition~\ref{tRqCLthm}, where $\wtau$ is the transitive reflexive closure of $\widehat\sigma$, described pictorially in the theorem below.

\begin{theorem}\label{terThM} For a slim rectangular diagram $D\in\bdia$ and let $u,v\in\Traj D$,  $\pair{\tblokk u}{\tblokk v}\in\widehat\sigma$ iff there are $u'\in\tblokk u$ and $v'\in \tblokk v$ such that $u'\lessdesc v'$.
\end{theorem}

\begin{proof} First of all, note that for any $w\in\Traj D$ and $w'\in\tblokk w$, we have  that $\birthno(w')=\birthno(w)$. This allows us to define $\birthno(\tblokk w)$ as $\birthno (w)$. 

To prove the ``if'' part, assume that $u'\lessdesc v'$. Since $\pair{\tblokk {u}}{\tblokk {v}}=\pair{\tblokk {u'}}{\tblokk {v'}}$, 
what we have to show is that $\pair{\tblokk {u'}}{\tblokk {v'}}\in\widehat\sigma$. Actually, to ease the notation, we that assume  $u\lessdesc v$, and we want to show that
\begin{equation}
\pair{\tblokk {u}}{\tblokk {v}}\in\widehat\sigma\text.
\label{dkMbT}
\end{equation}
We know from $u\lessdesc v$ that the halo square  of $u$ is within $\oterr v$. Clearly, $u$ is a hat-trajectory. Since $0_{\inh(u)}$ is within the interior of this square, it is geometrically (strictly) above the original lower border of $v$. By Remark~\ref{sZtGwRrgT}, $0_{\inh(u)}$ is geometrically above the  lower border of $v$. Hence, Corollaries~\ref{nslplQcor} and \ref{dkzTbR} imply that $0_{\inh(u)}\nleq 0_{\inh(v)}$. On the other hand, the position of the halo square of $u$  yields that $1_{\inh(u)}$ is within the original territory of $v$. Hence, Corollary~\ref{nslplQcor}  and 
Lemma~\ref{trajorigtershape} imply that 
$1_{\inh(u)}\leq 1_{\inh(v)}$. Thus, we 
conclude that $\pair uv\in\sigma$, which 
implies \eqref{dkMbT} and the ``if'' part of Theorem~\ref{terThM}.

To prove the ``only if'' part, assume that $\pair{\tblokk u}{\tblokk v}\in\widehat\sigma$.  Hence, there are $u'\in\tblokk u$ and $v^\ast \in\tblokk v$ such that $\pair{u'}{v^\ast }\in\sigma$. This means that $u'$ is a hat-trajectory, $0_{\inh(u')} \nleq 0_{\inh(v^\ast )}$,  and  $1_{\inh(u')} \leq 1_{\inh(v^\ast )} = 1_{\inh(v)}$.   Our purpose is to find a $v'\in\tblokk v$ such that $u'\lessdesc v'$.   We claim that $u'$ is ``younger'' than $v^\ast $, that is,
\begin{equation}
i:=\birthno(u)= \birthno(u')>\birthno(v^\ast ) =\birthno(v)=:j\text.
\label{dkzGhdvnP}
\end{equation}
The equalities are trivial by the definitions of $i$ and $j$.  
To show the inequality in \eqref{dkzGhdvnP}, there are two cases to consider. First, assume that  $1_{\inh(u')} = 1_{\inh(v^\ast )}$. Since $\pair{u'}{v^\ast }\notin \Theta$ by the  definition of $\widehat\sigma$ and $u'$ is a hat-trajectory, we obtain that $v^\ast $ is a straight trajectory. Thus, we conclude that $i=\birthno(u')>0=\birthno(v^\ast )=j$. Second, assume  that $1_{\inh(u')} < 1_{\inh(v^\ast )}$. 
Clearly, $i=\birthno(u')\neq\birthno(v^\ast )=j$. Suppose, for a contradiction, that $i<j$. Then $v$ is a hat-trajectory and $1_{\inh(u' )} <    1_{\inh(v^\ast )} =   1_{\inh(v )}$.  
Since $u'$ is a hat-trajectory, $1_{\inh(u')}$ has at least three lower covers in $D_i$  , and the same is true in $D_{j-1}$ by  Observation~\ref{cdsrMsdGs}\eqref{cdsrMsdGsc}. But this contradicts 
\eqref{ifdistno3covered}, because $1_{\inh(v )}$ is the top of the halo square $H_{j-1}$, which is distributive in $D_{j-1}$. We have proved \eqref{dkzGhdvnP}. 
%

\begin{figure}[htb]
\includegraphics[scale=1.0]{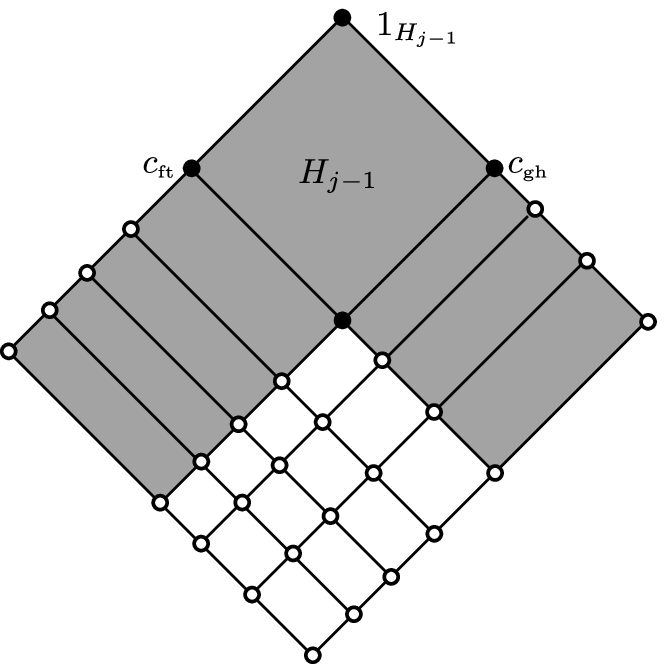}
\caption{$\ideal{1_{H_{j-1}}}$ in $D_{j-1}$ and the territory $S$}\label{fig-idealinDj1}
\end{figure}

\begin{figure}[htb]
\includegraphics[scale=1.0]{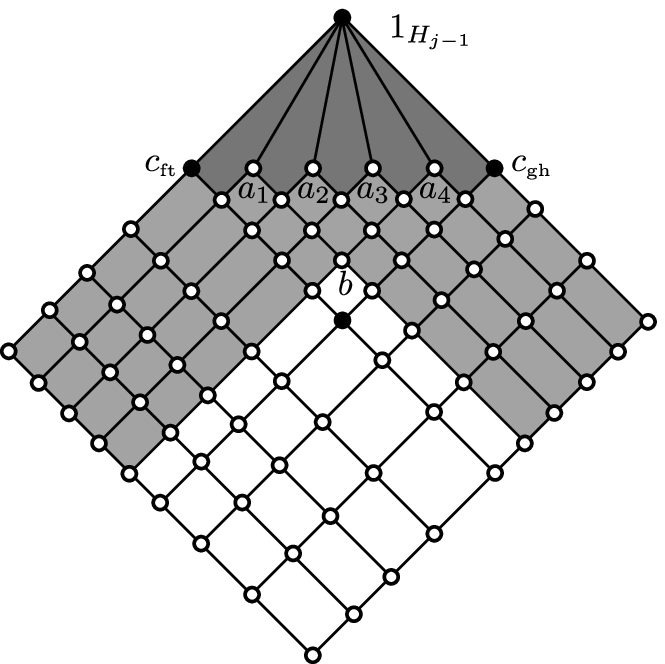}
\caption{$\ideal{1_{H_{j-1}}}$ in $D_{j}$ and the territory $S'$}\label{fig-samidealinDj}
\end{figure}

Next, there are two cases to consider depending on $j>0$ or $j=0$.

First, we assume that $j>0$. Then $v$ and $v^\ast $ are hat-trajectories. 
Note that $1_{\inh(v)}$ belongs to $D_{j-1}$ and equals $1_{H_{j-1}}$. The left and right corners of $H_{j-1}$ are denoted by $\clft$ and $\crght$, respectively. Since the halo square $H_{j-1}$ is distributive in $D_{j-1}$, the ideal $\ideal{1_{H_{j-1}}}$ of $D_{j-1}$ is a grid by Lemma~\ref{dkhgTbFmk}\eqref{dkhgTbFmka}. This ideal is illustrated in Figure~\ref{fig-idealinDj1}. Corollary~\ref{edgesOfdistCellNSlope} yields that 
the edges of this ideal are of normal slopes.
%
We denote by $S$  the planar territory that consists of  the 4-cells (in $D_{j-1}$) of the trajectory through $[\crght,1_{H_{j-1}}]$ that are before (to the left of)  $[\crght,1_{H_{j-1}}]$  and also of the 4-cells  of the trajectory through $[\clft,1_{H_{j-1}}]$  that are after (to the right of)  $[\clft,1_{H_{j-1}}]$. Note that  $S$ is usually concave and that  $H_{j-1}$ is within $S$.
In Figure~\ref{fig-idealinDj1}, $S$ is the grey-colored polygon. 
Since the edges of the grid $\ideal {1_{H_{j-1}}}$ are of normal slopes,  $S$  in $D_{j-1}$ is bordered by compatible straight lines of  normal slopes. 
Hence, by Observation~\ref{cdsrMsdGs}\eqref{cdsrMsdGsa}, $S$ in $D_j$, and also in $D_{i-1}$, is bordered by compatible straight line segments of normal slopes. Listed from left to right, 
let $a_1,\dots, a_k$ be the new lower covers of $1_{\inh(v)}=1_{H_{j-1}}$ in $D_j$; for $k=4$, 
see  Figure~\ref{fig-samidealinDj}. 
By Observation~\ref{cdsrMsdGs}\eqref{cdsrMsdGsc},  
$\clft,a_1,\dots, a_k,\crght$ is the full list, again from left to right, of all lower covers of $1_{\inh(v)}$ in $D$. With the notation $b=a_1\wedge a_k=a_1\wedge\dots\wedge a_k$, the ideal $\ideal b$ of $D_j$ determines a territory $I$. Since every element of $\ideal {1_{H_{j-1}}}$ in $D_{j-1}$ has at most two lower covers by \eqref{ifdistno3covered}, it follows from the multifork construction that every element of $\ideal b$ in $D_j$ has at most two lower covers. 
Therefore, Observation~\ref{dksRptG} or \ref{dkhGrmnT} and Observation~\ref{cdsrMsdGs}\eqref{cdsrMsdGsa} yield that the territory $I$  is bordered by edges of normal slopes in $D_j$ and by compatible straight line segments of normal slopes in $D_{i-1}$.  Consequently,  the territory $S'=S\setminus \interr I$ is again bordered by edges of $D_j$ and by compatible straight line segments of $D_{i-1}$ with normal slopes. In Figure~\ref{fig-samidealinDj}, $S'$ is the grey (dark and light grey together) territory. 
As earlier,  a 4-cell is \emph{$4$-cell with normal slopes} if all of its edges are of normal slopes. In $D_j$, $S'$ is a union of 4-cells.   Namely, it is the union of $4$-cells that belong to the new trajectories that the latest (the $j$-th) fork extension yielded. Among these 4-cells, those containing $1_{\inh(v)}$ are not with normal slopes. They  will be called the \emph{dark-grey cells}, and they are depicted in Figure~\ref{fig-samidealinDj}  accordingly.  
We know from  \init{G.\,}Cz\'edli and \init{E.\,T.\ }Schmidt~\cite[Lemma 13]{czgschvisual}, see also
\init{G.\ }Cz\'edli and \init{G.\ }Gr\"atzer~\cite[Ex.\ 3.41]{czgggltsta}, that two neighboring lower covers of an element in a slim semimodular diagram always generate a cover-preserving square, that is, a 4-cell. 
Hence,  we obtain from Observation~\ref{cdsrMsdGs}~\eqref{cdsrMsdGsc} that
\begin{equation}
\text{the dark-grey 4-cells are also 4-cells in }D\text{ and in }D_{i}\text.
\label{drgrClls}
\end{equation}
It follows from \eqref{tHnDstRbtV}, the multifork construction, and  Corollary~\ref{edgesOfdistCellNSlope} that the rest of the 4-cells of $D_j$ within $S'$ are of normal slopes by  Corollary~\ref{edgesOfdistCellNSlope}; they are called \emph{light-grey 4-cells}, and so they are depicted in  Figure~\ref{fig-samidealinDj}. Although the light-grey 4-cells are not necessarily 4-cells of $D_{i-1}$, we know from Observation~\ref{cdsrMsdGs}\eqref{cdsrMsdGsa} that they are bordered by compatible straight line segments of $D_{i-1}$. Finally, $S$ includes some additional $4$-cells that are not in $S'$; they are uncolored in the figure. Clearly, the \emph{compatible}  straight line segments of $D_{i-1}$ 
\begin{equation}
\text{cannot cut the 4-cell }H_{i-1}\text{ into two halves of positive area.}
\label{cntcthlVs}
\end{equation} 
Furthermore, since $H_{i-1}$ has interior elements in $D_i$, \eqref{drgrClls} gives that $H_{i-1}$ cannot be within a dark-grey 4-cell. Hence, we conclude from \eqref{cntcthlVs}  that  there is a unique light-grey  or uncolored 4-cell $C$ of $D_{j}$ within $S$ such that $H_{i-1}$ is within the territory determined by $C$. (Possibly but not necessarily, $H_{i-1}=C$.) By definitions, $\inh(v^\ast)$ and $\inh(v)$ are in the set $\{[a_1,1_{\inh(v)}]$, \dots, $[a_k,1_{\inh(v)}]\}$ of edges. If $0_{\inh(u')}\leq b=a_1\wedge\dots\wedge a_k$, then $0_{\inh(u')}\leq a_m$ for all $m\in\set{1,\dots,k}$, which contradicts $0_{\inh(u')}\nleq 0_{\inh(v^\ast )} = 0_{\inh(v)}$. Thus, $0_{\inh(u')}\nleq b$.  Combining this with $0_{\inh(u')}\leq 1_C$, it follows that $C$ cannot be an uncolored 4-cell. Hence, $C$ is a light-grey colored $4$-cell in $D_j$. Therefore,  there 
is a (unique) $m\in\set{1,\dots,k}$ such that   
$C$ is a $4$-cell of the  hat-trajectory $w_m$ of $D_j$  with top edge $[a_m, 1_{\inh(v)}]$.
Since $a_m\prec_D 1_{\inh(v)}$ by Observation~\ref{cdsrMsdGs}\eqref{cdsrMsdGsc},  we can also consider the trajectory $v'\in\Traj D$ that contains $[a_m, 1_{\inh(v)}]$; actually,  $[a_m, 1_{\inh(v)}]$ is the top edge of $v'$ by  Observation~\ref{dkhGrmnT}.  Clearly, $w_m=\birthtr(v')$,  $C$ is within $\oterr{v'}=\terr {w_m}$, and $v'\in\tblokk v$.  But $H_{i-1}$ is within $C$, so $H_{i-1}$ is 
also within the original territory $\oterr{v'}$ of $v'$. Consequently,  $u'\lessdesc v'$. 

Second, we assume that $j=0$. Then $v$ is a straight trajectory, $\tblokk v$ is a singleton, and  $u'\lessdesc v'$ follows in a similar but much easier way; the details are omitted. This  completes the proof of Theorem~\ref{terThM}
\end{proof}

\section{\init{G.\ }Gr\"atzer's Swing Lemma}
\label{swingsection}
For a  slim rectangular lattice diagram  $D\in\adia$ and prime intervals $\inp$ and $\inq$ of $D$, we say that $\inp$ \emph{swings} to $\inq$, in notation, $\inp \swings \inq$, if $1_\inp=1_\inq$, $1_\inp$ has at least three lower covers, and $0_\inq$ is neither the leftmost, nor the rightmost lower cover of $1_\inp$.  
If $D$ is in  $\bdia$, not only in $\adia$, then 
Observation~\ref{dkhGrmnT} implies that 
\begin{equation}
\inp\swings\inq\quad\text{ iff }\quad 1_\inp=1_\inq\text{ and }\inq\text{ is a precipitous edge.}
\label{dhgprCpt}
\end{equation}
As usual, $\inp$ is \emph{up-perspective} to $\inq$, in notation, $\inp \perspup \inq$,  
if $1_\inp\vee 0_\inq=1_\inq$ and  $1_\inp\wedge 0_\inq=0_\inp$. Down-perspectivity is just the converse relation defined by $\inp\perspdn\inq \iff 
 \inq\perspup\inp$.  Although here we only  formulate the Swing Lemma for slim \emph{rectangular} lattices, the original version in \init{G.\ }Gr\"atzer~\cite{swinglemma} is the same statement for slim \emph{semimodular} lattices.  Speaking of diagrams rather than lattices is not an essential change.

\begin{lemma}[{Swing Lemma in \init{G.\ }Gr\"atzer~\cite{swinglemma}}]
\label{swnglmm}
Let $D\in\adia$ be a slim rectangular diagram, and let $\inp$ and $\inq$ be edges $($that is, prime intervals$)$ of $D$.
Then the following two conditions are equivalent.
\begin{enumeratei}
\item\label{swnglmma}
$\con(\inp) \geq \con(\inq)$ in the lattice of all congruences of $D$.
\item\label{swnglmmb} There exist an $n\in\mathbb N_0$ and edges $\inr=\inr_0,\inr_1,\dots,\inr_n=\inq$ in $D$ such that
$\inp\perspup\inr$, 
$ \inr_{i-1}\perspdn\inr_{i}$ for $i\in\set{1,\dots,n}$, $i$ odd, and 
$ \inr_{i-1}\swings\inr_{i}$ for $i\in\set{1,\dots,n}$, $i$ even.
\end{enumeratei}
\end{lemma}

\begin{proof}   Our argument relies, among other ingredients, on the multifork construction sequence of $D$, and the notation in \eqref{dkmfFtHv} will be in effect. 

Let $L$ be the lattice determined by $D$. 
By Theorem~\ref{pcvWsWdG}\eqref{pcvWsWdGb}, $L$ also has a diagram $D'$ in $\bdia$. Obviously, $L$ is glued sum indecomposable. Thus, Proposition~\ref{kTrTzmyY} implies that every planar diagram of $L$ is similar to $D$ or $\refl D$. Hence, $D'$ is similar to $D$ or to $\refl D$. Since the statement of the lemma is obviously invariant under left-right similarity, we can assume that $D=D'\in\bdia$.

Assume  \eqref{swnglmmb}. We claim that for prime intervals $\inr'$ and $\inr''$ of $D$, 
\begin{equation}
\text{if }\inr'\swings \inr''\text{, then }\con(\inr') \supseteq \con(\inr'')\text.
\label{sWnGthpZ}
\end{equation} 
Assume that $\inr'\swings \inr''$.  It follows from the definition of $\swings$ and Observation~\ref{dkhGrmnT}\eqref{dkhGrmnTc} that $0_{\inr''}$ is a source element in $D_{\birthno({\inr''})} \setminus D_{\birthno({\inr''})-1}$, and either the same holds for 
$0_{\inr'}$, or $0_{\inr'}$ is a corner of the halo square $H_{\birthno({\inr''})-1}$. In both cases, since $1_{\inr'}=1_{\inr''}$,  $\con(\inr') \supseteq \con(\inr'')$ follows in a straightforward way. This proves \eqref{sWnGthpZ}. On the other hand,  if $ \inr'\perspdn\inr''$, then $\con(\inr')=\con(\inr'')$. Combining this with \eqref{sWnGthpZ}, we obtain \eqref{swnglmma}. 
 Thus, \eqref{swnglmmb} implies \eqref{swnglmma}.

Before proving the converse implication, some preparations are necessary.
For edges $\intv e$ and $\intv e'$ of $D$, we say that $\intv e'$ is \emph{\accessible{}} from $\intv e$ if there exists a finite sequence $\inr_0=\intv e, \inr_1,\dots,\inr_n=\intv e'$ of edges such that, for every $i\in\set{1,\dots,n}$, either
$ \inr_{i-1}\swings\inr_{i}$, or $ \inr_{i-1}\perspdn\inr_{i}$.  
For $j\in\set{0,\dots,t}$, $D_j$ determines a sublattice in $D$. We claim that for common edges  $\intv e$ and $\intv e'$ of $D_j$ and $D$, 
\begin{equation}
\parbox{6cm}
{if $\intv e'$  is \accessible{} from $\intv e$ in $D_j$, then  it is also \accessible{} from $\intv e$ in $D$.}
\label{ddkszhGhRvBn}
\end{equation}
To prove this, we have to show that $\inr_n=\intv e'$, $\inr_{n-1}$, \dots are also edges in $D$. If $\inr_i$ is an edge, that is, a prime interval of $D$, $i>0$, and  $ \inr_{i-1}\perspdn\inr_{i}$, then $\inr_{i-1}$ is also a prime interval in $D$ by semimodularity.  If $\inr_i$ is  a prime interval of $D$, $i>0$, and 
$ \inr_{i-1}\swings\inr_{i}$, then $\inr_{i-1}$ is also a prime interval in $D$ by  Observation~\ref{cdsrMsdGs}\eqref{cdsrMsdGsc}. This completes the induction proving \eqref{ddkszhGhRvBn}.

For a trajectory $w\in\Traj D$, the original territory  $\oterr w$ of $w$ can be divided into two parts; note that one of these parts is empty iff $w$ is a straight trajectory. The union of the $4$-cells (as quadrangles in the plane)  of $\birthtr(w)$ before $\inh(\birthtr(w))$ (if we walk from left to right along $\birthtr(w)$) is the ``\emph{before the top edge}'' part, and this polygon is denoted by $B(w)$. Similarly, the union of the $4$-cells of $\birthtr(w)$ after $\inh(\birthtr(w))$  the ``\emph{after the top edge}'' part, and it is denoted by $A(w)$. Note that  
\begin{equation}
\oterr w=B(w)\cup A(w)\text.
\label{sKsbwbw}
\end{equation}
If $w$ is a hat-trajectory, then both $B(w)$ and $A(w)$ are polygons of positive area and 
\begin{equation}
\parbox{7.8 cm}{
each of $B(w)$ and $A(w)$ has one or two precipitous sides, which contain (that is, end at) $1_{\inh(w)}$, and the rest of the sides are of normal slopes;
}
\label{sidesofBwAw}
\end{equation}
this follows from  Observation~\ref{dkhGrmnT}, the construction of the multifork construction sequence \eqref{dkmfFtHv}, and Corollary~\ref{edgesOfdistCellNSlope}. If $w$ is a straight trajectory, then Corollary~\ref{edgesOfdistCellNSlope} yields that   one of $B(w)$ and $A(w)$ is a rectangle whose sides are of normal slope while the other one is the edge $\inh(\birthtr(w))$, that is,  a degenerate rectangle. 
No matter if $w$ is a hat-trajectory or a straight one, an edge $\ing $ of $D$ is said to be  \emph{quasi-parallel} to $\inh(w)$, in notation, $\ing  \qparallel  \inh(w)$, if  
$\ing $  is within (that is, both  $0_{\ing }$ and $1_{\ing }$ are within) the original territory  $\oterr w$ of $w$, and either $\ing $ is in $B(w)$ and it is of slope $3\pi/4$ (that is, $135^\circ$), or $\ing $ is in $A(w)$ and it is of slope $\pi/4$.  
If $\ing  \qparallel  \inh(w)$, then $\ing $ is or normal slope by definition. Observe that $\qparallel$ is not a symmetric relation.
Let us emphasize that, by definition, $\ing  \qparallel  \inh(w)$ implies that $\ing $ is within $\oterr w$. Note  that if  $\ing$ is within $\oterr w$, then it is within $B(w)$ or within $A(w)$ by Observation~\ref{cdsrMsdGs}\eqref{cdsrMsdGsa}--\eqref{cdsrMsdGsb}, but it is not necessarily quasi-parallel to $\inh(w)$.
We say that an edge $\intv f$ is, say, on the lower border of $\oterr w$ if both $0_{\intv f}$ and $1_{\intv f}$ are on this lower border. We conclude from Lemma~\ref{trajorigtershape} that
\begin{equation}
\parbox{7.5cm}{if $\ing\qparallel\inh(w)$, then $\ing$ is neither on the lower border, nor on the upper border of $\oterr w$.}
\label{notOnBoRdr}
\end{equation}
We claim that, for every edge $\ing $ of $D$ and every   $w\in\Traj D$,
\begin{equation}
\text{if }\ing  \qparallel  \inh(w)\text{, then }\ing \text{ is \accessible{} from }\inh(w)\text.
\label{dwTyBrTh}
\end{equation}
We prove this by induction on $\birthno(\ing )$. Assume that $\ing  \qparallel  \inh(w)$. 
We can also assume that  $\ing\neq\inh(w)$, since otherwise \eqref{dwTyBrTh}  trivially holds. 
By definitions, $\ing\in\oterr w=B(w)\cup A(w)$. By left-right symmetry, we assume that $\ing\in B(w)$. Since $\ing$ is within $B(w)$ and $\ing  \neq  \inh(w)$,  $B(w)$ is of positive area.

It follows from \eqref{notOnBoRdr}, the description of the multifork extension, and that of the multifork construction sequence \eqref{dkmfFtHv}   that $\birthno(w) \leq \birthno(\ing )$. Remember that $t$ denotes the length of the sequence \eqref{dkmfFtHv}.
Combining Observation~\ref{cdsrMsdGs}\eqref{cdsrMsdGsd} and \eqref{ddkszhGhRvBn}, it follows that we can assume that $\birthno(\ing )=t$. (Less formally speaking with more details, if $\ing$ came to existence earlier but not before $w$, then first we could show \eqref{dwTyBrTh} in $D_{\birthno(\ing)}$ for $\ing$ and the ancestor $\anc w{\birthno(\ing)}$ of $w$ the same way we are going to show \eqref{dwTyBrTh} in $D$, and then we could apply \eqref{ddkszhGhRvBn}.) 

That is, $\ing$ came to existence only in the last step  of the multifork construction sequence, and the induction hypothesis is that   for every edge $\ing'$ of $D$, if $\ing'  \qparallel  \inh(w)$ and $\ing'$ is an edge of $D_{t-1}$, then $\ing'$ is \accessible{} from $\inh(w)$ in $D$.
We can also assume that $\birthno(w)<t$, because otherwise $\ing\qparallel \inh(w)$ gives $\ing\in w$ and  \eqref{dwTyBrTh} follows from $\inh(w)\perspdn \ing$. 
It follows from the description of multifork extensions and \eqref{dkmfFtHv}  that  there is a hat-trajectory $z$ of $D$ that is ``responsible'' for the fact that $\ing$ came to existence. Since there are two  essentially different ways of the above-mentioned responsibility,  we have to distinguish two cases.

\begin{Caseone} 
We assume that $\ing \in z$. Let 
\[U=\set{\ing'\in   z: \ing'\qparallel\inh(w)\text{ and }\ing'\text{ is on the right of }\ing}\text.
\]
Being ``on the right'' above means that when we walk along $z$, then   $\ing'$ comes later than $\ing$ or $\ing'=\ing$. 
Note that  $\ing\in U$. For an illustration, see Figure~\ref{fig-gqparall}, where  $B(w)$ is the (light and dark) grey area and   $U=\set{\ing_0,\dots,\ing_4}$. 
If $w$ is a hat-trajectory, then, in accordance with \eqref{sidesofBwAw}, we denote the vertices of the polygon $B(w)$ by $a$, 
$b=0_{\inh(w)}$,  $c=1_{\inh(w)}$, $d$, and $e$; anticlock-wise, starting from the bottom $a$. Except possibly for the edge $[d,c]$,  which could be of slope $\pi/4$ (and then $d$ is not a vertex of the polygon), the slopes of the sides of $B(w)$  are faithfully depicted in Figure~\ref{fig-gqparall}. In particular, $\inh(w)=[b,c]$ is precipitous, if $w$ is a hat-trajectory. On the other hand, if $w$ is a straight trajectory, then 
$\inh(w)$ is on the upper right boundary of $D$ and $D_0$ (because otherwise $B(w)$ would not be of positive area), 
the edges $[b,c]$  and $[d,c]$ are of slopes $3\pi/4$ and $\pi/4$, respectively, while the slopes of the other sides of the polygon $B(w)$ are faithfully depicted. (Note that $d$ is not a vertex of the polygon in this case.)
We claim that, for every edge $\ing'$ of $D$,
\begin{equation}
\text{if }\ing'\in U,\quad\text{then }\ing'
\text{ is not on }\rightb D\text.
\label{inUthenNotRBC}
\end{equation}
To prove this, assume  that $\ing'\in U$. Since $\birthno(\ing')=t >\birthno(w)$, $\ing'\neq \inh(w)$. We know from $\ing'\qparallel \inh(w)$ that $\ing'$ is of slope $3\pi/4$. Observe that 
that $1_{\ing'}\neq 1_{\inh(w)}$, because otherwise either $\inh(w)$ is precipitous and $0_{\ing'}$ is not within $B(w)$, or the edge  $\inh(w)$ is of slope $3\pi/4$ and $\ing'=\inh(w)$. Being within $B(w)$, $1_{\ing'}$ cannot be strictly greater than $1_{\inh(w)}$. 
Hence, using that $1_{\inh(w)}$ is the only cover of $0_{\inh(w)}$ in $D$, we obtain that $1_{\ing'}\ngeq 0_{\inh(w)}$. We also obtain that $1_{\ing'}\nleq 0_{\inh(w)}$, because otherwise Corollary~\ref{nslplQcor} yields that $1_{\ing'}$, which is within $B(w)$, is on the lower right border (from $a$ to $b$) of $B(w)$, but then $0_{\ing'}$ cannot be within $B(w)$ since $\ing'$ is of slope $3\pi/4$. So, $1_{\ing'}\parallel 0_{\inh(w)}$. 
Since the lower right border of $B(w)$, from $a$ to $b$, is a compatible straight line by Observation~\ref{cdsrMsdGs}\eqref{cdsrMsdGsa}--\eqref{cdsrMsdGsb}, it is also a chain in $D$. Extend this chain and $c=1_{\inh(w)}$ to a maximal chain $C$ of $D$. Being within $B(w)$, $1_{\ing'}$ is on the left of $C$. Since $1_{\ing'}\parallel  0_{\inh(w)}=b\in C$, $1_{\ing'}$  is strictly on the left of $C$. This proves \eqref{inUthenNotRBC}.

Trajectories go from left to right.  We claim that, for every  every $\ing'\in U$,
\begin{equation}
z\text{ does not terminate at } \ing' \text{ and goes upwards at }\ing'\text.
\label{kpsgoingUp}
\end{equation}
The first part follows from  \eqref{inUthenNotRBC}. Suppose, for a contradiction, that $z$ goes downwards at $\ing'\in U$ or $z$ makes a turn to the lower right at $\ing'\in U$. This means that  $\ing'$ is the upper left edge of a 4-cell. The slope of the upper right edge of this 4-cell is greater than that of $\ing'$, which is $3\pi/4$ since $\ing'\qparallel \inh(w)$ and $\ing'$ is within $B(w)$. Hence, $D$ has an edge with slope greater than $3\pi/4$. This is a contradiction, because every edge is either precipitous or is of normal slope by Observation~\ref{dksRptG} . Thus, we conclude \eqref{kpsgoingUp}.

\begin{figure}[htb]
\includegraphics[scale=1.0]{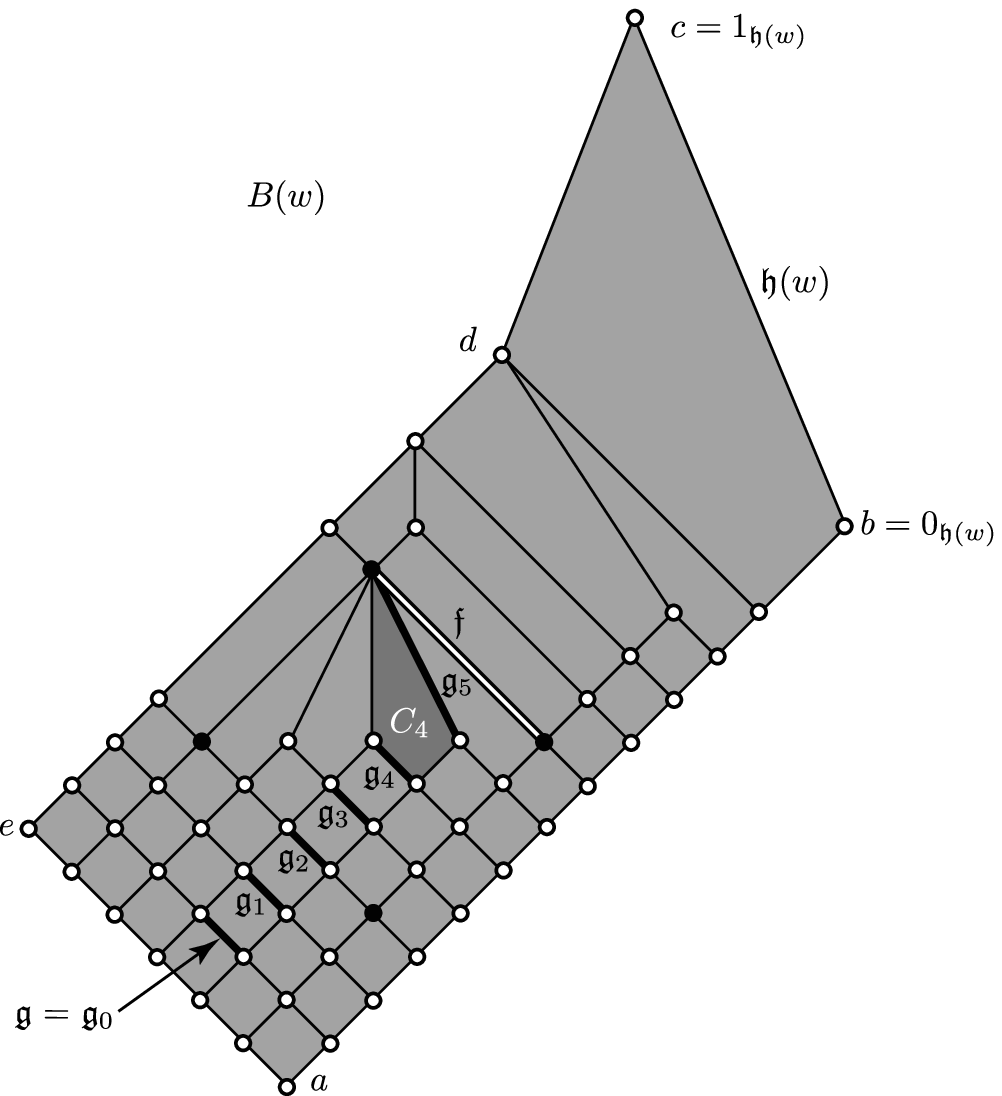}
\caption{$B(w)$  and $\ing\in z$}\label{fig-gqparall}
\end{figure}

Listing from left to right, let $\ing=\ing_0, \dots, \ing_k$ be the edges of $U$, and let $\ing_{k+1}$ be the next  edge of $z$. 
In Figure~\ref{fig-gqparall}, $k=4$ and $\ing_0,\dots,\ing_{k+1}$ are the thick edges. 
\eqref{kpsgoingUp} yields that  $\ing_{k+1}$ exists. Since $\ing_k$ belongs to $U$, it is  of slope $3\pi/4$. Hence,  except possibly for the side from $a$ to $e$,  $\ing_k$ does not lie  on the sides of the polygon  $B(w)$. Therefore, since $\ing_k$ is within $B(w)$, 
$B(w)\cap\terr{C_k}$ is of positive area. However, the sides of $B(w)$, which are compatible straight line segments, cannot divide $\terr{C_k}$, formed by edges of $D$, into two parts of positive area. 
Hence,  $\terr{C_k}$ is within $ B(w)$, that is,   $\terr{C_k} \subseteq B(w)$.  In particular, $\ing_{k+1}$ is within $B(w)$. So, the definition of $U$ implies that $\ing_{k+1} \qnparallel \inh(w)$. This and  Observation~\ref{dksRptG} implies that either $\ing_{k+1}$ is precipitous, or it is of slope $\pi/4$. Applying Corollary~\ref{dkzTbR} to $z$, we obtain that the slope of the edge $[0_{\ing_{k}}, 0_{\ing_{k+1}}]$, 
which is distinct from that of $\ing_k$, is  $\pi/4$. It follows that $\ing_{k+1}$ cannot be of slope $\pi/4$, because otherwise the slope of the edge $[1_{\ing_k}, 1_{\ing_{k+1}}]$ is less than $\pi/4$,  contradicting Observation~\ref{dksRptG}. Therefore, $\ing_{k+1}$ is precipitous, and Observation~\ref{dkhGrmnT} implies that $\ing_{k+1}=\inh(z)$. 

Consider the halo square $H_{t-1}$ in $D_{t-1}$. Its four elements in $D=D_t$ are the black-filled elements in the figure. Since the upper edges of $H_{t-1}$  are of normal slopes and $1_{C_k}=1_{\ing_{k+1}}=1_{\inh(z)}= 1_{H_{t-1}}$, Observation~\ref{dksRptG} implies that $\terr{H_{t-1}}\cap\terr{C_k}$ is of positive area. But $\terr{C_k}\subseteq  B(w)$, so a part of  $\terr{H_{t-1}}$ with positive are  is also in $B(w)$. This is also true in $D_{t-1}$. In $D_{t-1}$, where $H_{t-1}$ is a 4-cell, the sides of $B(w)$, which are compatible straight line segments, cannot divide $\terr{H_{t-1}}$ into two parts of positive area. Hence,  
$\terr{H_{t-1}}\subseteq B(w)$. In particular, both upper edges of $H_{t-1}$ are within $B(w)$. 
The halo square $H_{t-1}$ is distributive in $D_{t-1}$. Hence, Corollary~\ref{edgesOfdistCellNSlope} gives that   its upper edges are of normal slopes. Hence, exactly one of these upper edges, which we  denote by $\intv f$,  is quasi-parallel to $\inh(w)$. In the figure, $\intv f$ is  drawn with double lines. Since  $\birthno(\intv f)\leq t-1$, $\intv f$ is \accessible{} from $\inh(w)$by the induction hypothesis. On the other hand, $\intv f\swings \ing_{k+1}\perspdn \ing$. Thus, transitivity yields that $\ing$ is \accessible{} from $\inh(w)$, as required.
\end{Caseone}

\begin{figure}[htb]
\includegraphics[scale=1.0]{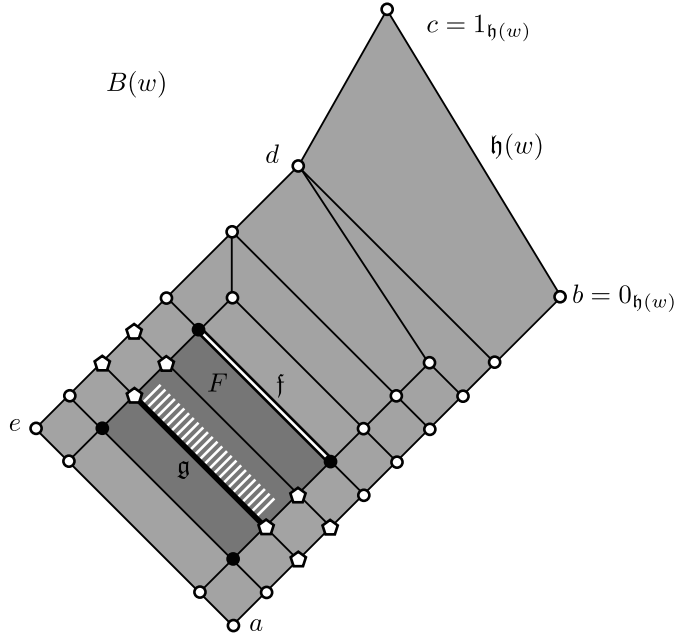}
\caption{$B(w)$ and $\ing\not\in z$}\label{fig-bwGnotinz}
\end{figure}

\begin{Casetwo} We assume that $\ing\notin z$. It follows from the description of a multifork extension that there is a 4-cell $F$ of $D_{t-1}$ that is divided into new cells in $D=D_t$, and $\ing$ is one of the new edges that divide $F$ into parts; see Figure~\ref{fig-bwGnotinz}, where $B(w)$ is the grey area as before, and $\terr F$ in $D$ is dark grey. The slopes of the sides of $B(w)$  in Figure~\ref{fig-bwGnotinz} are depicted with the same accuracy as in case of Figure~\ref{fig-gqparall}.  Corollary~\ref{edgesOfdistCellNSlope}, applied to $D_{t-1}$ and the halo square $H_{t-1}$ whose top is $1_{\inh (z)}$, implies that $F$ is of normal slope. Since $\ing\qparallel \inh(w)$, $\ing$ is of slope $3\pi/4$. Using that $\ing$ is within $B(w)$, both $\ing$ and  $\inh(w)$ are edges of $D_{t-1}$, and $\ing\neq\inh(w)$, we conclude that there is a narrow rectangular zone $S\subseteq B(w)$ of positive area and of normal slopes such that $S$ is on the right of and adjacent to  $\ing$. In the figure, $S$ is indicated by the striped area. Also, choosing it narrow enough, $S$ is within $\terr F$. Since $S\subseteq \terr F\cap B(w)$ holds not only in $D$ but also in $D_{t-1}$, where $F$ is a 4-cell, we conclude that $\terr S\subseteq B(w)$ as in Case 1. Hence, one of the upper edges of $F$, which we denote by $\intv f$, 
 is quasi-parallel to $\inh(w)$. Using $\birthno(f)< t-1$ and the induction hypothesis, we obtain that $\intv f$ is \accessible{} from $\inh(w)$. So is $\ing$, since $\intv f\perspdn \ing$.  This completes the induction, and the proof of \eqref{dwTyBrTh}
\end{Casetwo}

%
%
%

Now, we are in the position  to prove the converse implication of Lemma~\ref{swnglmm}.  Assume that \eqref{swnglmma} holds, that is, $\con(\inp)\geq\con(\inq)$. 
Denote by $u$ and $v$ the trajectories of $D$ that contain $\inp$ and $\inq$, respectively. 
We claim that 
\begin{equation}
\inh(v)\text{ is \accessible{} from }\inh(u)\text.
\label{dkjgmBtZyQ}
\end{equation}
Since  $\widehat\xi$ from 
Definition~\ref{tRajtaudef}\eqref{tRajtaudefe} is a coloring by Proposition~\ref{tRqCLthm}, \colcongamma{} yields that 
\[
\pair{\tblokk v}{\tblokk u}= 
\pair{\widehat\xi(\inq)}{\widehat\xi(\inp)}\in\wtau =\preogen(\widehat\sigma)\text.
\]
Thus, there exist an $n\in\mathbb N_0=\set{0,1,2,\dots}$ and a sequence $w_0=v$, $w_1$, \dots, $w_n=u$ of trajectories of $D$ such that $\pair{\tblokk {w_{j-1}}}{\tblokk {w_j}} \in \widehat\sigma$ for $j\in\set{1,\dots,n}$. By Theorem~\ref{terThM}, there are $w_0',w_1', w_1'', w_2', w_2'',   \dots, w_{n-1}', w_{n-1}'',    w_n''\in\Traj D$ such that $w_j', w''_j\in\tblokk{w_j}$ for $j\in\set{1,\dots,n-1}$,  $w_0'\in \tblokk{w_0}$,    $w_n''\in \tblokk{w_n}$,  and  $w_{j-1}' \lessdesc  w_j''$ for $j\in\set{1,\dots,n}$. Let $F_{j-1}$ denote the halo square of $w'_{j-1}$ when this trajectory is born. 
By the definition of $\,\lessdesc$, $F_{j-1}$   is within $\oterr{w''_j}$.  So are its upper edges, which are edges of $D$ with normal slopes by Observation~\ref{cdsrMsdGs}\eqref{cdsrMsdGsc}. Hence one of these two upper edges, which we denote by  $\intv f_{j-1}$, is quasi-parallel to   $\inh(w''_j)$.
Applying \eqref{dwTyBrTh}, we obtain that $\intv f_{j-1}$ is \accessible{} from $\inh(w''_j)$.  Since 
$\intv f_{j-1}\swings \inh(w'_{j-1})$, transitivity yields that 
\begin{equation}
\text{for }j\in\set{1,\dots,n},\text{ }\inh(w'_{j-1})\text{ is \accessible{} from }\inh(w''_j)\text.
\label{sldqkZhgtw}
\end{equation}
The top edges of any two trajectories in the same $\Theta$-block are \accessible{} from each other; either because they are equal, or by using a  $\swings$ step. Thus, 
\begin{equation}
\parbox{10.5cm}
{$\inh(w''_n)$ and $\inh(w_0)$ are \accessible{} from $\inh(w_n)$ and $\inh(w'_0)$, respectively, and, for $j\in\set{1,\dots, n-1}$, $\inh(w''_{j})$ is \accessible{} from $\inh(w'_j)$.}
\label{slhkZhgtw}
\end{equation}
Using transitivity, \eqref{sldqkZhgtw}, and 
\eqref{slhkZhgtw}, we conclude \eqref{dkjgmBtZyQ}. 
Finally, let $\inr=\inh(u)$. Since $\inp\in u$, we have that  $\inp\perspup \inr$. Similarly, $\inh(v)\perspdn\inq$. Combining these facts with \eqref{dkjgmBtZyQ}, we obtain that $\inq$ is \accessible{} from $\inr$. 
Hence, there exists a finite sequence
$\inr=\inr_0,\inr_1,\dots,\inr_k=\inq$ of edges such that, for each $j\in\set{1,\dots,k}$,  $\inr_{j-1}\perspdn \inr_j$ or  $\inr_{j-1}\swings \inr_j$. However, we still have to show that the relations $\perspdn$ and $\swings$ alternate and that $\perspdn$ is applied first, that is, $\inr_0\perspdn\inr_1$.  

To do so, first we prefix $\inr_0\perspdn\inr_0$ to the sequence, if necessary. Next,  we get rid of the unnecessary  repetitions. Namely, whenever we see the pattern $\intv b\swings \intv b'\swings \intv b''$ in the sequence, we correct it either to $\intv b\swings \intv b''$ or to $\intv b$, depending on $\intv b\neq  \intv b''$ or $\intv b =  \intv b''$. Knowing that  $\perspdn$ is transitive, we correct
every pattern $\intv b\perspdn \intv b'\perspdn \intv b''$  to $\intv b\perspdn \intv b''$ . Finally, there is no pattern to correct, and 
part \eqref{swnglmmb} of Lemma~\ref{swnglmm} holds.  This completes the proof of Lemma~\ref{swnglmm}.
\end{proof}

\end{document}